\documentclass[11pt]{article}
\usepackage[margin=1.00in]{geometry}
\usepackage{amsmath,amssymb,amsthm,mathtools}
\usepackage{enumitem}
\usepackage{graphicx}
\newcommand{\safeincludegraphics}[2][]{%
	\IfFileExists{#2}{\includegraphics[#1]{#2}}{%
		\fbox{\parbox{0.75\linewidth}{\centering Missing figure: \texttt{\detokenize{#2}}}}%
	}%
}
\usepackage{tikz}
\usetikzlibrary{arrows.meta}
\usepackage[colorlinks=true,linkcolor=blue,citecolor=blue,urlcolor=magenta]{hyperref}
\usepackage{microtype}
\usepackage{mathrsfs}
\usepackage[font=small]{caption}
\theoremstyle{plain}
\newtheorem{theorem}{Theorem}[section]
\newtheorem{proposition}[theorem]{Proposition}
\newtheorem{conjecture}[theorem]{Conjecture}
\newtheorem{lemma}[theorem]{Lemma}
\newtheorem{corollary}[theorem]{Corollary}

\newtheorem{definition}[theorem]{Definition}

\newcommand{\grayurl}[1]{%
  \href{#1}{\textcolor{gray}{\nolinkurl{#1}}}%
}

\theoremstyle{definition}
\newtheorem{remark}[theorem]{Remark}

\theoremstyle{plain}
\newtheorem{example}[theorem]{Example}

\newcommand{\R}{\mathbb R}
\newcommand{\E}{\mathbb E}
\newcommand{\Pp}{\mathbb P}
\renewcommand{\P}{\mathbb P}
\newcommand{\Ssym}{\mathbb S}
\newcommand{\operatornorm}[1]{\operatorname{#1}}
\newcommand{\Tr}{\operatorname{Tr}}
\newcommand{\diag}{\operatorname{diag}}
\newcommand{\off}{\operatorname{off}}
\newcommand{\op}{\operatorname{op}}
\newcommand{\one}{\mathbf 1}

\newcommand{\Var}{\operatorname{Var}}
\newcommand{\psione}{\psi_1}
\newcommand{\psitwo}{\psi_2}

\newcommand{\Vol}{\operatorname{Vol}}
\newcommand{\supp}{\operatorname{supp}}

\newcommand{\tr}{\operatorname{tr}}

\newcommand{\Ran}{\operatorname{Range}}
\newcommand{\Cov}{\operatorname{Cov}}

\newcommand{\eps}{\varepsilon}
\newcommand{\pto}{\stackrel{p}{\to}}
\newcommand{\fone}{\mathrm{(F1)}}
\newcommand{\ftwo}{\mathrm{(F2)}}

\providecommand{\HS}{\mathrm{HS}}

\title{Universality and sharp thresholds for ellipsoid fitting}
\author{Frederic Koehler \thanks{Department of Statistics and Data Science Institute, University of Chicago. \url{fkoehler@uchicago.edu}}\and Youngtak Sohn \thanks{Division of Applied Mathematics, Brown University, \url{youngtak_sohn@brown.edu}}}
\date{\today}

\begin{document}
	\maketitle
	
	\begin{abstract}
    We establish a sharp phase transition for fitting random vectors by an ellipsoid. The random vectors have independent subgaussian coordinates with mean zero, variance one, and a common fourth moment, and the number of vectors is proportional to the square of the dimension. We identify an explicit satisfiability threshold such that, with high probability, a positive definite ellipsoid passes through every data point below the threshold, whereas no positive semidefinite fit exists above it. We also determine the optimal squared fitting error throughout the unsatisfiable regime. In particular, the threshold depends on the coordinate distributions only through their common fourth moment, revealing a fourth moment universality phenomenon. For standard Gaussian data the threshold is $1/4$, resolving the ellipsoid fitting conjecture.
	\end{abstract}
	
	\tableofcontents

	\section{Introduction}
\label{sec:model-threshold-main}
Given points $x_1,\ldots,x_n\in\R^d$, the ellipsoid fitting problem asks
whether there exists a positive semidefinite matrix $R\succeq0$ such that
$$
	x_i^\top Rx_i=1,
	\quad\textnormal{for all}\quad i=1,\ldots,n.
$$
Any matrix satisfying these conditions is called an \emph{ellipsoid fit}.

This work studies this problem when $x_1,\ldots,x_n$ are \textit{random}. For independent standard Gaussian vectors, this problem was introduced by~\cite{saunderson2011subspace,saunderson-chandrasekaran-parrilo-willsky2012diagonal, saunderson-parrilo-willsky2013diagonal} in connection
with diagonal plus low rank decompositions. Numerical experiments in
\cite{saunderson2011subspace,
saunderson-parrilo-willsky2013diagonal}
suggested a sharp transition near $d^2/4$ leading to the Gaussian ellipsoid fitting conjecture.

\begin{conjecture}[Gaussian ellipsoid fitting conjecture]
	\label{conj:gaussian-ellipsoid-fitting}
	Let $x_1,\ldots,x_n\overset{\mathrm{i.i.d}}{\sim}N(0,I_d)$, where
	$n=n(d)$. For every fixed $\varepsilon>0$, as $d\to\infty$:
	\begin{enumerate}[label=\textup{(\roman*)}]
		\item If $n\le(1-\varepsilon)\frac{d^2}{4},$	then an ellipsoid fit exists with probability tending to one.
		
		\item If $n\ge(1+\varepsilon)\frac{d^2}{4}$, then no ellipsoid fit exists with probability tending to one.
	\end{enumerate}
\end{conjecture}
A real symmetric $d\times d$ matrix has $d(d+1)/2$ free parameters. If
positive semidefiniteness were ignored, dimension counting would therefore suggest a
feasibility threshold near $d^2/2$. Conjecture~\ref{conj:gaussian-ellipsoid-fitting}
predicts that the constraint $R\succeq0$ cuts this threshold in half, so
that this random semidefinite program (SDP) undergoes a sharp feasibility
transition near $d^2/4$.

Ellipsoid fitting is thus a simple yet nontrivial example of a random semidefinite program exhibiting a sharp phase transition. Such transitions are widely studied in random convex optimization and semidefinite relaxations; see, e.g.~\cite{almt2014living,barak2019nearly}. It is also closely connected to other
problems in theoretical computer science and machine learning, including
overcomplete independent component analysis
\cite{podosinnikova2019overcomplete} and Sum-of-Squares lower bounds for the
Sherrington--Kirkpatrick model
\cite{ghosh-jeronimo-jones-potechin-rajendran2020sos-sk}. We refer to
\cite{potechin-turner-venkat-wein2023nearoptimal,
maillard-kunisky2023replica}
for further discussion of these and other connections.

Previously, the works~\cite{hsieh-kothari-potechin-xu2023ellipsoid,tulsiani-wu2023ellipsoid,bandeira-maillard-mendelson-paquette2024quadratic} established that an ellipsoid fit exists for $n\le d^2/C$ for large enough constant $C>0$. Earlier works had established feasibility first for $n=o(d^{3/2})$~\cite{ghosh-jeronimo-jones-potechin-rajendran2020sos-sk}, and later for $n\le d^2/(\log d)^{O(1)}$~\cite{kane-diakonikolas2023nearly,potechin-turner-venkat-wein2023nearoptimal}. 

Substantial progress toward the sharp transition at $d^2/4$ was made by Bandeira and Maillard~\cite{bandeira-maillard2023ellipsoid}, who established the sharp threshold in an approximate sense: below the threshold, there exist well-conditioned approximate fits with a vanishing error, whereas above the threshold no well-conditioned fit can achieve this. Complementing this rigorous result, Maillard and Kunisky~\cite{maillard-kunisky2023replica} predicted the sharp transition and the
typical shape of fitting ellipsoids using the replica method from statistical physics~\cite{mezard2009information}. We refer to Section~\ref{subsec:further:literature} for further related literature.

\subsection{Main results}
In this work, we resolve
Conjecture~\ref{conj:gaussian-ellipsoid-fitting} and extend it to a broad
class of non-Gaussian distributions. The existence of an ellipsoid fit can
depend strongly on the distribution of the points. For example, if the
$x_i$ are uniform on $\{\pm 1\}^d$, then $S=I_d/d$ fits every sample,
regardless of how large $n$ is. Thus one cannot expect a
distribution independent threshold. It is natural to ask whether a sharp
transition persists for broader classes of distributions and, if so, which
characteristics determine its location. We show that, for random vectors with
independent subgaussian coordinates with mean zero and variance one, the threshold is determined solely
by their common fourth moment $\E[x_{ij}^4]\equiv \kappa$.

Recall that for a random variable $Z$, its subgaussian norm is defined by
\[
\|Z\|_{\psitwo}:=\inf\{t>0:\E\exp(Z^2/t^2)\le2\}.
\]
\begin{theorem}
	\label{thm:phase-transition}
	Let $(x_i)_{1\leq i\leq n}\in \R^d$ be i.i.d. random vectors such that $x_1=(x_{1j})_{1\leq j\leq d}$ has indpenedent coordinates satisfying
	\begin{equation}\label{eq:independent:coordinate:distribution:assumption}
		\E[x_{1j}]=0,\qquad \E \left[x_{1j}^2\right]=1,\qquad
		\E \left[x_{1j}^4\right]=\kappa,\qquad \sup_{j\le d}\|x_{1j}\|_{\psitwo}\le K_x,
	\end{equation}
    where $\kappa>1$ and $K_x>0$ do not depend on $n,d$. Let
	$n,d\to\infty$ with $n/d^2\to\alpha\in(0,\infty)$.  There is an explicit
	threshold $\alpha_\star(\kappa)$, defined in Eq.~\eqref{eq:alpha-star-def},
	which depends only on $\kappa$, such that the following holds.
	\begin{enumerate}[label=\textup{(\roman*)}]
		\item If $\alpha<\alpha_\star(\kappa)$, then with probability tending to one, there exists an ellipsoid fitting all $(x_i)_{1\leq i\leq n}$:
		\[
		\exists R\succ 0,\quad\textnormal{such that}\quad x_i^{\top}R x_i=1\quad \textnormal{for all}\quad 1\leq i\leq n.
		\]
		Moreover, such $R\in \R^{d\times d}$ can be chosen to be well-conditioned: $\lambda_{\max}(R)/\lambda_{\min}(R)=O(1)$. 
		\item If $\alpha>\alpha_\star(\kappa)$, then with probability tending to one
		there is no $R\succeq0$ such that $x_i^\top R x_i=1$ for all $1\leq i\leq n$. Moreover, the optimal squared fitting error converges in probability:
		\begin{equation}\label{eq:thm:least:squares}
			\Gamma_X:=\inf_{R\succeq0, \operatorname{Tr}R=d}\,\inf_{r\geq 0}\,
			\frac{1}{n}\sum_{i=1}^n\left(\frac{x_i^\top Rx_i}{\sqrt{d}}-r\right)^2
			\pto e_\star(\alpha,\kappa),
		\end{equation}
        where $e_\star(\alpha,\kappa)$ is defined in Eq.~\eqref{eq:optimal-fitting-error-main}. For $\alpha>\alpha_\star(\kappa)$, $e_\star(\alpha,\kappa)>0$ holds, thus an approximate ellipsoid fitting is not possible with any $R\succeq 0$ with $\Tr R=d$.
	\end{enumerate}
\end{theorem}
Using $\alpha_\star(3)=1/4$, we immediately have
\begin{corollary}
  The Gaussian ellipsoid fitting Conjecture~\ref{conj:gaussian-ellipsoid-fitting} holds.
\end{corollary}

\begin{figure}[t]
	\centering
	\safeincludegraphics[width=0.65\textwidth]{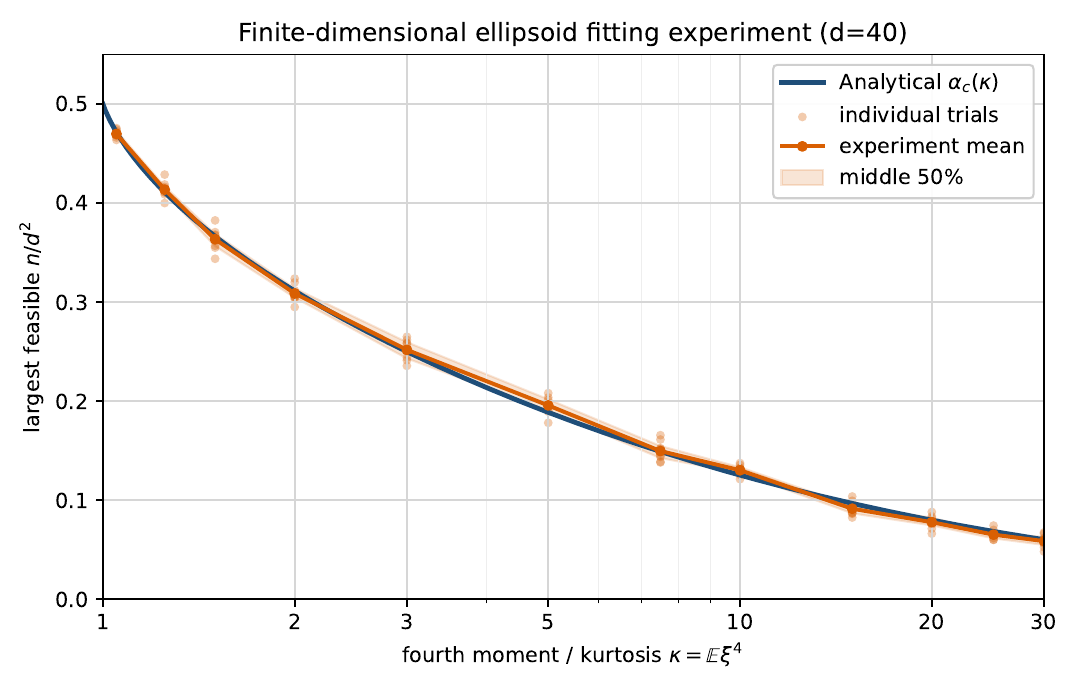}
    \caption[font=small]{The phase diagram and a numerical experiment
		in dimension $d=40$.  The curve is the theoretical threshold
		$\alpha_\star(\kappa)$.  For each value of $\kappa$, eight independent
		trials use i.i.d. coordinates drawn from a centered, variance one
		mixture of two Gaussians with fourth moment $\kappa$; we refer to Appendix~\ref{sec:numerical-experiment}. For each trial, a binary
		search in the number of samples $n$ (y-axis) estimates the largest sample size for which the SDP maximizing
		$t$ subject to $R-tI_d\succeq0$ and
		$x_i^\top R x_i=d$, $i=1,\ldots,n$, is feasible with nonnegative
		optimum. 
        }
	\label{fig:alpha-star-kappa-experiment}
\end{figure}

A couple of remarks concerning Theorem~\ref{thm:phase-transition} are in order.

First, the quantity $\Gamma_X$ in Eq.~\eqref{eq:thm:least:squares} has a simple interpretation as the optimal squared fitting error. The constraint $\Tr R=d$ fixes the average eigenvalue of $R$ to be one and, in particular, excludes the trivial choice $R=0$. For each $R$, the parameter $r\ge0$ selects the common level against which the values $x_i^\top Rx_i/\sqrt d$ are compared. Here, the normalization $1/\sqrt{d}$ reflects the fluctuation scale: for $R=I_d$, $x_i^\top Rx_i=\|x_i\|_2^2$ fluctuates around $d$ on the scale $\sqrt d$.  Thus, $\Gamma_X$ stays of order one as $d\to\infty$. Finally, if $\Gamma_X=0$ is attained at $(R,r)$ with $r>0$, then $S=R/(r\sqrt d)$ satisfies $x_i^\top Sx_i=1$ for all $i$, and hence defines an ellipsoid fit. Conversely, $\Gamma_X>0$ rules out the existence of an ellipsoid fit.

Second, the role of the fourth moment $\kappa$ has a probabilistic interpretation.  Under
the assumptions of Theorem~\ref{thm:phase-transition},
$$
	\Var\left(\|x_1\|_2^2\right)=d(\kappa-1).
$$
Thus larger values of $\kappa$ correspond to greater radial fluctuations around the sphere of radius $\sqrt d$, suggesting that the points should be harder to fit. This intuition is reflected in the fact that $\kappa\mapsto\alpha_\star(\kappa)$ is nonincreasing (see Figure~\ref{fig:alpha-star-kappa-experiment}). Here, if $\kappa=1$, then we must have $x_i\in \{\pm 1\}^d$ for $\kappa=1$, so $S=I_d/d$ fits every sample regardless of $n$. Thus, the nondegenerate regime is $\kappa>1$.

    Finally, Theorem~\ref{thm:phase-transition} also covers discrete coordinates.  For example, fix $p\in (0,1/2)$, and let $x_{ij}=\frac{\xi_{ij}-p}{\sqrt{p(1-p)}}$ where $(\xi_{ij})_{1\leq i\leq n,1\leq j\leq d}\stackrel{i.i.d.}{\sim}\operatorname{Ber}(p)$. Then $x_{ij}$ has mean zero, variance one, and its subgaussian norm is bounded by a constant depending only on $p$. Also, the fourth moment $\E x_{ij}^4=\frac1{p(1-p)}-3$ ranges over $(1,\infty)$ as $p$ ranges over $(0,1/2)$.
    
    For such discrete data, it is not a priori clear that exact fitting should remain possible when $n\asymp d^2$, let alone exhibit a sharp transition. Questions of singularity, invertibility, and anti-concentration for random matrices with discrete entries such as $X=[x_1,\ldots, x_n]\in \R^{d\times n}$ have long been central topics in random matrix theory; see e.g.~\cite{tao2007singularity, rudelson-vershynin2008littlewood, tao2009inverse, rudelson-vershynin2010nonasymptotic}. In fact, our proof of Theorem~\ref{thm:phase-transition}-(i) is motivated by several ideas from this literature. We refer to Section~\ref{sec:proof:overview} for a proof overview.

\begin{figure}[t]
	\centering
	\includegraphics[width=\linewidth]{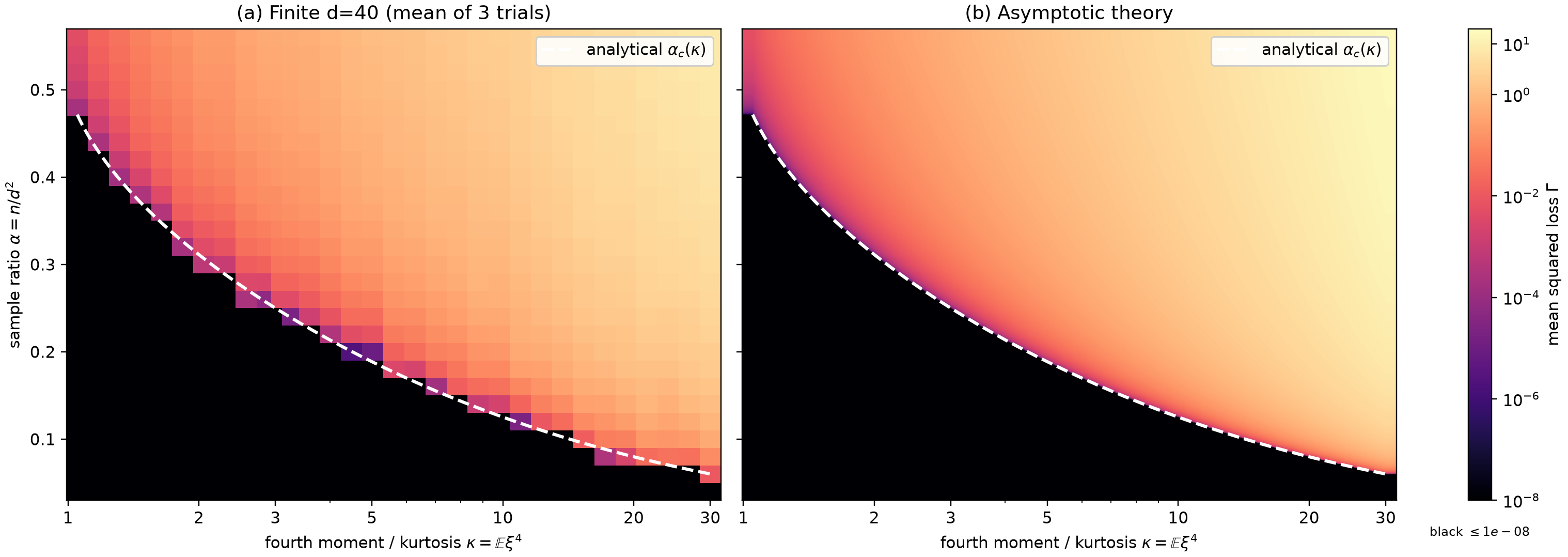}
	\caption[font=small]{Squared loss for the $d=40$ experiment. The setup is the same as
		Figure~\ref{fig:alpha-star-kappa-experiment}, except that for each cell we average the squared loss over three trials.  The left panel is empirical and the
		right panel is the asymptotic formula.  The x-axis is the fourth moment
		$\kappa$, the y-axis is $n/d^2$, and black denotes zero loss.}
	\label{fig:square-loss}
\end{figure}
Figure~\ref{fig:square-loss} shows the same phase diagram as Figure~\ref{fig:alpha-star-kappa-experiment} through the total optimal squared fitting error normalized by $d^2$, whose limit is $\alpha e_\star(\alpha,\kappa)$. Brighter colors indicate larger optimal fitting errors. We refer to Appendix~\ref{sec:numerical-experiment} for  implementation details.

\subsubsection{The Gaussian model and explicit formula}
Beyond numerical evidence, the value $d^2/4$ in the Gaussian ellipsoid
fitting conjecture is naturally connected to a Gaussian matrix model, which was made rigorous for approximate fitting by~\cite{bandeira-maillard2023ellipsoid}.

To formulate it, let $\Ssym^d$ denote the space of real symmetric
$d\times d$ matrices, equipped with the Frobenius inner product
\[
\langle A,B\rangle:=\Tr(AB)\,,\qquad A,B\in \Ssym^d.
\]
Define the linear map $\mathcal{L}_X:\Ssym^d \to \R^n$ as
\begin{equation}\label{eq:definition:L}
\mathcal L_X(R):=\big(\langle W_i,R\rangle\big)_{i\le n},\qquad\textnormal{where}\qquad W_i:=\frac{x_ix_i^\top-I_d}{\sqrt d}.
\end{equation}
Let $\one\in\R^n$ be the all-ones vector and let
$P_{\one^\perp}$ be the orthogonal projection onto $\one^\perp$. Since
this projection subtracts the empirical mean, $\Gamma_X$ in Eq.~\eqref{eq:thm:least:squares} can be written as
$$
	\Gamma_X
	= 	\inf_{R\succeq 0. \Tr R=d}\inf_{b\in \R} \frac1n\sum_{i=1}^{n}\left(\langle W_i, R\rangle -b\right)^2=
	\inf_{R\succeq 0. \Tr R=d}
	\frac1n\|\mathcal A_X R\|_2^2,\quad\textnormal{where}\quad \mathcal A_X:= P_{\one^\perp} \mathcal{L}_X.
$$
In particular, $\Gamma_X$ is an \textit{empirical risk minimization} (ERM) with respect to observations $(W_i)_{i\leq n}$. Note that the dimension of $W_i$ is $d(d+1)/2$, which is the same order as $n$ in the regime $n/d^2\to \alpha$. In this regime, recent literature on universality of ERM (e.g.~\cite{hu2022universality, montanari-saeed2022universality,han2023universality}) motivates
replacing the $(W_i)_{i\leq n}$ with i.i.d. Gaussian matrices $(G_i)_{i\leq n}\in \Ssym^d$ matching first two moments with $W_i$.

Note that $\E W_i=0$, and the covariance of $W_i$ depends
on the coordinate law only through its fourth moment $\kappa$: for every
$A,B\in\Ssym^d$,
\begin{equation}
	\label{eq:C-kappa-trace}
	\E\left[
	\langle W_i,A\rangle\langle W_i,B\rangle
	\right]
	=
	\frac1d\Big(2\langle A, B\rangle +(\kappa-3)\big\langle {\rm diag}(A), {\rm diag}(B)\big\rangle
	\Big)
	=:\mathcal{C}_\kappa(A,B),
\end{equation}
where ${\rm diag}(A):=(A_{11},\ldots, A_{dd})$. Let $G_1,\ldots,G_n$ be independent symmetric Gaussian matrices with
this covariance. Concretely, their upper-triangular entries are
independent and satisfy
\begin{equation}\label{eq:def:gaussian:features}
	(G_1)_{jj}\sim N\left(0,\frac{\kappa-1}{d}\right),
	\qquad
	(G_1)_{j\ell}
	\sim N\left(0,\frac1d\right),
	\quad j<\ell.
\end{equation}
Then, $\E[
\langle G_i,A\rangle\langle G_i,B\rangle]
=\mathcal{C}_\kappa(A,B).$ Define the corresponding Gaussian optimization problem by
$$
	\mathcal L_{G,\kappa}(A):=\bigl(\langle G_i,A\rangle_F\bigr)_{i\le n},
	\qquad
	\mathcal A_{G,\kappa}:=P_{\one^\perp}\mathcal L_{G,\kappa},
	\qquad
	\Gamma_{G,\kappa}
	:=
	\inf_{R\succeq 0, Tr R=d}
	\frac1n\|\mathcal A_{G,\kappa}R\|_2^2.
$$
Observe that when $\kappa=3$, the matrices $G_1,\ldots,G_n$ are
independent GOE matrices~\cite{anderson-guionnet-zeitouni2010random}.
Their law is rotationally invariant in $\Ssym^d$, and hence
$\ker(\mathcal A_{G,3})$ is a uniformly oriented random subspace of
codimension $n-1$. Moreover,
\[
\Gamma_{G,3}=0
\quad\Longleftrightarrow\quad
\ker(\mathcal A_{G,3})\cap\Ssym_+^d\neq\{0\},
\qquad
\Ssym_+^d:=\{A\in\Ssym^d:A\succeq0\}.
\]
This intersection undergoes a transition at the statistical dimension
 of $\Ssym_{+}^d$, which is known to be $d(d+1)/4$
\cite{chandrasekaran2012convex,almt2014living}; see also
\cite[Section~3]{bandeira-maillard2023ellipsoid}. This explains the
threshold $d^2/4$ in Conjecture~\ref{conj:gaussian-ellipsoid-fitting}.

For general $\kappa$, however, the Gaussian matrices $G_i$'s are no longer isotropic in $\Ssym^d$. Nevertheless, we show in Section~\ref{sec:gaussian-anisotropy-threshold} that the asymptotic value of $\Gamma_{G,\kappa}$ can be characterizes using a suitable convex Gaussian min-max theorem (CGMT)~\cite{gordon1988milman, thrampoulidis-oymak-hassibi2015regularized}. 
\begin{proposition}
	\label{thm:gaussian-master}
	Fix $\kappa>1$, and suppose $n,d\to\infty$ with $n/d^2\to\alpha\in(0,\infty)$.
	\begin{enumerate}[label=\textup{(\roman*)}]
		\item If $\alpha<\alpha_\star(\kappa)$, then there are constants
		$0<m<M<\infty$, depending only on $\alpha$ and $\kappa$ such that with probability tending to one there exists
		$R\in\Ssym^d$ satisfying
		$$
			mI_d\preceq R\preceq MI_d,
			\qquad
			\Tr R=d,
			\qquad
			\mathcal A_{G,\kappa} R=0.
		$$

		\item We have $\Gamma_{G,\kappa}\pto e_\star(\alpha,\kappa)$. If $\alpha>\alpha_\star(\kappa)$, then $e_\star(\alpha,\kappa)>0$ and
		$\P(\Gamma_{G,\kappa}\ge c_0)\to1$ for some
		$c_0=c_0(\alpha,\kappa)>0$.
	\end{enumerate}
\end{proposition}
We now define the threshold $\alpha_\star(\kappa)$ and the limiting optimal squared fitting error $e_\star(\alpha,\kappa)$. These formulas arise from the CGMT analysis of $\Gamma_{G,\kappa}$.

Let $\nu$ be the rescaled semicircle law such that ${\rm supp}(\nu)=[-\sqrt2,\sqrt2]$:
\[
d\nu(x)=\frac{1}{\pi}\sqrt{2-x^2}
\mathbf 1_{\{|x|\le\sqrt2\}}dx.
\]
For $\omega <\sqrt{2}$, let
\[
s(\omega):=\int(x-\omega)_+^2\,d\nu(x),\qquad m(\omega):=\int(x-\omega)_+\,d\nu(x).
\]
Here $\omega$ acts as a spectral cutoff, while $m(\omega)$ and
$s(\omega)$ measure the first two moments of the spectrum above this
cutoff. Set
\[
\mathcal E_\kappa(\omega)
:=
s(\omega)+\frac{\kappa-3}{2}m(\omega)^2.
\]
The second term accounts for the anisotropy and
vanishes in the GOE case $\kappa=3$. Then,
\begin{equation}
\label{eq:alpha-star-scalar-main}
\alpha_\star(\kappa)
=
\mathcal E_\kappa(\omega_\kappa)
=
s(\omega_\kappa)
+\frac{\kappa-3}{2}m(\omega_\kappa)^2,
\end{equation}
where $\omega_\kappa\in [-\sqrt{2},\sqrt{2}]$ is the unique solution of
\[
\omega_\kappa=\frac{\kappa-3}{2}m(\omega_\kappa).
\]
For $\kappa=3$, this gives $\omega_\kappa=0$, thus we have $\alpha_\star(3)=s(0)=1/4$.

We also prove in
Section~\ref{subsec:auxiliary} that $\alpha_\star(\kappa)$ admits a variational representation
\begin{equation}
\label{eq:alpha-star-def}
\alpha_\star(\kappa)
=
\sup_{\substack{f\in L^2(\nu),\ f\ge0\\ \int f\,d\nu=1}}
\frac{\left(\int xf(x)\,d\nu(x)\right)^2}
{\int f(x)^2\,d\nu(x)+(\kappa-3)/2}.
\end{equation}
This representation is useful for studying the dependence of the
threshold on $\kappa$; in particular,
$\alpha_\star(\kappa)\to1/2$ as $\kappa\downarrow1$.

For each $\alpha>0$ and $\kappa>1$, define the limiting optimal squared fitting error $e_\star(\alpha,\kappa)$ via
\begin{equation}
\label{eq:optimal-fitting-error-main}
e_\star(\alpha,\kappa)
=
\frac2\alpha\Big(\frac{\kappa-3}{2}m(\omega_{\alpha,\kappa})-\omega_{\alpha,\kappa} \Big)_+^2.
\end{equation}
where $\omega_{\alpha,\kappa}<\sqrt2$ is the unique solution of
\[
\mathcal E_\kappa(\omega_{\alpha,\kappa})=\alpha.
\]
\begin{remark}
	For rotationally invariant data, whereby $x_i$'s are i.i.d. with $Ox_i\stackrel{d}{=}x_i$ for any orthogonal matrix $O$, Maillard and Kunisky
	\cite{maillard-kunisky2023replica} used the replica method from statistical physics to predict a critical threshold depending only on the normalized radial
	variance $\tau=\lim_{d\to\infty}d^{-1}\Var(\|x\|_2^2)$.  Under the
	assumptions of Theorem~\ref{thm:phase-transition}, $\tau=\kappa-1$, and
	their threshold formula is equivalent to
	\eqref{eq:alpha-star-scalar-main}.  The limiting
	optimal squared fitting error in
	\eqref{eq:optimal-fitting-error-main} seems to be new.
\end{remark}

\begin{remark}
	\label{rem:gaussian-energy-quadratic-onset}
	As $\alpha\downarrow\alpha_\star(\kappa)$ from above, we can compute (Lemma~\ref{lem:gaussian-energy-quadratic-onset})
	\begin{equation}
		\label{eq:gaussian-energy-quadratic-onset}
		e_\star(\alpha,\kappa)
		=
		\frac{(\alpha-\alpha_\star(\kappa))^2}
		{2\alpha_\star(\kappa)m(\omega_\kappa)^2}
		+o\bigl((\alpha-\alpha_\star(\kappa))^2\bigr).
	\end{equation}
	Thus, in terms of the $\Gamma_X$ (or $\Gamma_{G,\kappa}$), the phase
transition at $\alpha=\alpha_\star(\kappa)$ is second order (``continuous''), because the first derivative is zero at both sides of the phase transition, whereas the second derivative jumps discontinuously from zero.
\end{remark}

\subsubsection{Universality and exact fitting}

Having analyzed the Gaussian model, we next transfer the asymptotic formula for $\Gamma_{G,\kappa}$ to $\Gamma_X$ through universality. In fact, at the level of approximate fitting, this universality holds for a substantially broader class of distributions and allows dependence among the coordinates of each $x_i$.

We say the random vector $x_1\in \R^d$ satisfies \textit{approximate tensorization of variance} with dimension-free constant $C>0$ if the following holds: for every measurable $f: \R^d\to \R$ with $\Var(f(x_1))<\infty$, 
\begin{equation}
	\label{eq:approximate-variance-tensorization}
	\Var(f(x_1))
	\le C\sum_{j=1}^d
	\E\left[\Var(f(x_1)\mid x_{1,-j})\right],
\end{equation}
where $x_{1,-j}$ denotes all coordinates of $x_1$ except the $j$'th. This condition allows nontrivial dependence among the coordinates. While product measures satisfy \eqref{eq:approximate-variance-tensorization} with $C=1$ by the Efron--Stein inequality~\cite{vanhandel2016probability}, it also holds for classes of weakly dependent measures; more precisely, those where the Glauber dynamics has a dimension-free spectral gap (see, e.g., \cite[Chapter 2]{vanhandel2016probability} or \cite{LP}).

The following theorem extends Theorem~\ref{thm:phase-transition}-\textup{(ii)} beyond independent coordinates.

\begin{theorem}[Universality of the optimal squared fitting error]
	\label{thm:intro-approx-tensorization-squared-loss}
	Let $x_1,\ldots,x_n\in\R^d$ be i.i.d. random vectors,
	and write $x_1=(x_{1j})_{j\le d}$. Assume that $\E[x_1]=0, \Cov(x_1)=I_d$, and that for $W_1:=\frac{x_1x_1^\top-I_d}{\sqrt d}$, we have for some fixed $\kappa>1$,
	$$
		\E\left[
		\langle W_1,A\rangle
		\langle W_1,B\rangle
		\right]
		=\mathcal{C}_\kappa(A,B),
		\qquad A,B\in\Ssym^d.
	$$
    Moreover, assume that $x_1$ satisfies
	\eqref{eq:approximate-variance-tensorization} with a dimension-free
	constant $C>0$, and $\sup_{j\le d}\E|x_{1j}|^8\le M$ for some constant $M<\infty$
	independent of $d$. If $n,d\to\infty$ with $n/d^2\to\alpha\in(0,\infty)$, then
	\begin{equation}
		\label{eq:coordinate-squared-loss-formula}
		\Gamma_X\pto  e_\star(\alpha,\kappa).
	\end{equation}
\end{theorem}

Although Theorem~\ref{thm:intro-approx-tensorization-squared-loss} is motivated by previous universality results~\cite{montanari-saeed2022universality,bandeira-maillard2023ellipsoid}, it is worth pointing out that our proof follows a different approach.

The approximate fitting result of~\cite{bandeira-maillard2023ellipsoid} uses free energy interpolation, building on the techniques of~\cite{montanari-saeed2022universality}, to show that interpolating $(W_i)_{i\leq n}$ to $(G_i)_{i\leq n}$ through $t\in [0,\pi/2]\mapsto \cos(t)W_i+\sin(t)G_i$ does not change the asymptotic free energy. Passing to
a (low temperature) limit yields universality of the associated ERM for $\ell_p$ losses for $1\le p<4/3$, but does
not cover the squared loss $p=2$.

More importantly, such comparisons require the optimizers to be
localized to a deterministic domain of Gaussianity $\mathcal D_d\subset\Ssym^d$ (cf.~\cite[Definition~1]{montanari-saeed2022universality}). In particular, the laws of the one-dimensional projections $\langle W_1,R\rangle$ and $\langle G_1,R\rangle$ must be asymptotically the same uniformly over $R\in\mathcal D_d$: for every bounded
Lipschitz function $\varphi:\R\to\R$,
\[
\sup_{R\in\mathcal D_d}
\left|
\E\varphi\left(\langle W_1,R\rangle\right)
-
\E\varphi\left(\langle G_1,R\rangle\right)
\right|
\longrightarrow0.
\]
However, this condition fails on the full set
$\{R\succeq 0:\Tr R=d\}$, even when $x_1\sim N(0,I_d)$. Indeed, recalling that $W_1=(x_1x_1^\top -I_d)/\sqrt{d}$, it is not hard to see that the above convergence fails when $R$ has a spectral spike, e.g.~$R=\left(1-\frac{c}{\sqrt d}\right)I_d
+c\sqrt d\,e_1e_1^\top$ for fixed $c>0$.

A natural sufficient condition excluding this behavior is the absence of spectral spikes.  When the coordinates $(x_{1j})_{j\leq d}$ are independent and satisfy Eq.~\eqref{eq:independent:coordinate:distribution:assumption}, de Jong's central limit theorem for quadratic forms~\cite{dejong1987quadratic} may be used to show that if sequence of matrices $(R_d)_{d\geq 1}$ satisfy
\[
\frac{\|R_d\|_{\op}}{\|R_d\|_F} \to 0,\quad\textnormal{then}\quad \frac{\langle W_1, R\rangle}{\sqrt{\Var(\langle W_1, R_d\rangle)}}\stackrel{d}{\to}N(0,1)\quad\textnormal{as $d\to\infty$.}
\]
Therefore, any $\mathcal{D}_d\subseteq \{R: \|R_d\|_{\op}\leq \eps_d \|R_d\|_F\}$ with $\eps_d=o(1)$ is a domain of Gaussianity. However, there is no a priori reason for all near-minimizers to lie in such a domain. Although each candidate can be decomposed into spiked and nonspiked parts, this decomposition depends on the candidate matrix and leads to a substantially more involved comparison.

Theorem~\ref{thm:intro-approx-tensorization-squared-loss} shows that
neither a no-spike restriction nor coordinate independence is intrinsic to
universality of $\Gamma_X$. Our proof bypasses these restrictions through a Lindeberg argument that replaces $W_i$ by $G_i$ one at a time. An important observation is that $\Gamma_X$ involves
a squared loss, whereas the free-energy interpolation of
\cite{montanari-saeed2022universality} treats general sufficiently smooth losses. For squared loss, the Lindeberg approach is particularly suitable: after a second-order expansion, the matching first and second moments of
$W_i$ and $G_i$ yield the required cancellation.

To carry out this Lindeberg argument, the main difficulty is that we have a constrained optimization problem (with p.s.d. cone), and the
original optimization problem need not be stable under replacing $W_i$ by $G_i$. We address this by adding a self-concordant barrier (cf. Definition~\ref{def:self-concordance}), which keeps the optimizer strictly in the interior and
provides the derivative bounds needed for second-order Taylor expansion. This role of self-concordant barriers is standard in
optimization, notably in interior-point methods
\cite{nesterov-nemirovskii1994interior} and bandit linear optimization
\cite{abernethy2009competing}. Their use in universality of ERM seems to be new and may be of independent
interest. We refer to Theorem~\ref{thm:positive-barrier-row-universality} for
a general universality theorem for squared loss.

\begin{remark}	\label{lem:zero-inflated-sphere-goe-matching-unsat}
   In view of Theorem~\ref{thm:intro-approx-tensorization-squared-loss}, it is natural to ask how far the exact-fitting conclusion of
Theorem~\ref{thm:phase-transition}\textup{(i)} extends beyond independent coordinates. The following example shows that isotropy, uniform subgaussianity, and covariance matching alone do not suffice. In this
sense, exact fitting is more delicate than approximate fitting; our proof of the latter relies on concentration estimates, whereas the
passage to exact fitting requires anti-concentration. This distinction is well-known in random constraint satisfaction problems (see Section~\ref{subsec:further:literature}): a vanishing fraction of exceptional constraints may be
	negligible for an averaged loss and yet destroy exact satisfiability.

	Let $z\in \R^d$ be uniform on $\{u\in \R^{d}: \|u\|_2=\sqrt{d}\}$, and let $J\in \{0,1\}$ be independent of $z$
	with $\P(J=1)=\frac{d}{d+2}$, and set $x=\sqrt{\frac{d+2}{d}}Jz$.
	Then, $\E[x]=0, \Cov(x)=I_d$, and $x$ is a uniformly subgaussian vector (cf. \cite[Definition 3.4.1]{vershynin2018highdimensional}). Moreover, using rotational invariance of $z$, a direct calculation shows that $x$ has the same fourth moments as $g\sim N(0,I_d)$, thus $W=(xx^\top-I_d)/\sqrt d$ has covariance form
	$\mathcal{C}_3(\cdot,\cdot)$ defined in \eqref{eq:C-kappa-trace}. 
    
    Nevertheless, if $x_1,\ldots,x_n$ are i.i.d. copies of $x$ and
	$n/d^2\to\alpha>0$, then with probability tending to one there is no ellipsoid fit. Indeed, since $\P(x=0)=2/(d+2)$ and $n\asymp d^2$, the sample contains a zero
vector with probability tending to one. Consequently, no ellipsoid can fit
all the sample points. We refer to
Appendix~\ref{subsec:exact-fitting-obstructions} for further examples and discussion of obstructions to exact fitting.
\end{remark}

To prove the exact fitting in Theorem~\ref{thm:phase-transition}-\textup{(i)}, it remains to improve from approximate fitting to exact fitting under the independence of coordinates. Our argument first constructs a well-conditioned matrix $mI_d\preceq \widehat R_d\preceq M I_d$ whose residual
$\widehat\rho:=-\mathcal A_X\widehat R_d\equiv -P_{\one^\perp}\mathcal L_X \widehat R_d$ satisfies
\[
\frac{\|\widehat\rho\|_2}{d}\pto 0,
\qquad
\frac{(\log d)^4}{\sqrt d}\|\widehat\rho\|_\infty\pto 0.
\]
The technical core of the exactification argument is the following
interpolation theorem.

\begin{theorem}
	\label{thm:small-residual-interpolation}
	Let $x_1,\ldots,x_n\in \R^d$ be i.i.d. random vectors with independent coordinates satisfying~\eqref{eq:independent:coordinate:distribution:assumption} for some $K_x>0$.  Suppose that $n/d^2\to\alpha\in(0,1/2)$.  There is a constant $C<\infty$,
	depending only on $\alpha$ and $K_x$, such that with probability tending
	to one, the following holds: for every
	$y\in\Ran(\mathcal A_X)$, there exists $\Delta_y\in\Ssym^d$ satisfying
	$$
		\mathcal A_X\Delta_y=y
	$$
	and
	$$
		\|\Delta_y\|_{\op}
		\le C\left(
		\frac{\|y\|_2}{d}
		+\frac{(\log d)^4}{\sqrt d}\|y\|_\infty
		\right).
	$$
\end{theorem}
Applying the theorem with $y=\widehat\rho$ yields
$\Delta_{\widehat\rho}$ with
$\|\Delta_{\widehat\rho}\|_{\op}\pto 0$. Consequently,
$\widehat R_d+\Delta_{\widehat\rho}$ remains positive definite and satisfies $\mathcal A_X
\bigl(\widehat R_d+\Delta_{\widehat\rho}\bigr)=0.$ After rescaling, $R=\widehat R_d+\Delta_{\widehat\rho}$ gives an exact ellipsoid fit. The proof of Theorem~\ref{thm:small-residual-interpolation} rests on a vector anti-concentration estimate in
Proposition~\ref{prop:vector-smallball}, which may be of independent
interest. We refer to Section~\ref{sec:proof:overview} for an overview of
the argument.

Beyond this application, Theorem~\ref{thm:small-residual-interpolation}
may be viewed as a minimum-norm interpolation result, a topic that has
received considerable attention recently in machine learning; see, e.g.,
\cite{belkin-hsu-ma-mandal2019reconciling,
bartlett-long-lugosi-tsigler2020benign,
hastie-montanari-rosset-tibshirani2022ridgeless,
chinot-loffler-vandegeer2022robustness,
koehler-zhou-sutherland-srebro2021uniform}.
In our setting, the
estimate must hold uniformly over all possible residuals $y$ (cf. \cite{chinot-loffler-vandegeer2022robustness}) and the interpolating correction is matrix-valued. The proof of Theorem~\ref{thm:small-residual-interpolation} rests on a vector anti-concentration estimate in
Proposition~\ref{prop:vector-smallball}, which may be of independent
interest. We refer to Section~\ref{sec:proof:overview} for an overview of
the argument. 
\subsection{Further related literature}
\label{subsec:further:literature}

\paragraph{Universality of ERM.} The basic question in the study of universality of empirical risk minimization (ERM) is whether the limiting value of a random optimization problem built from non-Gaussian observations (a.k.a. features)
$(W_i)_{i\le n}$ remains unchanged when they are replaced by
Gaussian variables $(G_i)_{i\le n}$ with matching low-order moments.
In shallow neural networks, this principle became known as the Gaussian
equivalence conjecture, and was used in \cite{gerace2020generalisation,loureiro2021learning, goldt2022gaussian} to derive asymptotic predictions using Gaussian techniques. Rigorous universality theorems were subsequently developed for random
features models and broader classes of ERM~\cite{hu2022universality,
montanari-saeed2022universality, han2023universality}. Further developments include Gaussian mixture~\cite{dandi2023universality}, multi-layer neural networks~\cite{schroder2024deterministic}, and max-margin classification~\cite{montanari2023universality}. This max-margin problem is particularly close to ours: its separability
threshold parallels the exact fitting transition in
Theorem~\ref{thm:phase-transition}. A key step in its analysis is a restricted strong
convexity estimate (cf. \cite[Lemma 1]{montanari2023universality}) from compressed sensing literature
\cite{candes-tao2007dantzig, donoho2006most}. Indeed, this connection motivated the restricted lower bounds in Theorem~\ref{thm:centered-compatibility} used in our
exact fitting argument.

\paragraph{Random matrix theory.} The proof of the interpolation result, Theorem~\ref{thm:small-residual-interpolation}, draws on methods from
discrete random matrix theory. A central object of study in this area is the least singular value of random matrices with independent, possibly discrete,
entries~\cite{tao2007singularity,rudelson-vershynin2008littlewood,
tao2009inverse,rudelson-vershynin2009rectangular, koltchinskii-mendelson2013smallest}.
Anti-concentration estimates, often formulated as small-ball probability
bounds, play an important role. Such bounds control the probability that a
random sum lies in a ball of small radius (see
\cite{rudelson-vershynin2010nonasymptotic} for a survey).  To prove Theorem~\ref{thm:small-residual-interpolation}, we recast the
statement as a nuclear norm lower bound (cf. Theorem~\ref{thm:centered-compatibility}), and adapt techniques from this
literature, including the decomposition into compressible and
incompressible vectors~\cite{rudelson-vershynin2008littlewood} and vector
anti-concentration estimates~\cite{rudelson-vershynin2015smallball}.

\paragraph{Convex optimization.}
As mentioned above, our universality proof of Theorem~\ref{thm:intro-approx-tensorization-squared-loss} uses regularization by a
self-concordant barrier. The theory of self-concordance was originally
developed for interior-point methods in convex optimization
\cite{nesterov-nemirovskii1994interior,nesterov2004introductory,
renegar2001interior}. The log-determinant barrier used in the proof of Theorem~\ref{thm:intro-approx-tensorization-squared-loss} is the classical example for semidefinite programs, such as the ellipsoid fitting problem. The crucial property of self-concordance is that it controls second-order Taylor expansions, which also lie at the core of the Lindeberg replacement argument.

In interior-point methods, the regularization level is gradually lowered
along the central path. In our argument, it remains fixed while the samples
are replaced one at a time. This sequential structure is reminiscent of
follow-the-regularized-leader (FTRL) methods in online convex optimization
\cite{hazan2016introduction}, and leads naturally to Bregman divergences (see Section~\ref{sec:shared-row-universality-proof}).
A related use of self-concordant barriers appears in the work of Abernethy,
Hazan, and Rakhlin~\cite{abernethy2009competing} on bandit linear
optimization.

\paragraph{Random CSPs} The ellipsoid fitting problem may be viewed as a continuous random
constraint satisfaction problem (CSP): the matrix $R$ is the continuous
variable, and each sample $x_i$ imposes the random constraint
$x_i^\top R x_i=1$. Thus,
Theorem~\ref{thm:phase-transition} identifies the SAT--UNSAT threshold
$\alpha_\star(\kappa)$: with high probability the problem is satisfiable (SAT)
below the threshold and unsatisfiable (UNSAT) above it. Sharp thresholds in
discrete random CSPs, such as random $k$-SAT and NAE-SAT, have long been
studied in statistical physics~\cite{mpz02, kmrsz07} and probability~\cite{friedgut1999sharp, dss16, dss22, NSS22FOCS}. Continuous random CSPs have also been
studied, with the spherical perceptron model as a canonical example
\cite{shcherbina2003rigorous,
franz2017universality,
montanari2024tractability}. While the rigorous establishment of the SAT
side is approached by using the second moment method in the literature, our SAT proof follows a different
route: it uses random-matrix anti-concentration to turn an approximate
ellipsoid fit into an exact one through
Theorem~\ref{thm:small-residual-interpolation}.

\paragraph{Concurrent work.} In the final preparation of the manuscript, we learned of two independent concurrent works: (1) de la Cerda, Potechin, Tulsiani and Xu \cite{delacerda2026sharpphasetransitionellipsoid} and (2) Misiakiewicz and Wen \cite{misiakiewicz2026sharpsatunsatphasetransition}. Both works establish the conjectured sharp threshold $d^2/4$ for Gaussian
ellipsoid fitting. The present
paper also studies universality beyond Gaussian ellipsoid fitting and the optimal squared fitting loss.

\paragraph{Acknowledgements}
We thank Sofia de la Cerda, Aaron Potechin, Madhur Tulsiani, and Jeff Xu for coordinating the preparation of the manuscripts. Y.S. is supported by NSF grant DMS-2601870, and thanks Korean Institute for
Advanced Study, where part of this work was carried out.

\paragraph{Statement of AI use} During the later stages of this project, the authors used ChatGPT 5.5 Pro for assistance with parts of the proofs and with writing
and debugging code for numerical simulations. The authors subsequently
used ChatGPT 5.6 for editorial revision of an earlier draft. All AI-assisted calculations and suggestions were independently checked by the authors, who take full responsibility for the contents of the paper.
\paragraph{Notation}
The space of real symmetric $d\times d$ matrices is $\Ssym^d$, and
$\Ssym_+^d$ is the positive semidefinite cone.  For symmetric matrices,
$A\succeq B$ means $A-B\in\Ssym_+^d$, and $A\succ0$ means $A$ is
positive definite.  We write $\langle A,B\rangle\equiv \langle A,B\rangle_F:=\Tr(AB)$ and denote the Frobenius norm by $ \|A\|_F:=\langle A,A\rangle^{1/2}$. We respectively denote $\|A\|_{\op}$ and $\|A\|_*$ as the operator and nuclear norms. The identity matrix is $I_d$, or simply $I$ when the dimension is clear.
The symbol $\Sigma_\kappa$ denotes the covariance operator on $\Ssym^d$
associated with $\mathcal{C}_\kappa$~\eqref{eq:C-kappa-trace}:
\begin{equation}
	\label{eq:Sigma-kappa-def-main}
	\langle A,\Sigma_\kappa B\rangle_F=\mathcal{C}_\kappa(A,B),\qquad A,B\in\Ssym^d.
\end{equation}
$\one \in \R^n$ denotes the all-ones vector and the Euclidean orthogonal projection onto $\one^\perp$ is denoted by
$P_{\one^\perp}: \R^n\to \R^n$.  We write
$[n]=\{1,\ldots,n\}$. We write $B_2^m$ and $S^{m-1}$ for
the Euclidean unit ball and sphere in $\R^m$, $\Vol(\cdot)$ for Lebesgue volume,
and $\Ran(A)$ for the range of a linear map $A$.  If $v\in\R^n$ and
$\mathcal I\subset[n]$, then $v_{\mathcal I}$ denotes the vector supported on
$\mathcal I$ that agrees with $v$ on $\mathcal I$ and is zero on
$\mathcal I^c$.  For a real random variable $Z$, and $p\geq 1$, it's $L^p$ norm is dentoed by $\|Z\|_{L^p}:=(\E[|Z|^p])^{1/p}$. Finally, $O_\P(a_d)$ means bounded in probability after division by $a_d$.

\section{Proof overview}
\label{sec:proof:overview}
This section derives
Theorems~\ref{thm:phase-transition}
and~\ref{thm:intro-approx-tensorization-squared-loss} from intermediate
results proved in later sections. Our Gaussian input is
Proposition~\ref{thm:gaussian-master}, proved in
Section~\ref{sec:gaussian-anisotropy-threshold} using the convex Gaussian
min--max theorem
\cite{gordon1988milman,thrampoulidis-oymak-hassibi2015regularized}. Following the terminology of random constraint satisfaction problems, we
call an instance \emph{SAT} if it admits an ellipsoid fit and \emph{UNSAT}
otherwise.

\subsection{SAT: approximate fitting and exactification}
\label{subsec:proof:overview:sat}
Recall from \eqref{eq:definition:L} that
$\mathcal A_X:=P_{\one^\perp}\mathcal L_X$. Since
$P_{\one^\perp}$ subtracts the empirical mean,
$\mathcal A_XR=0$ precisely when the values $x_i^\top Rx_i$ are constant in $1\leq i \leq n$. Thus a positive-definite matrix in $\ker(\mathcal A_X)$ yields,
after rescaling, an exact ellipsoid fit.

\subsubsection{The approximate fit}
To obtain an exact fit for $\alpha<\alpha_\star(\kappa)$, we first construct a
well-conditioned approximate fit $\widehat R_d$ such that
$n^{-1}\|\mathcal A_X\widehat R_d\|_2^2=o(1)$, together with an
$\ell_\infty$ bound that rules out large coordinates of $\mathcal{A}_X\widehat{R}_d$. For Gaussian data, Bandeira--Maillard
\cite[Theorem~1.4]{bandeira-maillard2023ellipsoid} proved an analogous approximate-fitting result below $\alpha_\star(3)=1/4$ with vanishing
normalized $\ell_p$ error for $1\le p<4/3$ but without the $\ell_\infty$ control. This $\ell_\infty$ control, which is obtained through a leave-one-out stability argument, is crucial for our exactification step as will be outlined below.

\begin{proposition}
	\label{thm:approx-fit}
	Assume that $x_1,\ldots,x_n\in\R^d$ are i.i.d. random vectors with
	independent coordinates satisfying
	\eqref{eq:independent:coordinate:distribution:assumption}. Suppose that
	$n/d^2\to\alpha$ with $0<\alpha<\alpha_\star(\kappa)$. Then, there are
	constants $0<m<1<M<\infty$, depending on $\alpha,\kappa,K_x$, and
	random matrices $\widehat R_d$ such that $mI_d\preceq \widehat{R}_d\preceq MI_d$ and with probability
	tending to one,
	\[
	\frac1n\|\mathcal A_X\widehat R_d\|_2^2
	\le \left(\frac{\log d}{d}\right)^{1/10},
	\qquad
	\|\mathcal A_X\widehat R_d\|_\infty\le(\log d)^2.
	\]
\end{proposition}

We prove Proposition~\ref{thm:approx-fit} in
Section~\ref{subsec:approx-tensorization-regularity}, after establishing
Theorem~\ref{thm:intro-approx-tensorization-squared-loss}, since the same
universality argument yields the required $\ell_2$ estimate.

\subsubsection{Small-residual interpolation}
\label{sec:sat-duality-proof}
The technical heart of Theorem~\ref{thm:phase-transition} is to prove \textit{exact fit}. As explained in Remark~\ref{lem:zero-inflated-sphere-goe-matching-unsat}, this is the central challenge in pinning down the satisfiability of a random CSP such as ellipsoid fitting. Let $\widehat R_d$ be the approximate fit supplied by
Proposition~\ref{thm:approx-fit}, and define its residual by $\widehat\rho:=-\mathcal{A}_X\widehat{R}_d$. If we can find $\Delta\in \Ssym^d$ such that
\[
\mathcal A_X\Delta=\widehat\rho,
\qquad
\|\Delta\|_{\op}=o(1),
\]
then $\mathcal A_X(\widehat R_d+\Delta)=0$.  Since
$\widehat R_d\succeq mI_d$, the $\widehat{R}_d+\Delta\succ 0$ for all large $d$.

The difficulty is that $\widehat R_d$, and hence $\widehat\rho$, depends
on the same data $(x_i)_{i\leq n}$ that define $\mathcal A_X$. Thus, considering a fixed deterministic residual does not suffice, and in order to prove exact fitting, we need an
estimate holding \textit{uniformly over all} $y\in\Ran(\mathcal A_X)$.
To formulate it, define
\[
\Xi_{\op}(y):=
\inf\bigl\{\|\Delta\|_{\op}:\mathcal A_X\Delta=y\bigr\}.
\]
Our goal is to prove that $\Xi_{\op}(y)=o(1)$ holds uniformly whenever
$y$ has sufficiently small  $\ell_2$ and $\ell_\infty$ norms (that Proposition~\ref{thm:approx-fit} supplies for
$y=\widehat\rho$).  We begin with a dual formulation of
$\Xi_{\op}(\cdot)$.

To this end, regard $\mathcal A_X$ as a map from $\Ssym^d$ into
$\one^\perp$, equip $\one^\perp$ with the Euclidean inner product, and
equip $\Ssym^d$ with the Frobenius inner product. The adjoint of
$\mathcal L_X$ is
\[
\mathcal L_X^*\lambda
=\sum_{i=1}^n\lambda_iW_i
=\frac1{\sqrt d}\sum_{i=1}^n\lambda_i(x_ix_i^\top-I_d),
\qquad \lambda\in\R^n.
\]
For $R\in\Ssym^d$ and $\lambda\in\one^\perp$, self-adjointness of
$P_{\one^\perp}$ gives
\[
\langle\lambda,\mathcal A_XR\rangle
=\langle\lambda,\mathcal L_XR\rangle
=\langle\mathcal L_X^*\lambda,R\rangle_F.
\]
Thus the adjoint of $\mathcal A_X:\Ssym^d\to\one^\perp$ is precisely the
restriction of $\mathcal L_X^*$ to $\one^\perp$. Moreover, for
$\lambda\in\one^\perp$, the identity terms cancel and
$\mathcal L_X^*\lambda=d^{-1/2}\sum_i\lambda_i x_ix_i^\top$.
\begin{lemma}\label{lem:range-duality}
	For every $y\in\Ran(\mathcal{A}_X)$,
	\[
	\Xi_{\op}(y)=\sup_{\lambda\in\one^\perp}
	\left\{\langle\lambda,y\rangle:\|\mathcal{L}_X^*\lambda\|_*\le1\right\}=
	\sup_{\substack{\lambda\in\one^\perp\\ \mathcal{L}_X^*\lambda\ne0}}
	\frac{|\langle\lambda,y\rangle|}{\|\mathcal{L}_X^*\lambda\|_*}.
	\]
\end{lemma}

\begin{proof}
 The primal problem is feasible because $y\in\Ran(\mathcal{A}_X)$.  In epigraph form it
	is
	\[
	\inf_{\Delta,t} t \quad\text{subject to }\mathcal{A}_X\Delta=y,\quad \|\Delta\|_{\op}\le t.
	\]
	If $\mathcal{A}_X\Delta_0=y$, then $(\Delta_0,t_0)$ with
	$t_0>\|\Delta_0\|_{\op}$ is strictly feasible relative to the affine equality
	constraint.  Hence strong finite-dimensional convex duality applies, by
	Slater's condition; see, for instance,
	Boyd--Vandenberghe~\cite[Section~5.2.3]{boyd-vandenberghe2004convex}.
	The Lagrangian, with multiplier $\lambda\in\one^\perp$, is
	\[
	\|\Delta\|_{\op}+\langle\lambda,y-\mathcal{A}_X\Delta\rangle
	=\langle\lambda,y\rangle+\|\Delta\|_{\op}
	-\langle \mathcal{L}_X^*\lambda,\Delta\rangle_F .
	\]
	The infimum over $\Delta$ is finite exactly when
	$\|\mathcal{L}_X^*\lambda\|_*\le1$, since the nuclear norm is dual to the operator norm.
	This gives the stated dual.  If $\mathcal{L}_X^*\lambda=0$, then
	$\lambda\perp\Ran(\mathcal{A}_X)$, so $\langle\lambda,y\rangle=0$.  Such directions do not
	affect the supremum.
\end{proof}
In view of Lemma~\ref{lem:range-duality}, a naive approach is to try to prove that with high probability,
\begin{equation}\label{eq:naive}
\|\mathcal L_X^* \lambda\|_* \gtrsim\max\left\{(\log d)^{2+\eta}\|\lambda\|_1\,,\, d\|\lambda\|_2\right\},\quad\textnormal{for all}\quad\lambda\in\one^\perp,
\end{equation}
for some fixed $\eta>0$, which would imply Theorem~\ref{thm:phase-transition}-$(i)$. Indeed, recall that the residual $\widehat{\rho}=-\mathcal{A}_X \widehat{R}_d$. satisfies $\|\widehat{\rho}\|_2\ll \sqrt{n}\asymp d$ and $\|\widehat{\rho}\|_{\infty}\leq (\log d)^2$ by Proposition~\ref{thm:approx-fit}. Setting $y=\widehat\rho$ in Lemma~\ref{lem:range-duality}, and using the two bounds $|\langle\lambda,\widehat\rho\rangle|\leq \|\lambda\|_2 \|\widehat\rho\|_2$ and $|\langle\lambda,\widehat\rho\rangle|\leq\|\lambda\|_1\|\widehat\rho\|_\infty$, \eqref{eq:naive} would give
\[
\Xi_{\op}(\widehat\rho)\lesssim \frac{\|\widehat \rho\|_{\infty}}{(\log d)^{2+\eta}}+ \frac{\|\widehat\rho\|_2}{d}=o(1).
\]
We could therefore choose $\Delta$ with $\mathcal A_X\Delta=\widehat\rho$ and $\|\Delta\|_{\op}=o(1)$.

However, it turns out that the naive inequality \eqref{eq:naive} is \textit{false} even for fixed $\lambda\in \one^\perp$.  A counterexample can be constructed as follows. Let $(e_i)_{1\leq i\leq n}$ denote the standard orthonormal basis of $\R^n$, and let $1\le m\le(n-2)/2$ and small $\eps>0$, to be chosen below, and set
\[
\lambda_0=h_0+t_0,\quad\textnormal{where}\quad h_0=e_1-e_2,\quad t_0=\varepsilon\left(\sum_{i=3}^{m+2}e_i-\sum_{i=m+3}^{2m+2}e_i\right).
\]
The vectors $h_0$ and $t_0$ may be respectively viewed as the sparse \textit{head} and dense \textit{tail} of the vector $\lambda_0$. Note that
\[
	\|\mathcal{L}_X^* h_0\|_*
	=
	\left\|
	\frac{x_1x_1^\top-x_2x_2^\top}{\sqrt d}
	\right\|_*
	\le
	\frac{\|x_1\|_2^2+\|x_2\|_2^2}{\sqrt d}
	=
	O_\P(\sqrt d).
\]
For the tail,
Proposition~\ref{thm:apriori-estimates} below establishes that with high probability,
$\|\mathcal L_X^*z\|_*\le Cd\|z\|_2$ holds uniformly in $z\in \R^n$.  Hence $\|\mathcal L_X^*t_0\|_*
\lesssim d\|t_0\|_2
\asymp d\sqrt m\,\varepsilon$. Consequently,
\[
	\|\mathcal{L}_X^*\lambda_0\|_*
	=
	O_\P\bigl(\sqrt d+d\sqrt m\,\varepsilon\bigr).
\]
Now take $m=\lfloor n/4\rfloor\asymp d^2$ and
$\varepsilon=d^{-3/2}$. Then, $\|\mathcal{L}_X^*\lambda_0\|_*
	=O_\P(\sqrt d)$ while
\[
	d\|\lambda_0\|_2= d\sqrt{2+2m\eps^2}\asymp d,
	\qquad
	(\log d)^{2+\eta}\|\lambda_0\|_1=(\log d)^{2+\eta}(1+m\eps)
	\asymp \sqrt d\,(\log d)^{2+\eta}.
\]
Thus \eqref{eq:naive} fails for every fixed $\eta>0$.

We note that this obstruction from $\lambda_0$ is reminiscent of the head/tail (or compressible/incompressible) decomposition from sparse recovery~\cite{candes-tao2007dantzig}) or random matrix literature~\cite{rudelson-vershynin2008littlewood}. Motivated by these ideas, we decompose every $\lambda\in \R^n$ into a head $h$ and a tail $t$, and measure the two parts in different norms.

\begin{theorem}\label{thm:centered-compatibility}
	Let $x_1,\ldots,x_n\in\R^d$ be i.i.d. random vectors with independent
	coordinates satisfying~\eqref{eq:independent:coordinate:distribution:assumption} fomr some $K_x>0$. Suppose that
	$n/d^2\to\alpha\in(0,1/2)$. Then there are constants $c,C>0$, depending
	only on $\alpha$ and $K_x$, such that with probability $1-o(1)$,
	every $\lambda\in\R^n$ admits a set $\mathcal I=\mathcal I(\lambda) \subset[n]$ consisting of
	its $|\mathcal I|$ largest coordinates of $\lambda$ in absolute value, with
    \[|\mathcal I|\le Cd\log d
    \]
    such that writing
	$h=\lambda_{\mathcal I}$ and $t=\lambda_{\mathcal I^c}$, we have
	\[
	\|\mathcal{L}_X^*\lambda\|_*\equiv \left\|\frac1{\sqrt d}\sum_{i=1}^n\lambda_i(x_ix_i^\top-I_d)\right\|_* \ge
	c\left(\frac{\sqrt d}{(\log d)^4}\|h\|_1+d\|t\|_2\right).
	\]
\end{theorem}
The proof of Theorem~\ref{thm:centered-compatibility} is one the main technical challenges of the paper and occupies Sections~\ref{sec:novel-sat-interpolation-details}. Here, we make a couple of remarks about the statement.

First, the condition $|\mathcal I|=O(d\log d)$ comes from a standard shelling argument in sparse recovery literature (see e.g.~\cite{candes-tao2007dantzig},~\cite{bickel-ritov-tsybakov2009lasso},~\cite{vandegeer-buhlmann2009conditions}). Indeed, order the coordinates of $\lambda$ by absolute value and move the largest ones into $\mathcal I$ until the remaining tail $t$ is \textit{$K$-light}, defined as $\|t\|_\infty\le (Kd)^{-1/2}\|t\|_2$. As long as the tail is not $K$-light, removing its largest coordinate decreases its squared $\ell_2$-norm by at least a $1/(Kd)$ fraction. After $\Theta(d\log d)$ steps, the tail must therefore either become $K$-light or have negligible $\ell_2$-norm. In the first case, $K$-lightness of $t$ plays a crucial role in the proof of Theorem~\ref{thm:centered-compatibility} as seen in Section~\ref{sec:light-tail-quotient}.

Second, the range $\alpha<1/2$ is optimal for Theorem~\ref{thm:centered-compatibility} in the following sense: if $\alpha>1/2$, then for all large $d$, $n>\dim\Ssym^d=\frac{d(d+1)}2,$ so the linear map $\mathcal{L}_X^*:\R^n\to\Ssym^d$ has a nonzero kernel.
	For any $\lambda\ne0$ in this kernel, $\mathcal{L}_X^* \lambda =0$ while the right-hand side of the asserted inequality is positive. We note in passing that in order to prove the SAT part of the Gaussian ellipsoid fitting conjecture, Conjecture~\ref{conj:gaussian-ellipsoid-fitting}, we only needs the range $\alpha<1/4$, in which case the proof can be
considerably simplified (see
Remark~\ref{rem:gaussian:simpler}).
	
The main challenge in Theorem~\ref{thm:centered-compatibility} is uniformity over all $\lambda\in\mathbb R^n$ throughout the range $\alpha<1/2$. Reaching this entire range is necessary for Theorem~\ref{thm:phase-transition} for every $\kappa>1$, since $\alpha_\star(\kappa)\to1/2$ as $\kappa\downarrow1$. To prove such a uniform estimate, we discretize all possible sparse heads~$h$ and light tails~$t$ by finite nets, and because $n/d^2\to \alpha$, these nets contain $\exp(\Theta(\alpha d^2))$ many points. A union bound therefore requires an exponentially small failure probability on the $d^2$ scale for each fixed pair, with enough decay to overcome the net entropy even as $\alpha\uparrow1/2$.

The following \textit{vector anti-concentration/small-probability estimate} proved in Section~\ref{sec:fixed-vector-smallball}, which might be of independent interest, plays a crucial role in obtaining the required failure probability.
\begin{proposition}
	\label{prop:vector-smallball}
	Let $(\xi_i)_{1\leq i\leq n}$ be independent random variables such that $\E \xi_i=0, \Var(\xi_i)=1$, and for some $M>0$
	\begin{equation}\label{eq:sec:mo:condition:prop:vector-smallball}
	\inf_{1\leq i\leq n}\E\bigl[\xi_i^2\mathbf 1_{\{|\xi_i|\le M\}}\bigr]\ge \frac12.
	\end{equation}
	Let $w_1,\ldots, w_n \in\R^m$ be fixed $m$ dimensional vectors. If 
	\[
	\sum_{i=1}^n w_iw_i^\top\succeq I_m,
	\qquad
	\max_{1\leq i\leq n}\|w_i\|_2\le\rho,
	\]
    for some $\rho>0$, then there is a constant $C=C(M)<\infty$ such that for every
	$\eta\ge C\rho$,
	\[
	\sup_{z\in\R^m}
	\P\left(
	\Big\|\sum_iw_i\xi_i-z\Big\|_2\le\eta\sqrt m
	\right)
	\le (C\eta)^m.
	\]
\end{proposition}
A couple of remarks are in order.

Proposition~\ref{prop:vector-smallball} is closely resembles the result by Rudelson--Vershynin~\cite[Theorem~1.1]{rudelson-vershynin2015smallball} (see also
\cite[Theorem~1.1]{livshyts-paouris-pivovarov2016marginal}).
Indeed, if one additionally assumes $\xi_i$'s have uniformly bounded density (such as in the Gaussian ellipsoid fitting), Theorem~\ref{prop:vector-smallball} follows from \cite[Theorem 1.1]{rudelson-vershynin2015smallball}, even without the condition $\max_i \|w_i\|_2\lesssim \eta$.

The point of Proposition~\ref{prop:vector-smallball} is that it applies
without any density assumption, and hence also to discrete random
variables. Its proof follows the strategy by \cite{rudelson-vershynin2015smallball}, where we use the Fourier-analytic method of Ball--Nazarov \cite{ball-nazarov-little-level} and the Brascamp--Lieb inequality~\cite{brascamp-lieb1976}. Related bounds for anti-concentration with atomic variables appear in Littlewood--Offord theory; see e.g.~\cite{friedland-sodin2007concentration, rudelson-vershynin2008littlewood, rudelson-vershynin2009rectangular}. Those results control anti-concentration through an arithmetic structure of the coefficients called \textit{essential least common denominator}. Proposition~\ref{prop:vector-smallball} instead requires no arithmetic hypothesis, at the cost of requiring $\eta \gtrsim \max_i \|w_i\|_2$.

On the other hand, a restriction of the form $\eta \gtrsim \max_i \|w_i\|_2$ is unavoidable for discrete $\xi_i$'s. Let $m=1$, and $\xi_1,\ldots,\xi_n$ be i.i.d. Rademacher variables with $n$ even. Set $w_i=n^{-1/2}\equiv \rho$. Note that $S=\frac1{\sqrt n}\sum_{i=1}^n\xi_i$ takes values on a lattice with spacing $2\rho$, and a local central limit theorem shows $\P(S=0)\asymp\rho$. If $\eta=\rho^2$, then the interval $[-\eta,\eta]$ contains no lattice value other than zero, and hence $\P(|S|\le\eta)\asymp\rho$. This is much larger than the proposed bound $C\eta=C\rho^2$ when $\rho$ is small.

Finally, condition~\eqref{eq:sec:mo:condition:prop:vector-smallball} says that at least half of the variance of $\xi_i$ is contributed from $[-M,M]$. The constant $1/2$ is immaterial and may be replaced by any $c\in(0,1)$, with the resulting constants depending on $c$. For a random variable $\xi$, define its L\'{e}vy concentration function by
\begin{equation*}\label{eq:def:anti:concentration}
Q(\xi,L):=\sup_{x\in\mathbb R}\P\left(\xi\in[x,x+L]\right).
\end{equation*}
We prove in Lemma~\ref{lem:bl-oned-factor} that if \eqref{eq:sec:mo:condition:prop:vector-smallball} holds, then the scalar anti-concentration estimate 
\[
\sup_{1\leq i\leq n} Q(\xi_i,1/8) \le 1-c
\]
holds for some $c=c(M)>0$.  More generally, the proof of Proposition~\ref{prop:vector-smallball} uses the assumption~\eqref{eq:sec:mo:condition:prop:vector-smallball} only through such a scalar anti-concentration assumption. It therefore remains valid if~\eqref{eq:sec:mo:condition:prop:vector-smallball} is replaced by $\sup_iQ(\xi_i,L)\le1-\varepsilon$ for some fixed $L,\varepsilon>0$. In that formulation, the conclusion holds for every $\eta\ge C(L,\varepsilon)\rho$, with the constant in $(C\eta)^m$ also depending only on $L$ and $\varepsilon$.

Assuming Proposition~\ref{thm:approx-fit} and Theorem~\ref{thm:centered-compatibility}, we now prove Theorem~\ref{thm:small-residual-interpolation} and the SAT part of Theorem~\ref{thm:phase-transition}. We also use the fact $\alpha_\star(\kappa)\leq 1/2$ proven in Proposition~\ref{prop:scalar-gaussian-energy-curve} via a direct calculation.

\begin{proof}[Proof of Theorem~\ref{thm:small-residual-interpolation} and the SAT assertion in Theorem~\ref{thm:phase-transition}]
	We first prove Theorem~\ref{thm:small-residual-interpolation}. Throughout, we let $c,C>0$ denote constants that depend only on $\alpha$ and $K_x$. By Theorem~\ref{thm:centered-compatibility}, with probability $1-o(1)$ the following holds for every $\lambda\in\one^\perp$: 
	\[
	\|\mathcal{L}_X^*\lambda\|_* \ge
	c\left(\frac{\sqrt d}{(\log d)^4}\|h\|_1+d\|t\|_2\right),
	\]
    where $\lambda=h+t$ is supplied by Theorem~\ref{thm:centered-compatibility}. Using $|\langle \lambda, y\rangle |\leq \|h\|_1\|y\|_{\infty}+\|t\|_2\|y\|_2$ for $y\in \Ran(\mathcal{A}_X)$, we thus have
	\[
	\frac{|\langle\lambda,y\rangle|}{\|\mathcal{L}_X^*\lambda\|_*}
	\le
	C\left(\frac{(\log d)^4}{\sqrt d}\|y\|_\infty+\frac{\|y\|_2}{d}\right).
	\]
	Taking the supremum over $\lambda\in \one^\perp$ and combining with Lemma~\ref{lem:range-duality}, with probability $1-o(1)$, for any $y\in \Ran(\mathcal{A}_X)$,
	\[
	\Xi_{\op}(y)
	\le C\left(\frac{(\log d)^4}{\sqrt d}\|y\|_\infty+\frac{\|y\|_2}{d}\right).
	\]
 By definition of $\Xi_{\op}(y)$, after increasing $C$ by a factor of $2$, there exists $\Delta_y\in \Ssym^d$ satisfying the asserted bound (for $y=0$, take $\Delta_y=0$), which proves Theorem~\ref{thm:small-residual-interpolation}.

	To prove SAT assertion in Theorem~\ref{thm:phase-transition}, suppose
	$\alpha<\alpha_\star(\kappa)$. Since
	$\alpha_\star(\kappa)\le1/2$ by
	Proposition~\ref{prop:scalar-gaussian-energy-curve}, we have $\alpha<1/2$, thus Theorem~\ref{thm:small-residual-interpolation} applies. On the other hand, by
	Proposition~\ref{thm:approx-fit}, there are constants $m,M>0$, depending only on $\alpha,\kappa,K_x$, such that with probability $1-o(1)$ there exists $\widehat{R}_d$ with $mI_d\preceq \widehat{R}_d\preceq MI_d$, and $\widehat{\rho}=-\mathcal{A}_X \widehat{R}_d$ satisfies
	$\|\widehat \rho\|_2/d=o(1)$ and
	$\|\widehat{\rho}\|_{\infty}\leq (\log d)^2$. Combining with Theorem~\ref{thm:small-residual-interpolation}, there exists $\Delta_{\widehat \rho}\in \Ssym^d$ such that $\mathcal{A}_X \Delta_{\widehat \rho}=\widehat{\rho}$ and 
    \[
    \left\|\Delta_{\widehat{\rho}}\right\|_{\op}\leq C\left(\frac{\|\widehat\rho\|_2}{d}+\frac{(\log d)^4\|\widehat\rho\|_{\infty}}{\sqrt{d}}\right)=o(1),
    \]
   Thus, for all sufficiently large $d$, $R:=\widehat R_d+\Delta_{\widehat\rho}$ satisfies
   \[
   \frac{m}{2}I_d\preceq R\preceq 2M I_d,\qquad\mathcal{A}_X R=0.
   \]
Finally, the last identity says that $\mathcal L_X R$ is constant, so $x_i^\top Rx_i=c_d\ge0$ for every $i$. Since $R\succ0$, we have $c_d>0$ whenever there exists $i$ such that $x_i\ne0$. This occurs with probability $1-o(1)$. Indeed, Cauchy--Schwarz gives $\P(x_{11}\ne0)\ge1/\kappa$, and therefore $\P(x_i=0,~\forall i)\le(1-1/\kappa)^n=o(1)$. On the complementary event, $S:=R/c_d$ is positive definite and satisfies $x_i^\top Sx_i=1$ for every $i$. Finally, scaling does not change the condition number, and the displayed spectral bounds give $\lambda_{\max}(S)/\lambda_{\min}(S)\le4M/m$. This proves the SAT assertion in Theorem~\ref{thm:phase-transition}.
\end{proof}

\subsection{UNSAT: universality of constrained least squares}
\label{subsec:proof:overview:unsat}
\label{sec:unsat-architecture-inputs}
\label{sec:self-concordant-ridge-universality}
In this section, we prove
Theorem~\ref{thm:phase-transition}~\textup{(ii)} and
Theorem~\ref{thm:intro-approx-tensorization-squared-loss} via a Lindeberg argument.

We work in a more general framework in which the observations (a.k.a. features) need not be matrices. Specifically, we consider i.i.d. feature
vectors $Z_1,\ldots,Z_n\in\R^N$ with $\E[Z_1]=0$, $\Cov(Z_1)=\Sigma\in \R^{N\times N}$, and write
\[
Z:=(Z_i)_{i\le n},
\qquad
\lambda_\Sigma:=\|\Sigma\|_{\op}>0.
\]
For $\beta\ge0$, define the ridge-regularized \textit{least squared objective}
\[
F_Z(\theta,b;\beta)
:=
\frac1{2n}\sum_{i=1}^n
\bigl(b-\langle Z_i,\theta\rangle\bigr)^2
+\beta\left(\lambda_{\Sigma}\|\theta\|_2^2
+b^2\right)\,,\qquad \theta\in \R^N\,,~b\in \R.
\]
We refer to
$\beta(\lambda_{\Sigma}\|\theta\|_2^2+b^2)$ as the
\emph{ridge regularization}. When $\beta>0$, this makes
$(\theta,b)\mapsto F_Z(\theta,b;\beta)$ strongly convex. We first work
with $\beta>0$ and later remove the ridge regularization.

For a nonempty bounded convex domain $\mathcal D\subset\R^N$, the constriated least squares problem is
\begin{equation}
	\label{eq:constrained-ridge-value}
	\Phi_Z(\mathcal D;\beta)
	:=\inf_{\theta\in\mathcal D,\,b\in\R}F_Z(\theta,b;\beta)\,,\qquad \Phi_Z(\mathcal D):=\Phi_Z(\mathcal D;0).
\end{equation}
We first explain how the set up of Theorem~\ref{thm:intro-approx-tensorization-squared-loss} fits in this framework. Fix an isometry $\operatornorm{vec}:\Ssym^d\to\R^{N_d}$, where
\[
N_d=\frac{d(d+1)}{2},
\qquad
\langle\operatornorm{vec}(A),\operatornorm{vec}(B)\rangle
=\langle A,B\rangle_F.
\]
For example, one may list the diagonal entries followed by the
off-diagonal entries multiplied by $\sqrt2$. With a slight abuse of
notation, we identify matrices and matrix domains with their images under
this isometry, and write
\[
W=\bigl(\operatornorm{vec}(W_i)\bigr)_{i\le n}.
\]
Given $X=(x_i)_{i\le n}$, let $W\equiv W(X)$ denote the associated \textit{quadratic features} 
\[
W(X)_i=(x_ix_i^\top-I_d)/\sqrt{d}.
\]
Then, $\Gamma_X$ can be expressed in terms of $\Phi_{W(X)}$ as follows. Define 
\[
\mathcal R_d:=\big\{R\succeq0: R\in \Ssym^d\,,\, \Tr R=d\big\},
\qquad
\mathcal R_d^\circ:=\big\{R\succ0:R\in \Ssym^d\,,\, \Tr R=d\big\}.
\]
Since $\mathcal R_d^\circ$ is dense in $\mathcal R_d$, we have by continuity
\begin{equation}
	\label{eq:unregularized-affine-energy-value}
	\Phi_{W(X)}(\mathcal R_d^\circ)=\Phi_{W(X)}(\mathcal R_d)=\frac{1}{2}\Gamma_X.
\end{equation}
Similarly, for the Gaussian features $(G_i)_{i\leq n}$ in \eqref{eq:def:gaussian:features}, we have $\Phi_G(\mathcal R_d^{\circ})=\Phi_G(\mathcal R_d)=\Gamma_{G,\kappa}/2$. Thus, since $\Gamma_{G,\kappa}\pto e_\star(\alpha,\kappa)$ by Proposition~\ref{thm:gaussian-master}, in order to establish
Theorem~\ref{thm:intro-approx-tensorization-squared-loss}, it suffices to
prove
\begin{equation}\label{eq:goal:matrix:universality}
\Phi_{W(X)}(\mathcal R_d^\circ)-\Phi_G(\mathcal R_d^\circ)\pto0
\end{equation}
We next isolate the regularity properties of a feature vector used in our universality arguments.

\begin{definition}[Regular feature]
	\label{def:regular-feature-row}
	Fix $K_{\rm feat}>0$ and $\tau\ge0$.  A random vector
	$Z\in\R^N$ is called a
	$(K_{\rm feat},\tau)$-regular feature, if $\E[Z]=0$, $\Cov(Z)=\Sigma$ with
	$\lambda_{\Sigma}=\|\Sigma\|_{\op}>0$, and the following holds:
	\begin{enumerate}[leftmargin=*]
		\item[$\mathrm{(F1)}$] For every fixed $\theta\in\R^N$,
		\[
		\|\langle Z,\theta\rangle\|_{L^4}^2
		\le K_{\rm feat}^2\lambda_{\Sigma}\|\theta\|_2^2.
		\]
		\item[$\mathrm{(F2)}$] For every fixed symmetric matrix
		$T\in\R^{N\times N}$,
		\[
		\Var\left(Z^\top TZ\right)
		\le \tau^2\|T\|_F^2.
		\]
	\end{enumerate}
\end{definition}
To interpret, both conditions are concentration-type assumptions. Indeed, $\fone$ follows from a $2\to 4$
hypercontractive estimate for linear functions of $Z$: if $\|\langle Z,\theta \rangle_{L^4}\leq K_{\rm feat} \|\langle Z, \theta \rangle \|_{L^2}$, then $\fone$ follows from
$\|\langle Z,\theta\rangle\|_{L^2}^2
\le\lambda_\Sigma\|\theta\|_2^2$. $\ftwo$ is a uniform variance bound for quadratic functions of $Z$.

Note that the Gaussian feature $G\sim N(0,\Sigma_\kappa)$ defined in
\eqref{eq:def:gaussian:features} is a $(C,C/d)$-regular feature, where
$C$ depends only on $\kappa$. Indeed, since
$\|\Sigma_\kappa\|_{\op}\asymp_\kappa d^{-1}$ (see \eqref{eq:Sigma-kappa-def-main} for the definition of $\Sigma_{\kappa}$), Gaussian moment formulas
give
\begin{equation}
\label{eq:gaussian-feature-regularity}
\|\langle G,\theta\rangle\|_{L^4}^2=\sqrt{3}\theta^\top \Sigma_{\kappa}\theta
\le \sqrt3\,\|\Sigma_\kappa\|_{\op}\|\theta\|_2^2,
\qquad
\Var(G^\top TG)=2\left\| \Sigma_{\kappa}^{1/2}T \Sigma_{\kappa}^{1/2}\right\|_{F}^2
\le2\|\Sigma_\kappa\|_{\op}^2\|T\|_F^2.
\end{equation}
Corollary~\ref{cor:matrix-scale-universality}~\textup{(ii)} below shows that
\eqref{eq:goal:matrix:universality} follows once the rows of the
quadratic feature array $W(X)$ are $(K_{\rm feat},\tau_d)$-regular with
\[
K_{\rm feat}=O(1),
\qquad
\tau_d\lesssim(\log d)^{-3}.
\]
The following proposition, proved in
Section~\ref{subsec:approx-tensorization-regularity}, verifies these
conditions.

\begin{proposition}
	\label{thm:approx-tensorization-feature-regularity}
	Consider a random vector $x\in \R^d$ such that $\E[x]=0$ and $\Cov(x)=I_d$. Assume that $x$ satisfies approximate tensorization of variance~\eqref{eq:approximate-variance-tensorization} with constant
	$C>0$, and that $\max_{j\le d}\E|x_j|^8\le M,$ where $M$ is independent of $d$. Set
	\[
	W:=\frac{xx^\top-I_d}{\sqrt d},
	\]
	and suppose that
	$\Cov(\operatornorm{vec}W)=\Sigma_\kappa$ for some fixed
	$\kappa>1$. Then there exists
	$C_{\rm reg}=C_{\rm reg}(C,M,\kappa)<\infty$ such that $W$ is a
	$(K_{\rm feat},\tau_d)$-regular feature with
	\[
	K_{\rm feat}\le C_{\rm reg},
	\qquad
	\tau_d\le C_{\rm reg}d^{-1/2}.
	\]
\end{proposition}

\subsubsection{Universality from self-concordant barriers}

Although ridge regularization with $\beta>0$ makes
$(\theta,b)\mapsto F_Z(\theta,b;\beta)$ strongly convex, its constrained minimum may lie at the boundary of $\mathcal D$. We therefore add a \emph{self-concordant barrier}, which places the regularized minimizer strictly in the interior, and more importantly, provides the local control needed for a second-order
Taylor expansion (cf. Lemma~\ref{lem:self-concordance-admissibility-consequences}) in our Lindeberg argument.

We use the following formulation from
\cite[Definition~2.3.1]{nesterov-nemirovskii1994interior}; see also
\cite[Definition~2.1.1, Remark~2.1.1]
{nesterov-nemirovskii1994interior},
\cite[Sections~4.1, 4.2]{nesterov2004introductory}, and
\cite[Sections~2.2, 2.3]{renegar2001interior}.
\begin{definition}[$\vartheta$-self-concordant barrier]
	\label{def:self-concordance}
Fix $\vartheta\ge1$. Let $N\ge1$, and let
$\Omega\subset\R^N$ be convex and relatively open in its affine hull,
that is, in the smallest affine subspace containing $\Omega$. Set
	\[
	\mathsf V_\Omega
	:=\operatorname{span}\{z-z':z,z'\in\Omega\}.
	\]
	Let $\Psi:\Omega\to\R$ be convex and $C^3$ along line segments in
	$\Omega$. Define the local seminorm
	\[
		\|h\|_{\Psi,z}^2:=\mathrm D^2\Psi(z)[h,h],
		\qquad z\in\Omega,\quad h\in\mathsf V_\Omega.
	\]
	We call such a convex function $\Psi$ a \emph{$\vartheta$-self-concordant barrier} for
	$\Omega$ if $\Psi(z)\to\infty$ as $z$ approaches the relative boundary
	of $\Omega$ and, for every $z\in\Omega$ and $h\in\mathsf V_\Omega$,
	\begin{equation*}
		\tag{SC}
		\bigl|\mathrm D^3\Psi(z)[h,h,h]\bigr|
		\le 2\|h\|_{\Psi,z}^3.
	\end{equation*}
	It must also satisfy the barrier-parameter bound for all $h\in \mathsf V_\Omega$
	\begin{equation*}
		\tag{$\mathrm{BP}_\vartheta$}
		|\mathrm D\Psi(z)[h]|
		\le \sqrt{\vartheta}\,\|h\|_{\Psi,z}.
	\end{equation*}
\end{definition}
The constant $2$ in $\mathrm{(SC)}$ fixes the standard normalization.
Under this normalization, $\Psi(t)=-\log t$ satisfies both
$\mathrm{(SC)}$ and $\mathrm{(BP}_1)$ with equality. The two conditions
play distinct roles: $\mathrm{(SC)}$ provides the local control needed for
the Taylor expansion in the Lindeberg argument, while
$\mathrm{(BP}_\vartheta)$ controls the cost of removing the barrier.

For $\beta\geq 0,\gamma\geq 0$, define
\begin{equation}
	\label{eq:shared-regularized-ls-value}
	\Phi_Z(\mathcal D,\mathcal B;\beta,\gamma)
	:=
	\inf_{\theta\in\mathcal D,\,b\in\R}
	\{F_Z(\theta,b;\beta)+\gamma\mathcal B(\theta)\}.
\end{equation}
We will first compare $\Phi_W(\mathcal D,\mathcal B;\beta,\gamma)$ and
$\Phi_G(\mathcal D,\mathcal B;\beta,\gamma)$ for $\beta,\gamma>0$, and then remove the barrier and ridge
regularizations. Note that for $\gamma=0$, we have $	\Phi_Z(\mathcal D,\mathcal B;\beta,0)=
	\Phi_Z(\mathcal D;\beta)$ by definition.

It is worth noting that every proper convex domain of affine dimension $m$ admits a
self-concordant barrier with parameter $O(m)$
\cite[Theorem~2.5.1]{nesterov-nemirovskii1994interior}; see also
\cite[Theorem~4.3.2]{nesterov2004introductory}. The matrix domain $\mathcal{R}_d^{\circ}$ in our application admit sharper log-determinant barrier with parameter
$\vartheta=O(d)$, although its dimension is $O(d^2)$.
This sharper scaling plays an essential role when we remove the barrier in
Section~\ref{subsubsec:barrier-ridge-removal}. We introduce the relevant barriers next.

To prove the approximate fitting in Proposition~\ref{thm:approx-fit}, introduce the domain
\begin{equation}
	\label{eq:two-sided-barrier-pair}
	\mathcal R_d^{(m,M)}:=\{R:mI_d\prec R\prec MI_d,\ \Tr R=d\},\qquad 0<m<1<M.
\end{equation}

\begin{lemma}
	\label{lem:logdet-barriers}
	On the domain $\mathcal R_d^\circ$, the function
    \begin{equation}
	\label{eq:one-sided-barrier-pair}
	\mathcal B_+(R):=-\log\det R.
\end{equation}
is a non-negative $d$-self-concordant barrier. For fixed $0<m<1<M$, on the domain $\mathcal R_d^{(m,M)}$, the function
    \begin{equation}
	\label{eq:two-sided-logdet-barrier}
	\mathcal B_{m,M}(R)
	:=-\log\det(R-mI_d)-\log\det(MI_d-R)
	+d\log\bigl((1-m)(M-1)\bigr).
\end{equation}
	is a non-negative $2d$-self-concordant barrier. Both barriers vanish at $I_d$. Finally, for
	either $\mathcal{D}=\mathcal{R}_d^\circ$ or $\mathcal{D}=\mathcal{R}_d^{(m,M)}$,
	\[
	\mathsf V_{\mathcal D}
	=\{U\in\Ssym^d:\Tr U=0\},
	\qquad
	\dim\mathsf V_{\mathcal D}=\frac{d(d+1)}2-1.
	\]
\end{lemma}

\begin{proof}
	By \cite[Proposition~5.4.5]{nesterov-nemirovskii1994interior},
	$X\mapsto-\log\det X$ is a $d$-self-concordant barrier on the
	positive-definite cone. Moreover, by \cite[Proposition~2.3.1]{nesterov-nemirovskii1994interior}, affine restriction preserves the barrier
	parameter, while parameters add under sums.
	Restricting to $\mathcal{R}_d^\circ$ gives the assertion for
	$\mathcal B_+$. Composing with the two affine maps
	$R\mapsto R-mI_d$ and $R\mapsto MI_d-R$, and then adding the resulting
	barriers, gives the assertion for $\mathcal B_{m,M}$.

	The AM--GM inequality gives $\det R\le1$ when $\Tr R=d$, so
	$\mathcal B_+(R)\ge0$, with equality at $R=I_d$. For $\mathcal{B}_{m,M}$, apply the tangent-line inequality at $1$ to the convex function $s\mapsto-\log(s-m)-\log(M-s)$ and sum over the eigenvalues of $R$. The linear terms cancel because
	$\Tr R=d$, and the normalization in
	\eqref{eq:two-sided-logdet-barrier} gives
	$\mathcal B_{m,M}(R)\ge0$ and $\mathcal B_{m,M}(I_d)=0$.
	The direction space of either domain and its dimension follow directly
	from the trace constraint.
\end{proof}

The local seminorms associated with these barriers can be written explicitly.
A direct differentiation gives
\begin{equation}
	\label{eq:logdet-local-norm}
	\|U\|_{\mathcal B_+,R}^2
	=
	\mathrm D^2\mathcal B_+(R)[U,U]
	=
	\Tr(R^{-1}UR^{-1}U)
	=
	\|R^{-1/2}UR^{-1/2}\|_F^2.
\end{equation}
For the two-sided barrier, writing
$S_-=R-mI_d$ and $S_+=MI_d-R$, we similarly obtain
\[
	\|U\|_{\mathcal B_{m,M},R}^2
	=
	\|S_-^{-1/2}US_-^{-1/2}\|_F^2
	+
	\|S_+^{-1/2}US_+^{-1/2}\|_F^2.
\]
Thus the local seminorms become large in directions toward a nearby spectral boundary.

We are now ready to state the main comparison theorem used to prove Theorem~\ref{thm:intro-approx-tensorization-squared-loss}.
\begin{theorem}[Barrier-regularized universality]
	\label{thm:positive-barrier-row-universality}
	Fix $N,n\ge1$, $K_{\rm feat}>0$, and $\beta,\gamma>0$.
	Let $W\in\R^N$ be a $(K_{\rm feat},\tau)$-regular feature with
		covariance matrix $\Sigma$, and set
		$\lambda_{\Sigma}=\|\Sigma\|_{\op}>0$.  Let $G$ be the centered Gaussian
	feature with covariance $\Sigma$, and let
	$W_1,\ldots,W_n$ and $G_1,\ldots,G_n$ be independent arrays of i.i.d.
	copies of $W$ and $G$, respectively.

	There are constants $C,n_0<\infty$, depending only on
	$K_{\rm feat}$, with the following property.  Let
	$\mathcal D\subset\R^N$ be nonempty, bounded, convex, and relatively
	open in its affine hull, and let
	$\mathcal B:\mathcal D\to[0,\infty)$ be a
	$\vartheta$-self-concordant barrier for some $\vartheta\ge1$.
	Define the reference scale
	\begin{equation}
		\label{eq:intrinsic-reference-scale}
		\Lambda_{\rm ref}
		:=\left(1+\frac{1}{\beta}\right)
		\left(
		1+\inf_{\theta\in\mathcal D}
		\{\lambda_{\Sigma}\|\theta\|_2^2+\gamma\mathcal B(\theta)\}
		\right).
	\end{equation}
	If $n\gamma\ge n_0$, then
		\begin{equation}
			\label{eq:feature-universality-expectation}
			\left|
			\E\Phi_W(\mathcal D,\mathcal B;\beta,\gamma)
			-
			\E\Phi_G(\mathcal D,\mathcal B;\beta,\gamma)
			\right|
			\le
			C\Lambda_{\rm ref}^2\left(
			\frac1{\sqrt{n\gamma}}
			+\frac{\sqrt{\dim\mathsf V_{\mathcal D}}(\tau/\lambda_{\Sigma}+1)}{n}
			\right).
		\end{equation}
	Moreover,
		\begin{equation}
			\label{eq:feature-universality-variance}
			\Var\!\left(\Phi_W(\mathcal D,\mathcal B;\beta,\gamma)\right)
			+
			\Var\!\left(\Phi_G(\mathcal D,\mathcal B;\beta,\gamma)\right)
			\le C\frac{\Lambda_{\rm ref}^2}{n},
		\end{equation}
	and hence, for every $t>0$,
		\begin{equation}
			\label{eq:feature-universality-high-probability}
			\P\left(
			\left|\Phi_W(\mathcal D,\mathcal B;\beta,\gamma)
			-\Phi_G(\mathcal D,\mathcal B;\beta,\gamma)\right|
			>C\Lambda_{\rm ref}^2\left(
			\frac1{\sqrt{n\gamma}}
			+
			\frac{\sqrt{\dim\mathsf V_{\mathcal D}}(\tau/\lambda_{\Sigma}+1)}{n}
			\right)+t
			\right)
			\le \frac{C\Lambda_{\rm ref}^2}{nt^2}.
		\end{equation}
\end{theorem}
The first term on the right-hand side of
\eqref{eq:feature-universality-expectation} controls the error
from the second-order approximation used in feature replacement argument and
uses the condition $\mathrm{(F1)}$. The second term arises from
fluctuations of the corresponding quadratic terms and uses the condition $\mathrm{(F2)}$. Theorem~\ref{thm:positive-barrier-row-universality} is proved in
Section~\ref{subsubsec:proof-positive-barrier-row-universality}.

\subsubsection{Removing the barrier and ridge regularization}
\label{subsubsec:barrier-ridge-removal}

We next remove the two regularizations introduced above. We first compare
$\Phi_Z(\mathcal D,\mathcal B;\beta,\gamma)$ with
$\Phi_Z(\mathcal D;\beta)$, and then remove the ridge regularization to recover
$\Phi_Z(\mathcal D)$. The following two lemmas show how the barrier parameter $\vartheta$ enters the first comparison.
\begin{lemma}
	\label{lem:self-concordant-barrier-contraction}
	Let $\Psi:\Omega\to[0,\infty)$ be a
	$\vartheta$-self-concordant barrier on a convex domain.  For
	$x_0,x\in\Omega$ and $0<\varepsilon<1$,
	\[
		\Psi((1-\varepsilon)x+\varepsilon x_0)
		\le \Psi(x_0)+\vartheta\log(1/\varepsilon).
	\]
\end{lemma}
\begin{proof}
This is a consequence of the standard global gradient bound
\[
  \mathrm D\Psi(z)[y-z]\le \vartheta,
  \qquad z,y\in\Omega,
\]
for self-concordant barriers (see e.g.~\cite[Theorem~4.2.4]{nesterov2004introductory}); we recall the proof of this first. If the left-hand side is nonpositive this is immediate.
	Otherwise, put $f(t)=\Psi(z+t(y-z))$.  Convexity gives $f'(t)>0$, while
	the barrier-parameter bound gives $f'(t)^2\le\vartheta f''(t)$.  Hence
	$(1/f')'\le-1/\vartheta$; integration over $[0,1]$ yields
	$f'(0)\le\vartheta$ which is equivalent to the claim.

	Now set $x_s=(1-s)x_0+sx$.  Since
	$x-x_s=(1-s)(x-x_0)$, the preceding inequality gives $	\frac{d}{ds}\Psi(x_s)
		\le\frac{\vartheta}{1-s},$ and integrating from $0$ to $1-\varepsilon$ proves the claim.
\end{proof}

\begin{lemma}
	\label{lem:general-barrier-removal-sandwich}
	Let $\mathcal D\subset\R^N$ be nonempty, bounded, convex, and relatively
	open in its affine hull, and let
	$\mathcal B:\mathcal D\to[0,\infty)$ be a
	$\vartheta$-self-concordant barrier for some $\vartheta\ge1$.  Fix 
	$(\theta_{\rm ref},b_{\rm ref})\in\mathcal D\times\R$.  Then, for every
	features $Z=(Z_i)_{i\leq n}$, every $\beta\ge0$, every $\gamma>0$, and every
	$0<\varepsilon<1$,
	\begin{equation}
		\label{eq:general-barrier-removal-sandwich}
		0\le
		\Phi_Z(\mathcal D,\mathcal B;\beta,\gamma)
		-\Phi_Z(\mathcal D;\beta)
		\le
		\varepsilon F_Z(\theta_{\rm ref},b_{\rm ref};\beta)
		+\gamma\mathcal B(\theta_{\rm ref})
		+\gamma\vartheta\log(1/\varepsilon).
	\end{equation}
\end{lemma}

\begin{proof}
	The left inequality follows from $\mathcal B\ge0$.  Given
	$(\theta,b)\in\mathcal D\times\R$, contract it toward the $(\theta_{\rm ref}, b_{\rm ref})$:
	\[
		\theta_\varepsilon
		=(1-\varepsilon)\theta+\varepsilon\theta_{\rm ref},
		\qquad
		b_\varepsilon
		=(1-\varepsilon)b+\varepsilon b_{\rm ref}.
	\]
		By convexity and non-negativity of $F_Z(\cdot,\cdot;\beta)$,
	\[
		F_Z(\theta_\varepsilon,b_\varepsilon;\beta)
		\leq F_Z(\theta,b;\beta)+\eps F_Z(\theta_{\rm ref}, b_{\rm ref}; \beta).
	\]
   Lemma~\ref{lem:self-concordant-barrier-contraction} also gives $\mathcal B(\theta_\varepsilon)
		\le\mathcal B(\theta_{\rm ref})
		+\vartheta\log(1/\varepsilon)$.
	Using $(\theta_\varepsilon,b_\varepsilon)$ in the barrier-regularized
	problem~$\Phi_Z(\mathcal{D},\mathcal{B};\beta,\gamma)$ and taking the infimum over $(\theta,b)$ proves the upper bound.
\end{proof}
Applying \eqref{eq:general-barrier-removal-sandwich} for $Z=W,G$ and
using the triangle inequality gives
\begin{align}
	\label{eq:general-barrier-removal-comparison}
	\left|\Phi_W(\mathcal D;\beta)-\Phi_G(\mathcal D;\beta)\right|
	&\le
	\left|\Phi_W(\mathcal D,\mathcal B;\beta,\gamma)
	-\Phi_G(\mathcal D,\mathcal B;\beta,\gamma)\right|\notag\\
	&\quad+
	\varepsilon\sum_{Z\in\{W,G\}}
	F_Z(\theta_{\rm ref},b_{\rm ref};\beta)
	+2\gamma\mathcal B(\theta_{\rm ref})
	+2\gamma\vartheta\log(1/\varepsilon).
\end{align}
Combining with Theorem~\ref{thm:positive-barrier-row-universality}, the schematic cost of comparing the unregularized values is
\[
	\underbrace{(n\gamma)^{-1/2}}_{\text{feature comparison}}
	+\underbrace{\gamma\vartheta\log(1/\varepsilon)}_{\text{barrier removal}}
	+\underbrace{\varepsilon}_{\text{contraction of the objective}},
\]
in addition to $\frac{\sqrt{\dim\mathsf V_{\mathcal D}}(\tau/\lambda_{\Sigma}+1)}{n}$ in
\eqref{eq:feature-universality-expectation}. Choosing $\varepsilon=n^{-1}$ and
$\gamma\asymp(n\vartheta^2\log^2 n)^{-1/3}$ gives an error of order
\[
	\left(\frac{1+\vartheta\log n}{n}\right)^{1/3};
\]
see Proposition~\ref{prop:constrained-ridge-comparison}. Thus, when $\mathcal{D}=\mathcal{R}_d^\circ$ and $n\asymp d^2$, the generic scaling $\vartheta\asymp d^2$ would not yield
a vanishing error. Lemma~\ref{lem:logdet-barriers} instead gives
$\vartheta\le2d$, yielding the rate $(\log d/d)^{1/3}=o(1)$.

It remains to remove the ridge regularization. For this step, we restrict
to $\mathcal D=\mathcal R_d^\circ$. We use the
small-ball method~\cite{koltchinskii-mendelson2013smallest, wainwright2019high} to localize an unregularized minimizer. This requires a Rademacher complexity bound for
$\mathcal R_d^\circ$. Since $\|R\|_*=\Tr R=d$ for every
$R\in\mathcal R_d$, this domain is contained in the nuclear-norm ball of
radius $d$; see \cite{foygel2011concentration} for related complexity
bounds. The following lemma is proved in
Section~\ref{subsec:feature-ridge-removal-proof}.

\begin{lemma}[Removing the ridge]
	\label{lem:feature-ridge-removal}
	Fix $\kappa>1$ and $K_{\rm feat}>0$. For each
	$d$, let $W_1,\ldots,W_n\in\Ssym^d$ be i.i.d. copies of a
	$(K_{\rm feat},\tau_d)$-regular feature with covariance
	$\Sigma_\kappa$, and let $G_1,\ldots,G_n$ be independent Gaussian
	features with the same covariance. Write $W=(W_i)_{i\le n}$ and
	$G=(G_i)_{i\le n}$.
    
    Suppose that $n/d^2\to\alpha\in(0,\infty)$ with $\tau_d\le C_0(\log d)^{-3}$ for some fixed $C_0>0$. Then there is
	$C=C(\alpha,\kappa,K_{\rm feat},C_0)>0$ such that, with
	probability tending to one, simultaneously for $Z=W,G$ and every
	$\beta>0$,
	\[
		0\le
		\Phi_Z(\mathcal R_d^\circ;\beta)-\Phi_Z(\mathcal R_d^\circ)
		\le C\beta\log d.
	\]
\end{lemma}


\begin{corollary}
	\label{cor:matrix-scale-universality}
	Fix $\kappa>1$ and $K_{\rm feat}>0$. For each $d$, let
	$W_1,\ldots,W_n\in\Ssym^d$ be i.i.d. copies of a
	$(K_{\rm feat},\tau_d)$-regular feature with covariance
	$\Sigma_\kappa$, and let $G_1,\ldots,G_n$ be independent Gaussian
	features with the same covariance. Write $W=(W_i)_{i\le n}$ and
	$G=(G_i)_{i\le n}$, and assume
	$n/d^2\to\alpha\in(0,\infty)$.
    Then, the following holds.
	\begin{enumerate}[label=\textup{(\roman*)}]
		\item Let $(\mathcal D_d,\mathcal B_d)$ be either
		$(\mathcal R_d^\circ,\mathcal B_+)$ or $(\mathcal R_d^{(m,M)},\mathcal B_{m,M})$ for fixed
		$0<m<1<M<\infty$,
		.
		For every deterministic sequence $0<\beta_d\le1$,
		\begin{equation}
			\label{eq:matrix-scale-ridge-universality}
			\left|\Phi_W(\mathcal D_d;\beta_d)
			-\Phi_G(\mathcal D_d;\beta_d)\right|
			=O_\P\left((1+\beta_d^{-1})^2
			\left(\tau_d+
			\left(\frac{\log d}{d}\right)^{1/3}\right)\right).
		\end{equation}

		\item If, in addition, $\tau_d\le C_0(\log d)^{-3}$ for some fixed
		$C_0>0$, then
		\[
			\Phi_W(\mathcal R_d^\circ)
			-\Phi_G(\mathcal R_d^\circ)\pto0.
		\]
	\end{enumerate}
\end{corollary}

\begin{proof}
	We first prove \textup{(i)}. Fix either $(\mathcal{D}_d, \mathcal{B}_d)=(\mathcal R_d^\circ,\mathcal B_+)$ or $(\mathcal{D}_d, \mathcal{B}_d)=(\mathcal R_d^{(m,M)},\mathcal B_{m,M})$, and write
	$\lambda_d:=\|\Sigma_\kappa\|_{\op}$. By
	Lemma~\ref{lem:logdet-barriers}, there exists $\vartheta_d$-self-concordant barrier $\mathcal{B}_d$ with $\mathcal B_d(I_d)=0$ and $\vartheta_d\asymp d$. Moreover, a direct calculation shows for $Z=W,G$,
	\[
		\E F_Z(I_d,0;\beta_d)
		= \frac{1}{2}\E\left[\left(\langle Z, I_d\rangle\right)^2\right]+\beta_d\lambda_d\|I_d\|_{F}^2=
		\frac{1}{2}\mathcal{C}_{\kappa}(I_d,I_d)
		+\beta_d\lambda_d d
		=O_\kappa(1),
	\]
	uniformly over $0<\beta_d\le1$. Hence
	$F_Z(I_d,0;\beta_d)=O_\P(1)$. Set
	\[
		\varepsilon=n^{-1},
		\qquad
		\gamma=d^{-4/3}(\log d)^{-2/3},
	\]
	so that $(n\gamma)^{-1/2}	\asymp \gamma\vartheta_d\log (1/\eps)\asymp
		\left(\frac{\log d}{d}\right)^{1/3}$. Applying \eqref{eq:general-barrier-removal-comparison} (cf. Lemma~\ref{lem:general-barrier-removal-sandwich}) with reference
	point $(I_d,0)$ therefore gives
	\begin{align}
		\left|\Phi_W(\mathcal D_d;\beta_d)
		-\Phi_G(\mathcal D_d;\beta_d)\right|
		\le
		\left|\Phi_W(\mathcal D_d,\mathcal B_d;\beta_d,\gamma)
		-\Phi_G(\mathcal D_d,\mathcal B_d;\beta_d,\gamma)\right|
		+O_\P\left(\left(\frac{\log d}{d}\right)^{1/3}\right).
		\label{eq:matrix-barrier-comparison-direct}
	\end{align}
	Evaluating \eqref{eq:intrinsic-reference-scale} at $I_d$ in the definiton of $\Lambda_{\rm ref}$ gives
	\[
		\Lambda_{\rm ref}
		\le C_\kappa(1+\beta_d^{-1}),
	\]
	since $\lambda_d d=O_\kappa(1)$. Furthermore, since $\dim\mathsf V_{\mathcal D_d}\asymp d^2$ by Lemma~\ref{lem:logdet-barriers},
	\[
		\frac{\sqrt{\dim\mathsf V_{\mathcal D_d}}}{n}
		\left(\frac{\tau_d}{\lambda_d}+1\right)
		=O_\kappa(\tau_d+d^{-1}).
	\]
	Note that $d^{-1}$ term is absorbed by $(\log d/d)^{1/3}$. Thus, Theorem~\ref{thm:positive-barrier-row-universality}, with
	\[
    t=(1+\beta_d^{-1})^2q_d\,, \quad\textnormal{where}\quad q_d:=\tau_d
			+\left(\frac{\log d}{d}\right)^{1/3}
    \]
    yields
	\[
		\left|\Phi_W(\mathcal D_d,\mathcal B_d;\beta_d,\gamma)
		-\Phi_G(\mathcal D_d,\mathcal B_d;\beta_d,\gamma)\right|
		=
		O_\P\bigl((1+\beta_d^{-1})^2q_d\bigr).
	\]
	Here its probability bound tends to zero because $nq_d^2\to\infty$.
	Combining this with
	\eqref{eq:matrix-barrier-comparison-direct} proves \textup{(i)}.

	For \textup{(ii)}, choose $\beta_d:=\left(\frac{q_d}{\log d}\right)^{1/3}$. Part~\textup{(i)}, applied to
	$(\mathcal R_d^\circ,\mathcal B_+)$, and
	Lemma~\ref{lem:feature-ridge-removal} give
	\begin{align*}
		\left|\Phi_W(\mathcal R_d^\circ)
		-\Phi_G(\mathcal R_d^\circ)\right|=
		O_\P\bigl(q_d(1+\beta_d^{-1})^2+\beta_d\log d\bigr)=
		O_\P\bigl(q_d^{1/3}(\log d)^{2/3}\bigr)
		=o_\P(1),
	\end{align*}
	where the last step follows since $\tau_d\le C_0(\log d)^{-3}$. This concludes the proof.
\end{proof}
\begin{proof}[Proofs of
	Theorem~\ref{thm:intro-approx-tensorization-squared-loss} and
	Theorem~\ref{thm:phase-transition}~\textup{(ii)}]
	Under the assumptions of
	Theorem~\ref{thm:intro-approx-tensorization-squared-loss},
	Proposition~\ref{thm:approx-tensorization-feature-regularity} shows
	that $W_1=(x_1x_1^\top-I_d)/\sqrt{d}$ is regular with $\tau_d=O(d^{-1/2})$.
	By \eqref{eq:unregularized-affine-energy-value} and
	Corollary~\ref{cor:matrix-scale-universality}~\textup{(ii)},
	\[
		\Gamma_X-\Gamma_{G,\kappa}
		=
		2\left[
		\Phi_{W(X)}(\mathcal R_d^\circ)
		-\Phi_G(\mathcal R_d^\circ)
		\right]
		\pto0.
	\]
	Combining with Proposition~\ref{thm:gaussian-master}~\textup{(ii)} therefore proves
	Theorem~\ref{thm:intro-approx-tensorization-squared-loss}.

Under the assumptions of Theorem~\ref{thm:phase-transition}, independence
gives approximate tensorization~\eqref{eq:approximate-variance-tensorization}
with $C=1$, while subgaussianity gives the required eighth-moment bound.
Moreover, a direct expansion shows that
$\Cov(\operatorname{vec}W_i)=\Sigma_\kappa$, with $\Sigma_\kappa$ defined
in~\eqref{eq:Sigma-kappa-def-main}. Thus
Theorem~\ref{thm:intro-approx-tensorization-squared-loss} gives
$\Gamma_X\pto e_\star(\alpha,\kappa)$. If
$\alpha>\alpha_\star(\kappa)$, then $e_\star(\alpha,\kappa)>0$ by
Proposition~\ref{thm:gaussian-master}, whereas an ellipsoid fit would imply
$\Gamma_X=0$. Consequently, the probability of an ellipsoid fit tends to
zero, proving Theorem~\ref{thm:phase-transition}~\textup{(ii)}.
\end{proof}

\subsubsection{Lindeberg argument: proof of
Theorem~\ref{thm:positive-barrier-row-universality}}
\label{subsubsec:proof-positive-barrier-row-universality}

We now prove Theorem~\ref{thm:positive-barrier-row-universality}. We begin
by introducing the notations used in the Lindeberg argument. Write $\mathsf V:=\mathsf V_{\mathcal D}$ and let
$P_{\mathsf V}$ denote the Euclidean projection onto $\mathsf V$. Equip
$\mathsf V\oplus\R$ with the inner product
\[
	\langle u,v\rangle
	:=
	\langle u^\theta,v^\theta\rangle+u^bv^b,
	\qquad
	u=(u^\theta,u^b),\quad v=(v^\theta,v^b).
\]
Let $W=(W_{i})_{i\leq n}$ and $G=(G_i)_{i\leq n}$ be independent, and define for $0\le i\le n$, 
\[
	Z^{[i]}:=(W_1,\ldots,W_i,G_{i+1},\ldots,G_n),
	\qquad
	\Phi_i
	:=
	\Phi_{Z^{[i]}}(\mathcal D,\mathcal B;\beta,\gamma).
\]
Thus $\Phi_0=\Phi_G(\mathcal D,\mathcal B;\beta,\gamma)$ and $\Phi_n=\Phi_W(\mathcal D,\mathcal B;\beta,\gamma)$, while $\Phi_i$ and
$\Phi_{i-1}$ differ only in the $i$'th feature. Their shared randomness is
generated by
$W_1,\ldots,W_{i-1},G_{i+1},\ldots,G_n$. The remaining feature is $W_i$
in $\Phi_i$ and $G_i$ in $\Phi_{i-1}$; these two features are independent
of each other and of the shared randomness. Accordingly, for each
$i\in[n]$, define
\begin{equation}\label{eq:def:F:leave:one:out}
	\mathcal F^{(i)}
	:=
	\sigma(W_1,\ldots,W_{i-1},G_{i+1},\ldots,G_n),
	\qquad
	\E_i[\,\cdot\,]
	:=
	\E[\,\cdot\mid\mathcal F^{(i)}].
\end{equation}
The objective obtained by omitting the $i$'th feature is
\begin{equation}\label{eq:def:objective:omit:i}
	F^{(i)}(\theta,b):=
	\frac1{2n}\bigg(\sum_{j<i}
		\bigl(\langle W_j,\theta\rangle-b\bigr)^2
	+\sum_{j>i}
		\bigl(\langle G_j,\theta\rangle-b\bigr)^2\bigg)
	+\beta\left(\lambda_\Sigma\|\theta\|_2^2
	+b^2\right)
	+\gamma\mathcal B(\theta).
\end{equation}
We use a superscript $(i)$ for quantities associated with this
leave-one-out objective.

To compare $\Phi_i$ and $\Phi_{i-1}$, we first remove the $i$th feature
and then measure the effect of inserting either $W_i$ or $G_i$. Denote the leave-one-out optimizer and optimal value by
\[
	(\theta^{(i)},b^{(i)})
	:=
	\operatorname*{arg\,min}_{\theta\in\mathcal D,\,b\in\R}
	F^{(i)}(\theta,b),
	\qquad
	m^{(i)}:=F^{(i)}(\theta^{(i)},b^{(i)}),
\]
For $z\in\R^N$, the corresponding \emph{insertion cost} $\Delta_i(z)$ is given by
\[
	\Delta_i(z)
	:=
	\inf_{\theta\in\mathcal D,\,b\in\R}
	\left\{
	F^{(i)}(\theta,b)
	+\frac{(b-\langle z,\theta\rangle)^2}{2n}
	\right\}
	-m^{(i)}.
\]
Thus $\Delta_i(z)$ is the increase in the optimal value caused by inserting $z$. Denote
\[
(\theta_i(z),b_i(z)):= 	\operatorname*{arg\,min}_{\theta\in\mathcal D,\,b\in\R}\left\{F^{(i)}(\theta,b)
	+\frac{(b-\langle z,\theta\rangle)^2}{2n}\right\}
\]
as the optimizer after this
insertion, and its change from the leave-one-out optimizer $(\theta^{(i)}, b^{(i)})$ as
\begin{equation}
	\label{eq:insertion-optimizer-update}
	u_i(z)
	:=
	\bigl(u_i^\theta(z),u_i^b(z)\bigr)
	:=
	\bigl(\theta_i(z)-\theta^{(i)},\,
	b_i(z)-b^{(i)}\bigr).
\end{equation}
Since inserting $W_i$ gives
$\Phi_i$, while inserting $G_i$ gives $\Phi_{i-1}$, we have
\[
	\Phi_i=m^{(i)}+\Delta_i(W_i),
	\qquad
	\Phi_{i-1}=m^{(i)}+\Delta_i(G_i).
\]
Consequently, the expectation difference in
Theorem~\ref{thm:positive-barrier-row-universality} reduces to
\begin{align}
	\E\Phi_W(\mathcal D,\mathcal B;\beta,\gamma)
	-\E\Phi_G(\mathcal D,\mathcal B;\beta,\gamma)
	=\sum_{i=1}^n\E\bigl[\Phi_i-\Phi_{i-1}\bigr]
	=\sum_{i=1}^n
	\E\bigl[\Delta_i(W_i)-\Delta_i(G_i)\bigr].
	\label{eq:lindeberg-insertion-telescoping}
\end{align}
We control each summand in the right-hand side by conditioning on $\mathcal F^{(i)}$ and
expanding the objective to second order around the leave-one-out optimizer.

To this end, define the Hessian of the
leave-one-out objective at its minimizer: let
$H^{(i)}:\mathsf V\oplus\R\to\mathsf V\oplus\R$ be the self-adjoint
operator defined by
\[
	\langle u,H^{(i)}v\rangle
	:=
	\mathrm D^2F^{(i)}(\theta^{(i)},b^{(i)})[u,v],
	\qquad u,v\in\mathsf V\oplus\R.
\]
The ridge regularization makes $H^{(i)}$ positive definite, and
$(H^{(i)})^{-1}$ denotes its inverse on $\mathsf V\oplus\R$.

Set
\[
	S^{(i)}
	:=
	\lambda_\Sigma\|\theta^{(i)}\|_2^2+(b^{(i)})^2,
	\qquad
	r_i(z)
	:=
	b^{(i)}-\langle z,\theta^{(i)}\rangle.
\]
Here $S^{(i)}$ measures the size of the leave-one-out optimizer $(\theta^{(i)},b^{(i)})$, while
$r_i(z)$ is the residual of $z$ at this optimizer. The definition of
$\Delta_i(z)$ gives
\begin{equation}
	\label{eq:elementary-insertion-bound}
	0\le \Delta_i(z)\le\frac{r_i(z)^2}{2n}.
\end{equation}
Indeed, the inserted loss is nonnegative, which gives the lower bound,
while evaluating the defining infimum at $(\theta^{(i)},b^{(i)})$ gives
the upper bound. To account for the adjustment of the optimizer, define
\[
	\ell_i(z)
	:=
	\left\langle
	a(z),(H^{(i)})^{-1}a(z)
	\right\rangle, \quad\textnormal{where}\quad a(z):=(-P_{\mathsf V}z,1)\in \mathsf V\oplus \R.
\]
The following proposition gives the approximation of $\Delta_i $ conditional on $\mathcal{F}^{(i)}$, and uses only the condition $\mathrm{(F1)}$ among Definition~\ref{def:regular-feature-row}.
	\begin{proposition}
		\label{prop:quadratic-one-row-response}
		Assume the hypotheses of
	Theorem~\ref{thm:positive-barrier-row-universality}. Then, uniformly over
	$i\in[n]$,
		\begin{equation}
			\label{eq:integrated-leave-one-out-scale}
			\E(S^{(i)})^2\le C\Lambda_{\rm ref}^2.
		\end{equation}
	For each
	$i\in[n]$ and $Z_i\in \{W_i,G_i\}$, we have conditionally on $\mathcal{F}^{(i)}$,
		\begin{equation}
			\label{eq:integrated-residual-moments}
			\E_i|r_i(Z_i)|^4\le C(S^{(i)})^2.
		\end{equation}
		Moreover, provided $n\gamma\ge n_0$,
		\begin{equation}
			\label{eq:soft-one-row-update}
			\E_i\left|
			\Delta_i(Z_i)-\frac{r_i(Z_i)^2}{2(n+\ell_i(Z_i))}
			\right|
			\le C\gamma^{-1/2}n^{-3/2}
			\bigl(1+(S^{(i)})^2\bigr).
		\end{equation}
	\end{proposition}
  The proof of Proposition~\ref{prop:quadratic-one-row-response} is
deferred to Section~\ref{subsec:row-replacement-proof}. Here we explain how self-concordance is used in the proof of~\eqref{eq:soft-one-row-update}. Abbreviate
	\[
		r:=r_i(Z_i),\qquad
		a:=a(Z_i),\qquad
		H:=H^{(i)},\qquad
		\ell:=\ell_i(Z_i).
	\]
	For $u=(u^\theta,u^b)\in\mathsf V\oplus\R$, the residual after the
	displacement $u$ is
	\[
		b^{(i)}+u^b-\langle Z_i,\theta^{(i)}+u^\theta\rangle
		=r+\langle a,u\rangle.
	\]
	Since $(\theta^{(i)},b^{(i)})$ minimizes $F^{(i)}$, its first-order
	variation vanishes. Moreover, in the definition of $F^{(i)}(\theta,b)$ in \eqref{eq:def:objective:omit:i}, every term is quadratic except for $\gamma \mathcal{B}(\theta)$. As a result,
	\begin{equation}
		\label{eq:leave-one-out-taylor-identity}
		F^{(i)}(\theta^{(i)}+u^\theta,b^{(i)}+u^b)
		-F^{(i)}(\theta^{(i)},b^{(i)})
		=
		\frac12\langle u,Hu\rangle+R_i(u),
	\end{equation}
	where
	\begin{equation}
		\label{eq:definition-R-i}
		R_i(u)
		:=
		\gamma\left(
		\mathcal B(\theta^{(i)}+u^\theta)-\mathcal B(\theta^{(i)})
		-\mathrm D\mathcal B(\theta^{(i)})[u^\theta]
		-\frac12\mathrm D^2\mathcal B(\theta^{(i)})
			[u^\theta,u^\theta]
		\right).
	\end{equation}
	Consequently,
	\[
		\Delta_i(Z_i)
		=
		\inf_{\substack{u\in\mathsf V\oplus\R\\
		\theta^{(i)}+u^\theta\in\mathcal D}}
		\left\{
		\frac12\langle u,Hu\rangle
		+\frac{(r+\langle a,u\rangle)^2}{2n}
		+R_i(u)
		\right\}.
	\]
  Dropping both $R_i(u)$ and the constraint $\theta^{(i)}+u^{\theta}\in \mathcal{D}$ gives the quadratic optimization problem:
	\begin{equation}
		\label{eq:quadratic-response-formula}
		\inf_{u\in\mathsf V\oplus\R}
		\left\{
		\frac12\langle u,Hu\rangle
		+\frac{(r+\langle a,u\rangle)^2}{2n}
		\right\}
		=
		\frac{r^2}{2(n+\ell)}.
	\end{equation}
Observe that the minimizer of
\eqref{eq:quadratic-response-formula} need not be feasible for the
original problem defining $\Delta_i(Z_i)$. This is where
self-concordance plays an essential role: for $\gamma>0$, it allows us
both to control the remainder $R_i$ and to verify the feasibility of the minimizer of~\eqref{eq:quadratic-response-formula} (see Lemma~\ref{lem:self-concordant-insertion-control}).

The following standard consequence of self-concordance
\cite{nesterov-nemirovskii1994interior,nesterov2004introductory} controls the error incurred by dropping $R_i$. To formulate it, recall
that for a differentiable convex function $F:\Omega\to\R$, the Bregman
divergence from $z_{\rm ref}$ to $z$ (see, e.g.,
\cite{hazan2016introduction}) is

	\begin{equation}
		\label{eq:bregman-divergence-def}
		D_F(z,z_{\rm ref})
		:=
		F(z)-F(z_{\rm ref})
		-\mathrm DF(z_{\rm ref})[z-z_{\rm ref}].
	\end{equation}

\begin{lemma}
	\label{lem:self-concordance-admissibility-consequences}
	Let $\mathcal D\subset\R^N$ be nonempty, bounded, convex, and
	relatively open in its affine hull, and let
	$\mathcal B:\mathcal D\to\R$ be a $\vartheta$-self-concordant barrier
	for some $\vartheta\ge1$. There are universal constants
	$c,c_0,\rho_0>0$ and $C>0$ such that, for every
	$\theta\in\mathcal D$ and $u\in\mathsf V_{\mathcal D}$:
	\begin{enumerate}[label=\textup{(\roman*)}]
		\item If $\theta+u\in\mathcal D$ and
		$D_{\mathcal B}(\theta+u,\theta)\le c_0$, then
		\[
			\|u\|_{\mathcal B,\theta}\le\rho_0,
			\qquad
			D_{\mathcal B}(\theta+u,\theta)
			\ge c\|u\|_{\mathcal B,\theta}^2.
		\]
		\item If $\|u\|_{\mathcal B,\theta}\le\rho_0$, then
		$\theta+u\in\mathcal D$ and
		\[
			\left|
			\mathcal B(\theta+u)-\mathcal B(\theta)
			-\mathrm D\mathcal B(\theta)[u]
			-\frac12\mathrm D^2\mathcal B(\theta)[u,u]
			\right|
			\le C\|u\|_{\mathcal B,\theta}^3.
		\]
	\end{enumerate}
\end{lemma}

\begin{proof}
		Identify the affine hull of $\mathcal D$ with a Euclidean space and put $r=\|u\|_{\mathcal B,\theta}$. For $\varphi(t):=\mathcal B(\theta+tu)$, self-concordance~$\mathrm{(SC)}$ gives $|\varphi'''(t)|\le2\varphi''(t)^{3/2}$, thus $\left|
		\frac{d}{dt}\varphi''(t)^{-1/2}
		\right|\le1$. By integrating from $0$ to $t$ and using $\varphi''(0)=r^2$,
	\[
		\frac{r^2}{(1+tr)^2}
		\le\varphi''(t)
		\le\frac{r^2}{(1-tr)^2}.
	\]
If $\theta+u\in\mathcal D$, convexity ensures that the lower bound holds
for all $0\le t\le1$. The upper bound requires $tr<1$. In particular, if
$r<1$, then $\theta+u\in\mathcal D$ by
\cite[Theorem~2.1.1]{nesterov-nemirovskii1994interior}, and both bounds hold throughout $0\le t\le1$. Using $	D_{\mathcal B}(\theta+u,\theta)
		=
		\int_0^1(1-t)\varphi''(t)dt$, we obtain
	\[
		\omega(r)
		\le D_{\mathcal B}(\theta+u,\theta)
		\le\omega_*(r),
		\qquad
		\omega(r):=r-\log(1+r),
		\qquad
		\omega_*(r):=-r-\log(1-r).
	\]
	Here the lower bound holds whenever $\theta+u\in\mathcal D$, while
	the upper bound holds whenever $r<1$. See also \cite[Theorems~4.1.5, 4.1.7, and 4.1.8]
	{nesterov2004introductory}.

	Take $\rho_0=1/2$ and $c_0=\omega(\rho_0)$. Since $\omega$ is
	increasing, the assumptions in part~\textup{(i)} imply
	$r\le\rho_0$. Moreover,
	\[
		\omega(r)
		=\int_0^r\frac{t}{1+t}\,dt
		\ge\frac{r^2}{2(1+\rho_0)}
		=\frac{r^2}{3},
	\]
	which proves part~\textup{(i)} with $c=1/3$. For part~\textup{(ii)}, $r\le\rho_0<1$ gives
	$\theta+u\in\mathcal D$. Since
	$\mathrm D^2\mathcal B(\theta)[u,u]=r^2$, the estimates
	\[
		0\le\frac{r^2}{2}-\omega(r)\le\frac{r^3}{3},
		\qquad
		0\le\omega_*(r)-\frac{r^2}{2}
		\le\frac{r^3}{3(1-r)}
		\le\frac{2r^3}{3}
	\]
	give the part~\textup{(ii)} with $C=2/3$.
\end{proof}
Using Lemma~\ref{lem:self-concordance-admissibility-consequences} and $\fone$, we show that $R_i(u)$ is sufficiently small evaluated at the relevant minimizers, so the quadratic approximation~\eqref{eq:soft-one-row-update} holds with the stated error. The
details are given in Section~\ref{subsec:row-replacement-proof}.

	It remains to compare $r_i(Z_i)^2/[2(n+\ell_i(Z_i))]$ for $Z_i=W_i$ and $Z_i=G_i$. The following proposition, proved in Section~\ref{subsec:row-replacement-proof}, controls the conditional fluctuations of
$\ell_i(Z_i)$ around its mean, and uses only the condition $\ftwo$ among Definition~\ref{def:regular-feature-row}.

	\begin{proposition}
		\label{prop:one-row-covariance-comparison}
			Assume the hypotheses of Theorem~\ref{thm:positive-barrier-row-universality}. For each$i\in[n]$, define
		\[
			\bar\ell_i
			:=\E_i\ell_i(W_i)
			=\E_i\ell_i(G_i),
		\]
		where the equality follows from covariance matching $\Cov(W_i)=\Cov(G_i)$.  Then
		$\bar\ell_i$ is nonnegative, and
		\[
			\max_{Z_i\in\{W_i,G_i\}}
			\E_i|\ell_i(Z_i)-\bar\ell_i|^2
			\le
			C\beta^{-2}\dim(\mathsf V)(\tau/\lambda_{\Sigma}+1)^2.
		\]
		Covariance matching gives $	\E_i r_i(W_i)^2
			=\E_i r_i(G_i)^2,$ and hence
		\[
			\E_i\frac{r_i(W_i)^2}{2(n+\bar\ell_i)}
			=
			\E_i\frac{r_i(G_i)^2}{2(n+\bar\ell_i)}.
		\]
	\end{proposition}

\begin{proof}[Proof of
	Theorem~\ref{thm:positive-barrier-row-universality}]
	Since $\Phi_i=m^{(i)}+\Delta_i(W_i)$ and $\Phi_{i-1}=m^{(i)}+\Delta_i(G_i)$, we have by \eqref{eq:soft-one-row-update} in Proposition~\ref{prop:quadratic-one-row-response} that
	\begin{align*}
	\left|\E_i(\Phi_i-\Phi_{i-1})\right|
		&=
		\left|
		\E_i\bigl[\Delta_i(W_i)-\Delta_i(G_i)\bigr]
		\right|\\
        &\le
	\left|
		\E_i\left[
		\frac{r_i(W_i)^2}{2(n+\ell_i(W_i))}
		-
		\frac{r_i(G_i)^2}{2(n+\ell_i(G_i))}
		\right]\right|+
		C\gamma^{-1/2}n^{-3/2}
		\bigl(1+(S^{(i)})^2\bigr).
	\end{align*}
By Proposition~\ref{prop:one-row-covariance-comparison}, we have $\E_i\frac{r_i(W_i)^2}{n+\bar\ell_i}
	=
	\E_i\frac{r_i(G_i)^2}{n+\bar\ell_i}$. Hence, adding and subtracting these terms and using
$\ell_i(Z_i),\bar\ell_i\ge0$, we obtain
\[
\begin{split}
	\left|
	\E_i\left[
	\frac{r_i(W_i)^2}{2(n+\ell_i(W_i))}
	-\frac{r_i(G_i)^2}{2(n+\ell_i(G_i))}
	\right]\right|
    &\le
	\frac{1}{2n^2}
	\sum_{Z_i\in\{W_i,G_i\}}
	\E_i\Big[
	r_i(Z_i)^2
	\big|\ell_i(Z_i)-\bar\ell_i\big|
	\Big]\\
    &\le \frac{1}{2n^2}
	\sum_{Z_i\in\{W_i,G_i\}}
	\E_i\Big[
	r_i(Z_i)^4\Big]^{1/2}
	\E_i\Big[\big(\ell_i(Z_i)-\bar\ell_i\big)^2
	\Big]^{1/2},
\end{split}
\]
where we used conditional Cauchy Schwartz in the last step. Since $\E_i[r_i(Z_i)^4]\leq C(S^{(i)})^2$~\eqref{eq:integrated-residual-moments},  combining with Proposition~\ref{prop:one-row-covariance-comparison} yields
	\[
		\left|\E_i(\Phi_i-\Phi_{i-1})\right|
		\le
		C\gamma^{-1/2}n^{-3/2}
		\bigl(1+(S^{(i)})^2\bigr)
		+
		C\beta^{-1}n^{-2}
		\sqrt{\dim\mathsf V}
		\left(\frac{\tau}{\lambda_\Sigma}+1\right)S^{(i)}.
	\]
    Therefore, taking expectations, and using $\E (S^{(i)})^2\leq C\Lambda_{\rm ref}^2$~\eqref{eq:integrated-leave-one-out-scale} with $\Lambda_{\rm ref}\ge1+\beta^{-1}$, it follows that
	\[
		\left|\E\Phi_i-\E\Phi_{i-1}\right|
		\le
		C\Lambda_{\rm ref}^2
		\left[
		\gamma^{-1/2}n^{-3/2}
		+n^{-2}
		\sqrt{\dim\mathsf V}
		\left(\frac{\tau}{\lambda_\Sigma}+1\right)
		\right].
	\]
Recalling from~\eqref{eq:lindeberg-insertion-telescoping} that $\E\Phi_W(\mathcal D,\mathcal B;\beta,\gamma)
	-\E\Phi_G(\mathcal D,\mathcal B;\beta,\gamma)=\sum_{i=1}^{n}\E[\Phi_i-\Phi_{i-1}]$, summing over $i\in [n]$ proves
	\eqref{eq:feature-universality-expectation}.

	It remains to prove the variance bound. Fix $Z\in\{W,G\}$ and write
$\Phi_Z=\Phi_Z(\mathcal D,\mathcal B;\beta,\gamma)$. Let $Z_i'$ be a fresh independent copy of $Z_i$, and let $\Phi_Z^{(i)}$ denote the optimized
value after replacing $Z_i$ by $Z_i'$. The Efron--Stein inequality gives
\[
	\Var(\Phi_Z)
	\le
	\frac12\sum_{i=1}^n
	\E\bigg[\bigl|\Phi_Z-\Phi_Z^{(i)}\bigr|^2\bigg].
\]
For this variance argument, we reuse
$\mathcal F^{(i)},\E_i,\Delta_i,r_i$, and $S^{(i)}$ for the analogous
leave-one-out quantities formed from $Z$; in
particular, $\mathcal F^{(i)}=\sigma(Z_j:j\ne i)$. The preceding
leave-one-out estimates apply unchanged. Conditional on
$\mathcal F^{(i)}$, the original and resampled problems therefore have the
same leave-one-out objective, and hence $\Phi_Z-\Phi_Z^{(i)}=
	\Delta_i(Z_i)-\Delta_i(Z_i')$.  Moreover, \eqref{eq:elementary-insertion-bound} and
\eqref{eq:integrated-residual-moments} remain valid for these quantities. Therefore,
\begin{align*}
	\E_i\Big[\bigl|\Phi_Z-\Phi_Z^{(i)}\bigr|^2\Big]
	=
	\E_i\Big[\bigl|\Delta_i(Z_i)-\Delta_i(Z_i')\bigr|^2\Big]\le
	\frac{1}{4n^2}
	\E_i\bigl(r_i(Z_i)^2+r_i(Z_i')^2\bigr)^2\leq 
	\frac{C}{n^2}(S^{(i)})^2.
\end{align*}
Taking expectations and using
\eqref{eq:integrated-leave-one-out-scale} gives \[
\E\bigl|\Phi_Z-\Phi_Z^{(i)}\bigr|^2
	\le
	C\Lambda_{\rm ref}^2n^{-2}.
    \]
    Hence, summing over $i\in [n]$ proves \eqref{eq:feature-universality-variance}. Finally,
\eqref{eq:feature-universality-expectation},
\eqref{eq:feature-universality-variance}, and Chebyshev's inequality give~\eqref{eq:feature-universality-high-probability}.
\end{proof}

  	\section{Proof of the nuclear norm lower bound}

	\label{sec:novel-sat-interpolation-details}
	
	This section proves the nuclear norm lower bound in
	Theorem~\ref{thm:centered-compatibility}. One of the main ingredient is the vector small-ball probability estimate in Proposition~\ref{prop:vector-smallball}, which is proved in Section~\ref{sec:fixed-vector-smallball}.

	\subsection{A priori estimates}\label{subsec:apriori:estimates}
	We begin with a set of probabilistic estimates on $\mathcal{L}_X$ and $\mathcal{L}_X^*$ that we call \textit{a priori estimates}. To state the result, for any subset $\mathcal I\subseteq[n]$ and $ (W_i)_{i\leq n}=((x_ix_i^\top -I_d)/\sqrt{d})_{i\leq n}$, let
	\[
		Q_{\mathcal I}:=\frac1d
		\bigl(\langle W_i,W_j\rangle\bigr)_{i,j\in \mathcal I}\in \R^{|\mathcal I|\times |\mathcal I|}\,, \qquad  X_{\mathcal I}:=[\,x_i\,]_{i\in\mathcal I}
\in\mathbb R^{d\times|\mathcal I|}.
	\]
    That is, $X_{\mathcal I}$ has columns $(x_i)_{i\in \mathcal I}$.
	\begin{proposition}[A priori estimates]\label{thm:apriori-estimates}
		Let $x_1,\ldots,x_n\in\R^d$ be i.i.d. random vectors with
		independent coordinates satisfying
		\eqref{eq:independent:coordinate:distribution:assumption} with $K_x>0$.
		Assume
		$n\le C_0d^2$ for a fixed $C_0<\infty$.  For every fixed
		$C_0'<\infty$, there is $C=C(C_0,C_0',K_x)<\infty$ such that, with
		probability tending to one, the following estimates hold simultaneously:
        \begin{equation}\label{eq:first:thm:apriori-estimates}
            	\sup_{|\mathcal I|\le C_0'd\log d}
				\|Q_{\mathcal I}-I_{|\mathcal I|}\|_{\op}
				\le C\frac{(\log d)^{3/2}}{\sqrt d},\qquad  	\sup_{|\mathcal I|\le C_0'd\log d}
				\|X_{\mathcal I}\|_{\op}\le C\sqrt d\log d,
        \end{equation}
        and
		\begin{equation}\label{eq:second:thm:apriori-estimates}
				\sup_{\|R\|_F\le1}\|\mathcal L_X R\|_2\le C\sqrt d,
				\qquad
				\sup_{\|z\|_2\le1}\|\mathcal L_X^*z\|_*\le Cd.
		\end{equation}
	\end{proposition}

	To prove this proposition, we use the following Hanson--Wright inequality.
	
	\begin{lemma}[Hanson--Wright inequality]
		\label{lem:hanson-wright}
		Let $z=(z_1,\ldots,z_d)\in\R^d$ be a random vector with independent mean-zero subgaussian coordinates satisfying
		$\max_{j\le d}\|z_j\|_{\psitwo}\le K_z$.  There are constants
		$c,C>0$, depending only on $K_z$, such that for every
		$A\in\Ssym^d$ and every $u>0$,
		\[
		\P\left(|z^\top Az-\Tr A|>u\right)
		\le
		2\exp\left(
		-c\min\left\{
		\frac{u^2}{\|A\|_F^2},
		\frac{u}{\|A\|_{\op}}
		\right\}\right),
		\]
		and $\|z^\top Az-\Tr A\|_{\psione}\le C\|A\|_F$. Thus, for $W=(zz^\top-I_d)/\sqrt{d}$, we have
		\[
		\|\langle W,A\rangle\|_{\psione}
		\le \frac{C}{\sqrt d}\|A\|_F.
		\]
	\end{lemma}
	\begin{proof}
		The first estimate follows from the Hanson--Wright inequality (see~\cite[Theorem 6.2.1]{vershynin2018highdimensional}).
		Since $\|A\|_{\op}\leq \|A\|_F$, this estimate and the standard
		characterization of subexponential random variables~\cite[Proposition~2.7.1]{vershynin2018highdimensional} imply
		$\|z^\top A z-\Tr A\|_{\psione}\leq C\|A\|_F$. Finally,
		$\langle W,A\rangle_F=d^{-1/2}(z^\top Az-\Tr A)$, so the last
		conclusion follows.
	\end{proof}
	
	\begin{proof}[Proof of Proposition~\ref{thm:apriori-estimates}]
	Throughout, we let $D=d(d+1)/2$ denote the dimension of $\Ssym^d$, and let $c,C>0$ denote constants that only depend on $C_0, C_0'$, and $K_x$, which may change from line to line. We may assume without loss of generality that $C_0\geq 1$ and $n=\lfloor C_0d^2\rfloor$, since each estimate remains valid after decreasing the number of samples $n$. Let
		$$
		\mathsf W:=[w_1,\ldots,w_n]\in\R^{D\times n}\quad\textnormal{where}\quad w_i:=\sqrt{\frac d2}\,\operatornorm{vec}(W_i).
		$$
	        Lemma~\ref{lem:hanson-wright} shows that $\sup_{i\leq n}\sup_{u\in \R^D: \|u\|_2=1}
	        \|\langle w_i,u\rangle\|_{\psi_1}\leq C$. Thus, we may apply the restricted-isometry estimate of
	        \cite[Theorem~3.2]{adamczak-litvak-pajor-tomczak2009rip} for matrices with subexponential columns.
	        Specializing that theorem to the columns $w_i$ and taking
	        $r=1$, $K=1$, and $K_0=2$ in the notation therein gives the following
	        estimate. For any integer $1\leq m\leq \min(n,D)$, let
		$$
		\delta_m:=
		\sup_{|\mathcal I|\le m}
		\left\|
		\frac1{D}\mathsf W_{\mathcal I}^{\top}\mathsf W_{\mathcal I}
		-I_{|\mathcal I|}
		\right\|_{\op},
		$$
	        where $\mathsf W_{\mathcal I}$ denotes the submatrix of $\mathsf W$ with columns $(w_i)_{i\in \mathcal I}$. Then, for any $\theta\in (0,1)$,
		\begin{equation}\label{eq:alptj-specialized}
		\P\left(
		\delta_m\geq
		C\sqrt{\frac m{D}}
		\log\left(\frac{en}{m\sqrt{\frac{m}{D}}}\right)+\theta
		\right) \le
		\exp\left(
		-c\sqrt m\log\left(\frac{en}{m\sqrt{\frac{m}{D}}}\right)
		\right)
		+2\P\left(
		\max_{i\le n}
		\left|\frac{\|w_i\|_2^2}{D}-1\right|
		\ge\theta
		\right).
		\end{equation}
		Using this estimate, we first prove the desired estimate on $Q_{\mathcal I}$. Note that 
        \[
        \frac{\|w_i\|_2^2}{D}= \frac{\|W_i\|_F^2}{d+1}=\frac{\Tr((x_ix_i^\top-I_d)^2)}{d(d+1)}=\frac{\|x_i\|_2^4-2\|x_i\|_2^2+d}{d(d+1)}.
        \]
        By standard norm concentration of subgaussian vectors~\cite[Theorem 3.1.1]{vershynin2018highdimensional}, we have 
        \[
        \P\left(\left|\frac{\|x_i\|}{\sqrt{d}}-1\right|\geq\theta\right)\leq 2\exp(-cd\theta^2).
        \]
	        Set $t_d=(\log d)/\sqrt d$. The preceding identity shows that, whenever
	        $|\|x_i\|_2/\sqrt d-1|\leq t_d$,
	        \[
	        \left|\frac{\|w_i\|_2^2}{D}-1\right|
	        \leq C\left(t_d+\frac1d\right).
	        \]
	        Consequently, with $\theta_d:=C(t_d+d^{-1})$, a union bound gives
	        \begin{equation}\label{eq:row:norm:estimate}
	        \P\left(
		\max_{i\le n}
		\left|\frac{\|w_i\|_2^2}{D}-1\right|
		\ge\theta_d
		\right)\leq 2n\exp(-cdt_d^2)=o(1),
	        \end{equation}
	        where we used $n\leq C_0d^2$ in the last step. Thus, taking
	        $m=\lfloor C_0'd\log d\rfloor$ and $\theta=t_d$ in
	        \eqref{eq:alptj-specialized}, and noting that
	        \[
	        \frac{1}{D}\mathsf{W}_{\mathcal{I}}^{\top}\mathsf{W}_{\mathcal{I}}=\frac{d}{d+1}Q_{\mathcal{I}},
	        \qquad
	        \sqrt{\frac mD}\log\left(\frac{en}{m\sqrt{m/D}}\right)
	        \leq C\frac{(\log d)^{3/2}}{\sqrt d},
	        \]
	         the first estimate in \eqref{eq:first:thm:apriori-estimates} directly follows from~\eqref{eq:alptj-specialized} and \eqref{eq:row:norm:estimate}.

       For the second estimate in \eqref{eq:first:thm:apriori-estimates}, it follows directly from standard operator norm bounds of matrices with subgaussian entries. Indeed,~\cite[Theorem 4.4.5]{vershynin2018highdimensional} shows that for a fixed $\mathcal I$ with $|\mathcal I|\leq C_0' d \log d$, we have for each fixed $C_1\geq 1$ and large enough $d$ that
       \[
       \P\left(\left\|X_{\mathcal I}\right\|_{\op}\geq C_1\sqrt{d}\log d\right)\leq 2\exp\left(-c C_1^2 d(\log d)^2\right).
       \]
	       The number of subsets $\mathcal I\subset[n]$ with $|\mathcal I|\leq m$ satisfies $\sum_{k=0}^m\binom nk
	       \leq\exp\bigl(Cd(\log d)^2\bigr)$. Thus, taking $C_1$ large enough and a union bound concludes the proof of the second estimate in \eqref{eq:first:thm:apriori-estimates}.

      We next prove the estimate~\eqref{eq:second:thm:apriori-estimates}. First, we claim that $\|\mathsf W\|_{\op}\leq Cd$ with probability $1-o(1)$. Indeed, in \eqref{eq:alptj-specialized}, take $m=D\leq n$ (since we assumed w.l.o.g. $n\geq d^2$) and
       $\theta=1/2$. Then, since $ \sqrt D\log\left(\frac{en}{D}\right)\geq c d$, it follows from \eqref{eq:alptj-specialized} and
       \eqref{eq:row:norm:estimate} that
	       \begin{equation*}
       \sup_{|\mathcal I|\leq D}\|\mathsf W_{\mathcal I}\|_{\op}
       \leq C\sqrt D.
       \end{equation*}
        We use this estimate to control the full matrix $\mathsf W$.
       Partition $[n]$ into disjoint sets
       $\mathcal I_1,\ldots,\mathcal I_q$, each of cardinality at most $D$.
       Since $n= \lfloor C_0d^2\rfloor$ and $D=d(d+1)/2$, we have $q\leq C$. Thus, with probability $1-o(1)$,
       \begin{equation*}
       \|\mathsf W\|_{\op}\leq \sum_{j=1}^{q}\|\mathsf{W}_{\mathcal I_j}\|_{\op}\leq C \sqrt{D}.
       \end{equation*}
       Now, by the definition of $\mathsf W$, for every $R\in\Ssym^d$, $ \mathsf W^\top\operatornorm{vec}(R)
       =\sqrt{\frac d2}\,\mathcal L_XR$ holds. Since $\operatornorm{vec}(\cdot)$ preserves the Frobenius norm, the preceding bound gives that with probability $1-o(1)$,
       \[
       \|\mathcal L_XR\|_2
       \leq \sqrt{\frac2d}\,\|\mathsf W\|_{\op}\|R\|_F
       \leq C\sqrt d\,\|R\|_F,
       \]
       proving the first estimate in
       \eqref{eq:second:thm:apriori-estimates}. Finally, combining this with duality, we have for every $z\in\R^n$
       \begin{align*}
       \|\mathcal L_X^*z\|_F=\sup_{\|R\|_F\leq1}
       |\langle \mathcal L_X^*z,R\rangle_F|=\sup_{\|R\|_F\leq1}
       |\langle z,\mathcal L_XR\rangle|
       \leq C\sqrt d\,\|z\|_2.
       \end{align*}
       Since every matrix in $\Ssym^d$ has rank at most $d$, we also have
       $\|A\|_*\leq\sqrt d\,\|A\|_F$ for every $A\in\Ssym^d$. Thus
       \[
       \|\mathcal L_X^*z\|_*
       \leq\sqrt d\|\mathcal L_X^*z\|_F
       \leq Cd\|z\|_2,
       \]
       which proves the second estimate in
       \eqref{eq:second:thm:apriori-estimates}, and concludes the proof.
	\end{proof}

      Using these a priori estimates from Proposition~\ref{thm:apriori-estimates}, we first prove the bound in Theorem~\ref{thm:centered-compatibility} when $\lambda$ is \textit{sparse}, i.e. it is supported on $C_0 d\log d$ coordinates, and has no tail.
    	\begin{proposition}\label{prop:sparse-restricted-lower}
		Let $x_1,\ldots,x_n\in\R^d$ be i.i.d. random vectors with
		independent coordinates satisfying
		\eqref{eq:independent:coordinate:distribution:assumption} with $K_x>0$.
		Suppose
		$n/d^2\to\alpha\in(0,\infty)$.
		For every fixed $C_0<\infty$, there is $c=c(C_0)>0$ such that, with probability
		tending to one,
		\begin{equation}\label{eq:sparse-L-lower}
		\|\mathcal{L}_X^*\lambda\|_*\ge \frac{c\sqrt d}{(\log d)^4}\|\lambda\|_1\quad\textnormal{for all}\quad \lambda\in\R^n\quad\textnormal{s.t.}\quad |\supp \lambda|\le C_0d\log d.
		\end{equation}
	\end{proposition}

\begin{proof}
	We work on the high-probability event in
	Proposition~\ref{thm:apriori-estimates}. In particular,
	\begin{equation}\label{eq:gram:estimate}
		\left\|Q_{\mathcal I}-I_{|\mathcal I|}\right\|_{\op}
		\le C\frac{(\log d)^{3/2}}{\sqrt d}
		\quad\text{for all }|\mathcal I|\le C_0d\log d.
	\end{equation}
	Fix a nonzero $\lambda\in\R^n$ satisfying
	$|\supp\lambda|\le C_0d\log d$, and set
	$\mathcal I:=\supp\lambda$. By nuclear/operator norm duality, it
	suffices to construct a matrix $A\in\Ssym^d$ such that
	$$
	\langle A,W_i\rangle=\operatorname{sgn}(\lambda_i)
	\quad\text{for every }i\in\mathcal I,
	\qquad
	\|A\|_{\op}\le \frac{C(\log d)^4}{\sqrt d}.
	$$
	Indeed, such a certificate satisfies $\left\langle A,\mathcal L_X^*\lambda\right\rangle
	=
	\left\langle A,\sum_{i\in\mathcal I}\lambda_iW_i\right\rangle
	=\|\lambda\|_1,$ which gives the desired lower bound by nuclear/operator norm duality. We construct such $A$ in the span of $(W_i)_{i\in\mathcal I}$. Let
	$$
	\varepsilon
	:=\bigl(\operatorname{sgn}(\lambda_i)\bigr)_{i\in\mathcal I}
	\in\R^{\mathcal I},
	\qquad
	y:=Q_{\mathcal I}^{-1}\varepsilon\in\R^{\mathcal I},
	\qquad
	A:=\frac1d\sum_{i\in\mathcal I}y_iW_i.
	$$
	Then, for every $j\in\mathcal I$,
	$$
	\langle A,W_j\rangle
	=\frac1d\sum_{i\in\mathcal I}y_i\langle W_i,W_j\rangle
	=(Q_{\mathcal I}y)_j
	=\varepsilon_j,
	$$
	so $A$ has the required interpolation property.

	It remains to bound $\|A\|_{\op}$. We first show that
	$\|y\|_2\le C\sqrt{d\log d}$ and
	$\|y\|_\infty\le C(\log d)^2$. Indeed,
	\eqref{eq:gram:estimate} gives
	$\|Q_{\mathcal I}^{-1}\|_{\op}\le2$ for all sufficiently large $d$, and
	therefore
	$$
	\|y\|_2
	\le\|Q_{\mathcal I}^{-1}\|_{\op}\|\varepsilon\|_2
	\le2\sqrt{|\mathcal I|}
	\le C\sqrt{d\log d}.
	$$
	For the $\ell_\infty$ estimate, \eqref{eq:gram:estimate} gives, for
	every $i\in\mathcal I$,
	$$
	\Bigg(\sum_{j\in\mathcal I, j\ne i}
	|(Q_{\mathcal I})_{ij}|^2\Bigg)^{1/2}
	\le\|(Q_{\mathcal I}-I)e_i\|_2
	\le C\frac{(\log d)^{3/2}}{\sqrt d}.
	$$
	Moreover, $(Q_{\mathcal I})_{ii}y_i
	=\varepsilon_i-
	\sum_{j\in\mathcal I, j\ne i}
	(Q_{\mathcal I})_{ij}y_j,$ and $(Q_{\mathcal I})_{ii}\ge1/2$ for all sufficiently large $d$.
	Thus, for every $i\in\mathcal I$,
	$$
	|y_i|
	\le2\Bigg(
		1+
		\Bigg(\sum_{j\in\mathcal I, j\ne i}
		|(Q_{\mathcal I})_{ij}|^2\Bigg)^{1/2}\|y\|_2
		\Bigg)
	\le C(\log d)^2,
	$$
thus $\|y\|_{\infty}\leq C (\log d)^2$. Finally, writing $D_y=\operatorname{diag}(y_i:i\in\mathcal I)$, we have
	$$
	A=\frac1{d^{3/2}}\left[
	X_{\mathcal I}D_yX_{\mathcal I}^\top
	-\left(\sum_{i\in\mathcal I}y_i\right)I_d
	\right].
	$$
	Since $\|X_{\mathcal I}\|_{\op}\le C\sqrt d\log d$ on the event from
	Proposition~\ref{thm:apriori-estimates},
	$$
	\|A\|_{\op}
	\le\frac1{d^{3/2}}\left(
	\|X_{\mathcal I}\|_{\op}^2\|D_y\|_{\op}
	+\sqrt{|\mathcal I|}\|y\|_2
	\right)
	\le\frac{C(\log d)^4}{\sqrt d}.
	$$
	Therefore, by nuclear/operator norm duality,
	$$
	\|\mathcal L_X^*\lambda\|_*
	\ge
	\frac{|\langle A,\mathcal L_X^*\lambda\rangle|}{\|A\|_{\op}}
	\ge
	\frac{c\sqrt d}{(\log d)^4}\|\lambda\|_1,
	$$
	which concludes the proof.
\end{proof}

	\subsection{Proof of Theorem~\ref{thm:centered-compatibility}}
   Having established Proposition~\ref{prop:sparse-restricted-lower}, it remains
to control the light tail after allowing a sparse correction on its
complement. Recall that a vector $t\in\R^n$ is called $K$-light if
$$
\|t\|_\infty\le (Kd)^{-1/2}\|t\|_2.
$$
For a fixed set $\mathcal I$ and a tail $t$ supported on $\mathcal I^c$, the
quantity $\inf_{\supp h\subset\mathcal I}
\|\mathcal L_X^*(h+t)\|_*$ is the nuclear norm distance from $\mathcal L_X^*t$ to the subspace spanned
by $(W_i)_{i\in\mathcal I}$. Equivalently, it is the \textit{quotient norm} of
$\mathcal L_X^*t$ modulo that subspace. The proposition below proves the
bounded-shift version needed in our argument, where the sparse correction
also satisfies $\|h\|_1\le d^{2/3}\|t\|_2$.

	\begin{proposition}[Light-tail quotient norm lower bound]\label{prop:light-tail-quotient}
		Let $x_1,\ldots,x_n\in\R^d$ be i.i.d. random vectors with
		independent coordinates satisfying
		\eqref{eq:independent:coordinate:distribution:assumption} with $K_x>0$.
		Suppose $n/d^2\to \alpha\in (0,1/2)$ as $n\to\infty$. There exist $K>0$ and $\gamma>0$, depending only on
		$\alpha$ and $K_x$, such that for every fixed $C_0>0$, with
		probability $1-o(1)$, the following holds.  For every
		$\mathcal I\subset[n]$ with $|\mathcal I|\le C_0d\log d$ and every $K$-light $t$
		supported on $\mathcal I^c$,
		\[
		\inf_{\substack{\supp h\subset \mathcal I\\ \|h\|_1\le d^{2/3}\|t\|_2}}
		\|\mathcal{L}_X^*(h+t)\|_*
		\ge \gamma d\|t\|_2 .
		\]
	\end{proposition}
    Together with Propositions~\ref{thm:apriori-estimates} and~\ref{prop:sparse-restricted-lower}, this quotient norm lower bound implies Theorem~\ref{thm:centered-compatibility} by a shelling
argument (see e.g.~\cite{candes-tao2007dantzig},~\cite{bickel-ritov-tsybakov2009lasso},~\cite{vandegeer-buhlmann2009conditions}). We first give this deduction and then devote the remainder of the
section to proving the quotient norm lower bound, Proposition~\ref{prop:light-tail-quotient}.

	
\begin{proof}[Proof of Theorem~\ref{thm:centered-compatibility}]
		Throughout, let $K$ and $\gamma$ be the constants from
		Proposition~\ref{prop:light-tail-quotient}, and denote $c,C>0$ by constants that only depends on $\alpha$ and $K_x$. Also, we set
		$$
		a_d:=\frac{\sqrt d}{(\log d)^4}.
		$$
		We work deterministically the event holding with probability $1-o(1)$, and that all of the conclusions in Propositions~\ref{thm:apriori-estimates},~\ref{prop:sparse-restricted-lower}, and~\ref{prop:light-tail-quotient} hold simultaneously. In particular, for some constants $c,C>0$, 
		$$
		\|\mathcal L_X^*h\|_*\ge ca_d\|h\|_1
		\quad\text{if }|\supp h|\le5Kd\log d,
		\qquad
		\|\mathcal L_X^*z\|_*\le Cd\|z\|_2
		\quad\text{for every }z\in\R^n.
		$$
		To this end, fix a vector $\lambda\in\R^n, \lambda \neq 0$ and order its coordinates by
		non-increasing absolute value.  Starting with $\lambda$, remove the largest
		remaining coordinate until either the remaining vector is $K$-light or
		$k_*:=\lceil4Kd\log d\rceil$ coordinates have been removed.  Whenever the
		current tail is not $K$-light, removing its largest coordinate decreases
		its squared $\ell_2$-norm by at least a $1/(Kd)$ fraction.  Hence, if the
		procedure reaches $k_*$ steps, the remaining tail satisfies
		\begin{equation}\label{eq:shelling:stop}
		\|t\|_2^2
		\le\left(1-\frac1{Kd}\right)^{k_*}\|\lambda\|_2^2
		\le d^{-4}\|\lambda\|_2^2.
		\end{equation}
		Let $\mathcal I$ be the set of removed coordinates and write
		$h=\lambda_{\mathcal I}$ and $t=\lambda_{\mathcal I^c}$.  Then
		$\mathcal I$ consists of the largest coordinates of $\lambda$,
		$|\mathcal I|\le5Kd\log d$ for all large $d$, and either $t$ is
		$K$-light or $\|t\|_2\le d^{-2}\|\lambda\|_2$.

		Note that by linearity of $\mathcal{L}_X^*$ and a triangle inequality,
		\begin{equation}\label{eq:compatibility-master}
		\|\mathcal L_X^*\lambda\|_*\geq \|\mathcal L_X^* h \|_*-\|\mathcal L_X^* t\|_*
		\geq ca_d\|h\|_1-Cd\|t\|_2.
		\end{equation}
		Suppose first that $t$ is $K$-light and
		$\|h\|_1\le d^{2/3}\|t\|_2$.  Proposition~\ref{prop:light-tail-quotient}
		also gives $\|\mathcal L_X^*\lambda\|_*\ge\gamma d\|t\|_2$.  Therefore
		$$
		d\|t\|_2\le\gamma^{-1}\|\mathcal L_X^*\lambda\|_*,
		\qquad
		a_d\|h\|_1\le C\|\mathcal L_X^*\lambda\|_*,
		$$
		where the second inequality follows from \eqref{eq:compatibility-master}.
		This proves the desired bound in this case.

			In all remaining cases, either $t$ is $K$-light and
	$\|h\|_1>d^{2/3}\|t\|_2$, or $\|t|_2\leq d^{-2}\|\lambda\|_2$ holds by \eqref{eq:shelling:stop}. In the
	first case, 
    \[
    d\|t\|_2<d^{1/3}\|h\|_1=o(a_d\|h\|_1).
    \]
    In the second case, since
	$\lambda=h+t$, we have
	$\|t\|_2\le2d^{-2}\|h\|_2\le2d^{-2}\|h\|_1$, and hence again $d\|t\|_2=o(a_d\|h\|_1)$. Thus, in either case, the negative tail term in
	\eqref{eq:compatibility-master} can be absorbed into the sparse lower
	bound. After decreasing the constant if necessary, we obtain 
    \[
    \|\mathcal L_X^*\lambda\|_*
	\ge c(a_d\|h\|_1+d\|t\|_2).
    \]
    This concludes the proof.
	\end{proof}

		\subsection{Proof of the Proposition~\ref{prop:vector-smallball}}
	\label{sec:fixed-vector-smallball}
 This section proves Proposition~\ref{prop:vector-smallball}, which is used to prove
Proposition~\ref{prop:light-tail-quotient}. As discussed in Section~\ref{subsec:proof:overview:sat}, the proof
follows the Fourier-analytic method by~\cite{ball-nazarov-little-level} and its multidimensional extension by~\cite{rudelson-vershynin2015smallball}.

For a random variable $Y$, recall 
	$Q(Y,L):=\sup_t\P\{Y\in[t,t+L]\}$. We use the following
	Kolmogorov--Rogozin/Esseen concentration-function estimate
	\cite{rogozin1961estimate,esseen1968concentration}: if
	$Y_1,\ldots,Y_N$ are independent, $S_N=\sum_{j=1}^N Y_j$, and
	$0<\lambda_j\le L$, then
	\begin{equation}\label{eq:kolmogorov:rogozin}
	Q(S_N,L)\le C L
	\left(\sum_{j=1}^N\lambda_j^2
	\bigl(1-Q(Y_j,\lambda_j)\bigr)\right)^{-1/2},
	\end{equation}
    where $C>0$ is an absolute constant.
    
    Its relevance to Proposition~\ref{prop:vector-smallball} is already clear when $m=1$: in this case,
	$S=\sum_{i=1}^n w_i\xi_i$, with $\sum_iw_i^2\ge1$ and
	$\max_i|w_i|\le\rho$. We prove in Lemma~\ref{lem:bl-oned-factor} that the truncated second moment assumption~\eqref{eq:sec:mo:condition:prop:vector-smallball} implies $Q(\xi_i,1/8)\le1-c$ for all $i$. Once this estimate holds, applying~\eqref{eq:kolmogorov:rogozin} with
	$Y_i=w_i\xi_i$,
	$\lambda_i=|w_i|/8$, and $L=2\eta$ (which is an admissible choice of parameters since $\eta \gtrsim \max_i|w_i|$) gives
	$$
	\sup_z\P(|S-z|\le\eta)
	\le C\eta
	\left(\sum_iw_i^2\bigl(1-Q(\xi_i,1/8)\bigr)\right)^{-1/2}
	\le C\eta.
	$$
   Thus Proposition~\ref{prop:vector-smallball} follows directly from~\eqref{eq:kolmogorov:rogozin} when $m=1$. The next
lemma provides the Fourier estimate needed to extend this argument to
$m>1$.

	\begin{lemma}
		\label{lem:bl-oned-factor}
		Assume that the random variable $\xi$ satisfy $\E\xi=0, \Var(\xi)=1$, and for some $M>0$,
		$$
		\E\bigl[\xi^2\mathbf 1_{\{|\xi|\le M\}}\bigr]\ge \frac12.
		$$
		There are constants $c=c(M)>0$, $C=C(M)<\infty$, and
		$s_0=s_0(M)>0$ such that $Q(\xi,1/8)\le1-c$. Moreover, if
		$\phi_\xi(u)=\E e^{iu\xi}$, then, for every $0<r\le s\le s_0$,
		\begin{equation}\label{eq:fourier:estimate}
		\frac1r\int_\R |\phi_\xi(u)|^{1/r^2}
		\exp\left(-\frac{s^2u^2}{2r^2}\right)du\le C.
		\end{equation}
	\end{lemma}

	\begin{proof}
		We may assume w.l.o.g. $M\ge2$. Let $I$ be any closed interval of
length $1/8$, and set
$$
\delta=\P(\xi\notin I).
$$
We show that $\delta\geq c(M)>0$. First suppose that $I\subset[-1/4,1/4]$. Using
$\E[\xi^2\mathbf 1_{\{|\xi|>M\}}]\le1/2$, we obtain
$$
1=\E\xi^2
\le \E[\xi^2 \mathbf 1_{\{|\xi|\leq \frac{1}{4}\}}]+\E[\xi^2 \mathbf 1_{\{\frac{1}{4}<|\xi|\leq M\}}]+\frac{1}{2}\leq \frac1{16}+M^2\delta+\frac{1}{2}.
$$
Consequently, $\delta\ge 7/(16M^2)$. Otherwise, either
$I\subset[1/8,\infty)$ or $I\subset(-\infty,-1/8]$. Consider the first
case; the second follows after replacing $\xi$ by $-\xi$. Since $\P(\xi\in I)=1-\delta$, and $\xi \geq 1/8$ a.s. on the event $\{\xi \in I\}$. we have $\E[\xi\mathbf 1_{\{\xi\in I\}}]\geq (1-\delta)/8$. Moreover, since $\E[\xi]=0$, $\E[\xi\mathbf 1_{\{\xi\in I\}}]
=-\E[\xi\mathbf 1_{\{\xi\notin I\}}]$.
Therefore, Cauchy--Schwarz and $\E\xi^2=1$ give
$$
\frac{1-\delta}{8}
\le
\left|\E[\xi\mathbf 1_{\{\xi\notin I\}}]\right|
\le
\sqrt{\E\xi^2\,\P\{\xi\notin I\}}
=\sqrt{\delta}.
$$
Thus, $\delta \geq c_0>0$ for an absolute constant $c_0>0$ in this case. Hence,
$$
\P(\xi\notin I)\ge
c(M):=\min\left\{\frac7{16M^2},c_0\right\}.
$$
Taking the supremum over $I$ gives $Q(\xi,1/8)\le1-c(M)$.

We turn to the Fourier estimate~\eqref{eq:fourier:estimate}. Since the integrand decreases as $s$ increases, it suffices to prove the
estimate when $s=r$. Let $\xi'$ be an independent copy of $\xi$ and set
$\widetilde\xi=\xi-\xi'$. Note that such symmetrization is useful since
$\phi_{\widetilde\xi}(u)=|\phi_\xi(u)|^2$. Moreover, conditioning on $\xi'$, a
translate of any interval of length $1/8$ is again an interval of length
$1/8$, thus
$$
Q(\widetilde\xi,1/8)\le Q(\xi,1/8)\le1-c(M).
$$
Choose $s_0\le1/2$ and set
$N=\lfloor(2r^2)^{-1}\rfloor$. Then $N\ge1/(4r^2)$ and
$$
|\phi_\xi(u)|^{1/r^2}
\le|\phi_\xi(u)|^{2N}
=\phi_{\widetilde\xi}(u)^N=\E[e^{iuS_N}],
$$
where $S_N$ is the sum of $N$ independent copies of $\widetilde\xi$. Moreover, using Fubini's theorem, we have
$$
\int_\R \E[e^{iuS_N}]e^{-u^2/2}du=
\sqrt{2\pi}\, \E[e^{-S_N^2/2}],
$$
where we used the characteristic function of a standard gaussian random variable. Note that
\[
\E [e^{-S_N^2/2}]=\int_0^1
\P\{|S_N|<\sqrt{2\log(1/t)}\}dt\le
C Q(S_N,1)
\int_0^1\bigl(1+\sqrt{\log(1/t)}\bigr)dt\leq C' Q(S_N,1).
\]
Since $Q(\widetilde \xi, 1/8)\leq 1-c(M)$, applying \eqref{eq:kolmogorov:rogozin} with $L=1$ and
$\lambda_j=1/8$ yields $Q(S_N,1)\le\frac{C(M)}{\sqrt N}.$ Therefore,
\[
\int_{\R}|\phi_{\xi}(u)|^{1/r^2} e^{-u^2/2} du \leq \frac{C(M)}{\sqrt{N}}\leq C'(M)r.
\]
Dividing by $r$ concludes the proof.
	\end{proof}
\begin{remark}
	Using the trivial bound $|\phi_\xi(u)|\le1$, the left-hand side of
	\eqref{eq:fourier:estimate} is bounded by $C/s$. Thus, \eqref{eq:fourier:estimate} improves the trivial bound by a
	factor of $s$. Moreover, ~\eqref{eq:fourier:estimate} says that, via Fourier inversion, for the random variable
	$S_N$ appearing in the proof and an independent standard Gaussian
	random variable $G$, the density of $S_N+G$ is bounded by $Cr$. Indeed, the
	integral in \eqref{eq:fourier:estimate} is the one-dimensional factor
	that arises after Gaussian regularization in the proof of
	Lemma~\ref{lem:regularized-density} below.
\end{remark}
	Another ingredient to establish Proposition~\ref{prop:vector-smallball} for $m>1$ is the following Brascamp--Lieb inequality, which was also used in~\cite{rudelson-vershynin2015smallball}.

	\begin{theorem}[Brascamp--Lieb {\cite{brascamp-lieb1976}; see
	\cite{ball1989volumes}}]
		\label{thm:geometric-brascamp-lieb}
		Let $u_1,\ldots, u_n \in\R^m$ be a finite family of unit vectors, and let $c_1,\ldots, c_n>0$ be real number 
		satisfying 
		$$
		\sum_i c_i u_i u_i^\top=I_m.
		$$
		Then, for every nonnegative measurable functions
		$F_1,\ldots, F_n:\R\to[0,\infty]$,
		$$
		\int_{\R^m}\prod_i F_i(\langle u_i,\theta\rangle)^{c_i}\,d\theta
		\le
		\prod_i\left(\int_\R F_i(t)\,dt\right)^{c_i}.
		$$
	\end{theorem}
	The next lemma is the central Gaussian regularization step. 
	\begin{lemma}\label{lem:regularized-density}
		Let $\xi_1,\ldots,\xi_n$ be independent mean zero random variables with variance one, and for some $M>0$,
		$$
		\inf_{1\le i\le n}
		\E\bigl[\xi_i^2\mathbf 1_{\{|\xi_i|\le M\}}\bigr]\ge\frac12.
		$$
		There are constants $C=C(M)<\infty$ and $s_0=s_0(M)>0$ with the
		following property. Let $w_1,\ldots,w_n\in\R^m$ satisfy
		$$
		\sum_{i=1}^n w_iw_i^\top=I_m,
		\qquad
		s:=\max_{1\le i\le n}\|w_i\|_2\le s_0.
		$$
		If $G\sim N(0,I_m)$ is independent of the $\xi_i$'s, then
		$Y_s:=\sum_{i=1}^n w_i\xi_i+sG$ has a density $f_{Y_s}$ satisfying
		$\|f_{Y_s}\|_\infty\le C^m$.
	\end{lemma}

	\begin{proof}
		Take $s_0$ no larger than the constant in
		Lemma~\ref{lem:bl-oned-factor}. Discard $w_i$'s such that $w_i=0$, and set
		$r_i=\|w_i\|_2$, $u_i=w_i/r_i$, and $c_i=r_i^2$. Then
		$0<r_i\le s$ and
		$$
		\sum_i c_i u_i u_i^\top=I_m,
		\qquad
		\sum_i c_i=m,
		$$
		where the second identity follows by taking $\tr(\cdot)$. Let $\phi_i(u)=\E e^{iu \xi_i}$ and $\phi_{Y_s}(\theta)=\E e^{i\langle \theta, Y_s\rangle} $ respectively denote the
		characteristic function of $\xi_i$ and
		$Y_s$. Then,
		$$
		|\phi_{Y_s}(\theta)|
		=\prod_i F_i(\langle u_i,\theta\rangle)^{c_i},
		\qquad
		F_i(t):=|\phi_i(r_it)|^{1/c_i}e^{-s^2t^2/2},
		$$
		where we used
		$\sum_i c_i\langle u_i,\theta\rangle^2=\|\theta\|_2^2$. Thus, Theorem~\ref{thm:geometric-brascamp-lieb} yields
		$$
		\int_{\R^m}|\phi_{Y_s}(\theta)|\,d\theta
		\le\prod_i\left(\int_\R F_i(t)\,dt\right)^{c_i}.
		$$
		Since $c_i=r_i^2$, the change of variables $u=r_it$ and
		Lemma~\ref{lem:bl-oned-factor} show that
		$$
		\int_\R F_i(t)\,dt
		=\frac1{r_i}\int_\R |\phi_i(u)|^{1/r_i^2}
		\exp\left(-\frac{s^2u^2}{2r_i^2}\right)du
		\le C.
		$$
		Thus $\int_{\R^m}|\phi_{Y_s}(\theta)|\,d\theta
		\le\prod_i C^{c_i}=C^m$. Fourier inversion, with its normalization
		absorbed into $C^m$, proves $\|f_{Y_s}\|_\infty\le C^m$.
	\end{proof}

	\begin{proof}[Proof of Proposition~\ref{prop:vector-smallball}]
		Let $Y=\sum_{i=1}^n w_i\xi_i$ and
		$\Sigma=\sum_{i=1}^n w_iw_i^\top$. Since $\Sigma\succeq I_m$,
		$\Sigma^{-1/2}$ is a contraction. Define
		$\widetilde w_i=\Sigma^{-1/2}w_i$,
		$\widetilde Y=\Sigma^{-1/2}Y$, and
		$\widetilde z=\Sigma^{-1/2}z$. Then
		$$
		\sum_{i=1}^n\widetilde w_i\widetilde w_i^\top=I_m,
		\qquad
		\max_{1\le i\le n}\|\widetilde w_i\|_2\le\rho,
		$$
		and, for every $z\in\R^m$,
		$$
		\{\|Y-z\|_2\le\eta\sqrt m\}
		\subseteq
		\{\|\widetilde Y-\widetilde z\|_2\le\eta\sqrt m\}.
		$$
		It therefore suffices to prove the result under the normalization
		$\sum_iw_iw_i^\top=I_m$. For simplicity, we relabel $\widetilde w_i, \widetilde Y, \widetilde z$ as $w_i,Y, z$.

		Let $s=\max_i\|w_i\|_2$. We may assum w.l.o.g. $C\eta<1$ since the result is trivial when $C\eta \geq 1$. Since
		$\eta\ge C\rho$ and $s\le\rho$, choosing $C$ sufficiently large
		ensures $s\le s_0$, where $s_0$ is the constant in
		Lemma~\ref{lem:regularized-density}. Thus, if we let $G\sim N(0,I_m)$ be
		independent and set $Y_s=Y+sG$, Lemma~\ref{lem:regularized-density} yields that $Y_s$ has density
		$\|f_{Y_s}\|_\infty\le C_1^m$ for some $C_1=C_1(M)>0$. Combining with the standard volume estimate for the Euclidean ball
        $\operatorname{Vol}(B_2^m(\sqrt m))\le C_2^m$, it follows that for any $z\in \R^m$,
		$$
		\P\left(\|Y_s-z\|_2\le2\eta \sqrt m\right)
		\le C_1^m\operatorname{Vol}\left(B_2^m\left(2\eta\sqrt m\right)\right)
		\le(C\eta)^m.
		$$ 
		It remains to remove the Gaussian noise. Since $\E\|G\|_2^2=m$ and
		$\eta\ge Cs$, Markov's inequality gives
		$$
		\P\left(s\|G\|_2\le\eta\sqrt m\right)
		\ge1-\frac{s^2}{\eta^2}\ge\frac12.
		$$
		This event is independent of $Y$, and if
		$\|Y-z\|_2\le\eta\sqrt m$ and
		$s\|G\|_2\le\eta\sqrt m$, then
		$\|Y_s-z\|_2\le2\eta\sqrt m$. Therefore
		$$
		\P\left(\|Y-z\|_2\le\eta\sqrt m\right)
		\le2\P\left(\|Y_s-z\|_2\le2\eta\sqrt m\right)\leq (2C\eta)^m,
		$$
		which concludes the proof.
	\end{proof}

	\subsection{Deterministic light-tail argument}
    \label{subsec:deterministic}
This subsection and the next prove Proposition~\ref{prop:light-tail-quotient}. We begin with a simpler deterministic tail version, in which the $K$-light tail $t$ is fixed rather than uniform over all such tails. This argument isolates the two main ingredients of the proof of Proposition~\ref{prop:light-tail-quotient}: the vector anti-concentration estimate in Proposition~\ref{prop:vector-smallball} and the covering estimate for nuclear norm balls in Lemma~\ref{lem:product-covering} below.

	To this end, fix a deterministic set $\mathcal I\subset[n]$, a $K$-light vector $t\in \R^n$
  supported on $\mathcal I^c$, with $\|t\|_2=1$, and $h\in \R^n$ supported on $\mathcal I$. Split the coordinate set as
  $[d]=I\sqcup J$, where 
 $$
u_i:=x_{i,I}\in\R^m,
\qquad
v_i:=x_{i,J}\in\R^q, \qquad m:=|I|\asymp d,\qquad q:=|J|\asymp d.
$$
Let $P_I:\R^d\to\R^m$ and $P_J:\R^d\to\R^q$ denote the corresponding
coordinate projections. Define the $I\times J$ block restriction of
$\mathcal L_X^*$ by
	$$
	\mathcal L_{X;I,J}^*\lambda
	:=P_I(\mathcal L_X^*\lambda)P_J^\top
	=\frac1{\sqrt d}\sum_{i=1}^n \lambda_i u_iv_i^\top,
	\qquad \lambda=(\lambda_i)_{i\leq n} \in\R^n.
	$$
	We record two essential observations.
    	\begin{enumerate}[label=\textup{(\roman*)}]
        \item  By coordinate independence, the families $(u_i)_{i\le n}$ and
	$(v_i)_{i\le n}$ are independent. Thus
	$\mathcal L_{X;I,J}^*$ may be viewed as a decoupled rectangular block of
	$\mathcal L_X^*$.
        \item By nuclear/operator norm duality, we have $\|BAC\|_*\le\|B\|_{\op}\|A\|_*\|C\|_{\op}$. Combining with $\|P_I\|_{\op}\leq 1$ and $\|P_J\|_{\op}\leq 1$, we have
\begin{equation}\label{eq:submatrix:L:nuclear:norm:lower:bound}
	\|\mathcal L_X^*\lambda \|_*
	\ge \|\mathcal L_{X;I,J}^*\lambda \|_*,
	\qquad \lambda\in\R^n.
	\end{equation}
    Thus, to lower bound $\|\mathcal{L}^*_X \lambda\|_*$, it suffices to lower bound its decoupled version $\|\mathcal L_{X;I,J}^*\lambda \|_*$.
    \end{enumerate}
	Set $\lambda=h+t$. We condition on the $\sigma$-algebra generated by $(u_i)_{i\notin \mathcal{I}}$ and $(x_i)_{i\in \mathcal{I}}$, i.e.
	$$
	 \mathcal F:=
	 \sigma(u_i:i\notin\mathcal I)\vee
	 \sigma(x_i:i\in\mathcal I).
	$$
	Observe that $B:=\mathcal L_{X;I,J}^*h$ is $\mathcal F$-measurable since $h$ is supported on $\mathcal I$. On the other hand, since $t$ is supported on
$\mathcal I^c$, we have that $Y:=\mathcal L_{X;I,J}^*t
=
\frac1{\sqrt d}\sum_{i\notin\mathcal I}t_i u_iv_i^\top$, where $u_i$ have been revealed, while the coordinates in $v_i$ is independent of $\mathcal F$. Consequently,
\begin{equation}\label{eq:structure:L:given:F}
\left.
\mathcal L_{X;I,J}^*\lambda
\,\right|\,\mathcal F
\stackrel{d}{=}
\underbrace{
B
}_{\mathcal F\text{-measurable shift}}
+
\underbrace{
Y
}_{\text{matrix with independent columns}}.
\end{equation}
We then apply Proposition~\ref{prop:vector-smallball} to show that the randomness in $(v_i)_{i\notin\mathcal I}$ prevents the latter term from being nearly cancelled by an arbitrary $\mathcal F$-measurable shift. More precisely, we have the following.

\begin{lemma}
	\label{lem:fixed-tail-quotient}
	Under the same assumption as in Proposition~\ref{prop:light-tail-quotient}, there exist $K>0$ and $\gamma>0$, depending only on
		$\alpha$ and $K_x$, such that for every fixed $C_0>0$, every fixed
		$\mathcal I\subset[n]$ with $|\mathcal I|\le C_0d\log d$, and every fixed
		$K$-light vector $t$ supported on $\mathcal I^c$ with $\|t\|_2=1$, with
		probability $1-o(1)$,
	$$
	\inf_{\substack{\supp h\subset\mathcal I\\
	\|h\|_1\le d^{2/3}}}
	\|\mathcal L_{X;I,J}^*(h+t)\|_*
	\ge\gamma d
	$$
	Consequently, the same conclusion holds with $\mathcal L_X^*$ in place
	of $\mathcal L_{X;I,J}^*$.
\end{lemma}
To prove this lemma, we must pass from anti-concentration for a single
column to a lower bound for the nuclear norm of the entire matrix.
Conditionally on $\mathcal F$,
Proposition~\ref{prop:vector-smallball} controls each column of $Y$
separately. On the other hand, the event $\|Y+B\|_*\le\gamma d$ means that $Y$ lies in the shifted nuclear norm ball
$$
-B+\{A\in\R^{m\times q}:\|A\|_*\le\gamma d\}.
$$
We therefore cover this set by products of Euclidean column balls. Let
$$
\|A\|_{1,2}
:=
\sup_{\|x\|_1\le1}\|Ax\|_2
=
\max_{1\le j\le q}\|A_j\|_2,\quad\textnormal{where}\quad A=\begin{bmatrix}
A_1&...&A_q
\end{bmatrix}
\in\R^{m\times q}
$$
The following covering lemma plays a crucial role for the proof of Lemma~\ref{lem:fixed-tail-quotient}, and more generally, Proposition~\ref{prop:light-tail-quotient}.
	\begin{lemma}\label{lem:product-covering}
		Fix $0<c_0<C_0<\infty$, and suppose
		$c_0d\le m,q\le C_0d$. There is a constant
		$C=C(c_0,C_0)>0$ such that
		$$
			N(\mathcal K,\|\cdot\|_{1,2},\gamma)\le C^{mq},\quad\textnormal{where}\quad \mathcal K=\{A\in\R^{m\times q}:\|A\|_*\le\gamma d\},
		$$
        and $N(\mathcal K,\|\cdot\|_{1,2},\gamma)$ is the least cardinality
		of a set $\mathcal N\subset\R^{m\times q}$ such that every
		$A\in\mathcal K$ has some $A_0\in\mathcal N$ with
		$\|A-A_0\|_{1,2}\le\gamma$. 
	\end{lemma}
	
	\begin{proof}
		Let $\mathcal P=\{A:\|A\|_{1,2}\le\gamma\}$, the product of $q$
		Euclidean $m$-balls of radius $\gamma$.  The standard volumetric packing
		estimate (see e.g.~\cite[Section~4.2]{vershynin2018highdimensional}) gives
		$$
		 N(\mathcal K,\mathcal P)
		 \le 2^{mq}\frac{\Vol(\mathcal K+\mathcal P)}{\Vol(\mathcal P)};
		$$
	If $A=[A_1,\ldots,A_q]\in\mathcal P$, then
$A=\sum_{j=1}^q A_je_j^\top$, and hence
$$
\|A\|_*
\le\sum_{j=1}^q\|A_je_j^\top\|_*
=\sum_{j=1}^q\|A_j\|_2
\le q\gamma.
$$
As a result, $\mathcal P\subset q\gamma B_*$, where $B_*$ denotes the nuclear norm unit ball in
$\R^{m\times q}$. Since
$\mathcal K=\gamma d B_*$ and $q\asymp d$, another application of the
triangle inequality gives
$$
\mathcal K+\mathcal P
\subset C\gamma d\,B_*,
$$
It remains to compare the volumes of this nuclear norm ball and the
product set $\mathcal P$. Let $G\in \R^{m\times q}$ with i.i.d. $N(0,1)$ entries. Urysohn's inequality (see~\cite[Corollary~1.4 and Remark~1.5]{pisier1989volume}) and
nuclear/operator norm duality give
$$
\left(
\frac{\Vol(B_*)}{\Vol(B_2^{mq})}
\right)^{1/(mq)}
\le
\frac{\E\|G\|_{\op}}{\E\|G\|_F}
\le
\frac{C}{\sqrt d},
$$
where the last inequality uses
$\E\|G\|_{\op}\le C(\sqrt m+\sqrt q)\le C\sqrt d$ (see e.g.~\cite[Theorem 4.4.5]{vershynin2018highdimensional}) and
$\E\|G\|_F\asymp\sqrt{mq}\asymp d$. We have
$\Vol(B_2^{mq})^{1/(mq)}\asymp (mq)^{-1/2}\asymp d^{-1}$ by standard volume estimate of a Euclidean ball. Therefore,
$$
\Vol(B_*)^{1/(mq)}\le Cd^{-3/2},
$$
On the other hand, $\mathcal P$ is the product of $q$ Euclidean
$m$-balls of radius $\gamma$, so
$$
\Vol(\mathcal P)^{1/(mq)}
=
\gamma\,\Vol(B_2^m)^{1/m}
\ge
c\frac{\gamma}{\sqrt m}.
$$
Substituting these into the volumetric packing bound gives
$$
N(\mathcal K,\mathcal P)
\le
2^{mq}\frac{\Vol(\mathcal K+\mathcal P)}{\Vol(\mathcal P)}\leq (C\gamma d)^{mq}\times \left(\frac{C}{\gamma d}\right)^{mq}
\le (C')^{mq},
$$
which proves the lemma.
	\end{proof}
Then, Lemma~\ref{lem:fixed-tail-quotient} follows by combining Proposition~\ref{prop:vector-smallball} and Lemma~\ref{lem:product-covering}. Since Lemma~\ref{lem:fixed-tail-quotient} is only to motivate the proof of Proposition~\ref{prop:light-tail-quotient}, we only give a brief sketch of the proof.
	\begin{proof}[Sketch of the proof of Lemma~\ref{lem:fixed-tail-quotient}]
	For $j\in J$, the $j$-th column of
	$Y=\mathcal L_{X;I,J}^*t$ is
	$$
	Y_j=\frac1{\sqrt d}\sum_{i\notin\mathcal I}t_i u_i x_{ij},
	\qquad
	\E[Y_jY_j^\top\mid\mathcal F]=\frac{1}{d}\Sigma_t,
	\qquad
	\Sigma_t:=\sum_{i\notin\mathcal I}t_i^2u_iu_i^\top.
	$$
	Since $t$ is deterministic, $\|t\|_2=1$, and
	$\|t\|_\infty^2\le(Kd)^{-1}$, Bernstein's inequality followed by a
	sphere-net argument show that provided $K$ is sufficiently large, with probability $1-o(1)$,
\begin{equation}\label{eq:deterministic:tail:covariance:lower:bound}
    \Sigma_{t}\succeq \frac{1}{2}I_m    
    \end{equation}
    We also work on the
	w.h.p. event $\max_i(\|u_i\|_2+\|v_i\|_2)\le C\sqrt d$. On this event, Proposition~\ref{prop:vector-smallball} gives, for every
	column $Y_j$ and every $\eta\ge CK^{-1/2}$,
	$$
	\sup_{z\in\R^m}
	\P\left(
	\|Y_j-z\|_2\le\eta\sqrt{m/d}
	\mathrel{}\middle|\mathrel{}\mathcal F
	\right)
	\le(C\eta)^m.
	$$
	Since $Y_j$'s are independent given $\mathcal{F}$, covering a shifted
	nuclear norm ball by Lemma~\ref{lem:product-covering} and multiplying
	the columnwise estimates therefore gives, uniformly over every
		$\mathcal F$-measurable shift $B$,
	$$
	\P\left\{
		\|Y+B\|_*\le2\gamma d
	\mathrel{}\middle|\mathrel{}\mathcal F
	\right\}
	\le(C\gamma)^{mq},
	\qquad
	\gamma\ge CK^{-1/2}.
	$$
	Finally, cover
	$\{h:\supp h\subset\mathcal I,\ \|h\|_1\le d^{2/3}\}$ by an
	$\ell_1$-net of radius $c\gamma\sqrt d$. Its cardinality is $\left(1+C\gamma^{-1}d^{1/6}\right)^{|\mathcal I|}
	=\exp(o(d^2)).$ The row-norm bound shows that replacing $h$ by its net point changes
	$\|\mathcal L_{X;I,J}^*(h+t)\|_*$ by at most $\gamma d/2$.
	A union bound therefore gives a failure probability bounded by
	\begin{equation}\label{eq:failure:probability:easy}
	o(1)+\exp(o(d^2))(C\gamma)^{mq}=o(1),
	\end{equation}
	after first choosing $\gamma>0$ sufficiently small and then $K$
	sufficiently large. Since
	$\|\mathcal L_X^*\lambda\|_*
	\ge\|\mathcal L_{X;I,J}^*\lambda\|_*$, the same conclusion holds for
	$\mathcal L_X^*$.
\end{proof}
\begin{remark}\label{rem:gaussian:simpler}
	The above argument also explains the limitation of using a single rectangular
	block $I\times J$. Since $m+q=d$, we have
	$mq\le d^2/4$, with equality for a balanced
	partition. Thus, the deterministic-tail failure probability~\eqref{eq:failure:probability:easy} contributes at best a factor
	$(C\gamma)^{d^2/4+o(d^2)}$. Making the estimate uniform over all tails
	requires a net of size $(C/\gamma)^{n+o(d^2)}$. After making the
	covariance estimate uniform (see Lemma~\ref{lem:effective-rank} below), the resulting union bound closes for
	$\alpha<1/4$, which covers the Gaussian case for which
	$\alpha_\star(3)=1/4$. To treat general $\kappa>1$, however, where
	$\alpha_\star(\kappa)$ approaches $1/2$ as $\kappa\downarrow1$, the sequential block revealing explained in the next subsection is important.
\end{remark}
	\subsection{Uniform light-tail lower bound via sequential block revealing}\label{sec:light-tail-quotient}
	This section proves Proposition~\ref{prop:light-tail-quotient} by modifying the argument for deterministic light-tail argument from the previous section. The main idea is to use multiple blocks instead of single block $I\times J$, and reveal the blocks sequentially. 
    
	 To this end, choose a number of blocks $L=L(\alpha)\ge2$ to be determined later, and partition
	$$
	 [d]=I_1\sqcup\cdots\sqcup I_L,\qquad |I_\ell|=d/L+O(1).
	$$
	For $\ell=2,\ldots,L$, set
	$$
	U_\ell:=I_1\cup\cdots\cup I_{\ell-1},\qquad
	p_\ell:=|U_\ell|,\qquad q_\ell:=|I_\ell|,
	$$
	and write
	$$
	u_{i,\ell}:=x_{i,U_\ell}\in\R^{p_\ell},\qquad
	v_{i,\ell}:=x_{i,I_\ell}\in\R^{q_\ell}.
	$$
	Define the maximum of nuclear norms
	$$
	\mathcal M_L(\lambda):=\max_{2\le \ell\le L}\|\mathcal L_{X,\ell}^*\lambda\|_*,\quad\textnormal{where}\quad \mathcal L_{X,\ell}^*\lambda
	:=\mathcal L_{X;U_\ell,I_\ell}^*\lambda
	=\frac1{\sqrt d}\sum_{i=1}^n \lambda_i u_{i,\ell}v_{i,\ell}^\top\in \R^{p_\ell\times q_\ell}.
	$$
   Because $U_\ell$ and $I_\ell$ are disjoint,
$\mathcal L_{X,\ell}^*\lambda$ is the $U_\ell\times I_\ell$ coordinate submatrix
of $\mathcal L_X^*\lambda$. Therefore, as in~\eqref{eq:submatrix:L:nuclear:norm:lower:bound},
\begin{equation}
\label{eq:block-dominated}
\|\mathcal L_X^*\lambda\|_*
\ge \mathcal M_L(\lambda),
\qquad
\lambda\in\R^n.
\end{equation}
It is thus enough to prove
$\mathcal M_L(h+t)\gtrsim d\|t\|_2$ uniformly over the heads and tails
appearing in Proposition~\ref{prop:light-tail-quotient}. The total number of entries across $U_\ell\times I_\ell$ is
$$
D_L
:=
\sum_{\ell=2}^L p_\ell q_\ell
=
\sum_{1\le r<\ell\le L}|I_r||I_\ell|
=
\left(\frac{1-1/L}{2}+o(1)\right)d^2.
$$
For $L=2$, we have $D_2=(1/4+o(1))d^2$, which is precisely the
exponent in the failure probability bound~\eqref{eq:failure:probability:easy}.
The sequential block revealing argument below replaces $D_2$ by $D_L$, up to a negligible loss. Since
$D_L=((1-1/L)/2+o(1))d^2$, choosing $L=L(\alpha)$ sufficiently large
yields the full range $\alpha<1/2$.
Figure~\ref{fig:block-map-schematic} illustrates these blocks.
	\begin{figure}[t]
		\centering
		\begin{tikzpicture}[x=1.35cm,y=1cm,font=\scriptsize]
			\fill[black!10] (1,0) rectangle (2,-1);
			\fill[black!16] (2,0) rectangle (3,-2);
			\fill[black!22] (3,0) rectangle (4,-3);
			\fill[black!28] (4,0) rectangle (5,-4);
			\fill[black!6] (0,-1) rectangle (1,-2);
			\fill[black!6] (0,-2) rectangle (2,-3);
			\fill[black!6] (0,-3) rectangle (3,-4);
			\fill[black!6] (0,-4) rectangle (4,-5);
			\fill[black] (0,0) rectangle (1,-1);
			\fill[black] (1,-1) rectangle (2,-2);
			\fill[black] (2,-2) rectangle (3,-3);
			\fill[black] (3,-3) rectangle (4,-4);
			\fill[black] (4,-4) rectangle (5,-5);
			\draw[step=1,black!45,thin] (0,0) grid (5,-5);
			\draw[black,thick] (0,0) rectangle (5,-5);
			\draw[black,line width=1.1pt] (1,0) rectangle (2,-1);
			\draw[black,line width=1.1pt] (2,0) rectangle (3,-2);
			\draw[black,line width=1.1pt] (3,0) rectangle (4,-3);
			\draw[black,line width=1.1pt] (4,0) rectangle (5,-4);
			\draw[black,line width=1.8pt] (1,-1) -- (2,-1);
			\draw[black,line width=1.8pt] (2,-2) -- (3,-2);
			\draw[black,line width=1.8pt] (3,-3) -- (4,-3);
			\draw[black,line width=1.8pt] (4,-4) -- (5,-4);
			\node at (0.5,0.35) {$I_1$};
			\node at (1.5,0.35) {$I_2$};
			\node at (2.5,0.35) {$I_3$};
			\node at (3.5,0.35) {$I_4$};
			\node at (4.5,0.35) {$I_5$};
			\node at (-0.35,-0.5) {$I_1$};
			\node at (-0.35,-1.5) {$I_2$};
			\node at (-0.35,-2.5) {$I_3$};
			\node at (-0.35,-3.5) {$I_4$};
			\node at (-0.35,-4.5) {$I_5$};
			\node at (1.5,-0.5) {$\boldsymbol{\mathcal L_{X,2}^*\lambda}$};
			\node at (2.5,-1.0) {$\boldsymbol{\mathcal L_{X,3}^*\lambda}$};
			\node at (3.5,-1.5) {$\boldsymbol{\mathcal L_{X,4}^*\lambda}$};
			\node at (4.5,-2.0) {$\boldsymbol{\mathcal L_{X,5}^*\lambda}$};
		\end{tikzpicture}
		\caption{Sequential blocks for $L=5$.  At stage $\ell$, the shaded rectangle
			$\mathcal L_{X,\ell}^*\lambda$ uses the previously exposed rows
			$U_\ell=I_1\cup\cdots\cup I_{\ell-1}$ and the fresh columns $I_\ell$.}
		\label{fig:block-map-schematic}
	\end{figure}
	
For $\mathcal I\subset[n]$, consider a decomposition $\lambda=h+t$, where
$h$ is supported on $\mathcal I$ and $t$ is supported on $\mathcal I^c$.
Define the filtration $(\mathcal F_{\ell-1}^{\mathcal I})_{2\le\ell\le L}$
by
$$
\mathcal F_{\ell-1}^{\mathcal I}
:=
\sigma\bigl(u_{i,\ell}:i\notin\mathcal I\bigr)
\vee
\sigma\bigl(x_i:i\in\mathcal I\bigr),
\qquad \ell=2,\ldots,L.
$$
Observe that since $h$ is supported on $\mathcal I$,
$\mathcal L_{X,\ell}^*h$ is
$\sigma(x_i:i\in\mathcal I)$-measurable, and hence
$\mathcal F_{\ell-1}^{\mathcal I}$-measurable. Moreover, the vectors
$u_{i,\ell}$ are $\mathcal F_{\ell-1}^{\mathcal I}$-measurable, whereas
the family $(v_{i,\ell})_{i\notin\mathcal I}$ is independent of
$\mathcal F_{\ell-1}^{\mathcal I}$. Thus, conditionally on
$\mathcal F_{\ell-1}^{\mathcal I}$,
$\mathcal L_{X,\ell}^*\lambda$ has the same structure as in
\eqref{eq:structure:L:given:F}. In particular, the matrix
$\mathcal L_{X,\ell}^*t$ has conditionally independent columns
$$
Y_{\ell,j}
=
\frac1{\sqrt d}\sum_{i\notin\mathcal I}t_i u_{i,\ell}x_{ij},
\qquad j\in I_\ell,
$$
with
$$
\Cov\left(Y_{\ell,j}\mid\mathcal F_{\ell-1}^{\mathcal I}\right)
=
\frac1d\Sigma_{t,\ell},
\qquad
\Sigma_{t,\ell}
:=
\sum_{i\notin\mathcal I}
t_i^2u_{i,\ell}u_{i,\ell}^\top.
$$
We will apply Proposition~\ref{prop:vector-smallball} to these fresh
columns $(Y_{\ell,j})_{j\in I_{\ell}}$ conditional on $\mathcal{F}_{\ell-1}^{\mathcal I}$.

Compared with the deterministic light-tail argument in
Section~\ref{subsec:deterministic}, the new difficulty is to control
$\Sigma_{t,\ell}$ uniformly over light tails $t$. A uniform bound
$\Sigma_{t,\ell}\succeq cI_{p_\ell}$ is generally false and is also
stronger than needed. It suffices to show that $\Sigma_{t,\ell}$ is
bounded below by a positive constant on a subspace of dimension at least
$(1-\eps)p_\ell$ for any fixed $\eps>0$. The following lemma provides precisely this
effective-rank estimate, uniformly over all $K$-light tails $t\in \R^n$, i.e. $\|t\|_{\infty}\leq (Kd)^{-1/2}\|t\|_2$.
    
	
	\begin{lemma}\label{lem:effective-rank}
		Suppose $c_0p\le d\le C_0p$ and $n\le C_0d^2$ hold for some $0<c_0<C_0<\infty$.  Let $z_1,\ldots,z_n\in\R^p$ be
		i.i.d. random vectors whose coordinates are independent, mean zero, variance
		one, common fourth moment $\kappa$, and satisfy
		$\max_{i\le n}\max_{j\le p}\|z_{ij}\|_{\psi_2}\le K_z$.  
        
        Then, there exists an absolute constant $\mu>0$ such that the following holds. For any $\eps \in (0,1)$, there
		exists $K_0>0$ depending only on
		$c_0,C_0,\eps,K_z$, such that for every $K\ge K_0$, with probability $1-o(1)$, the following holds: for all $t\in \R^n$ with $\|t\|_{\infty}\leq (Kd)^{-1/2}$ and $ \|t\|_2=1$, the matrix
		$$
		\Sigma_t:=\sum_{i=1}^{n}t_i^2z_i z_i^\top
		$$
		has at most $\lfloor \eps p\rfloor$ eigenvalues below $\mu$.  Equivalently,
		there is a subspace $\mathsf V\equiv \mathsf V_t \subset\R^p$ such that
		$$
		\dim \mathsf V\ge (1-\eps)p,
		\qquad
		P_{\mathsf V}\Sigma_tP_{\mathsf V}\succeq \mu P_{\mathsf V}, 
		$$
        where $P_{\mathsf V}\in \R^{p\times p}$ denotes the projection matrix onto $\mathsf V$.
	\end{lemma}
\begin{proof}
	Since $n\le C_0d^2\le C_0^3p^2$, Proposition~
\ref{thm:apriori-estimates} applies to
$Z=[z_1,\ldots,z_n]\in\R^{p\times n}$ in dimension $p$, with
$K_x=K_z$. In particular, we have on this event
	$$
	\Lambda:=
	\sup_{\|A\|_F\le1}
	\left(
	\sum_{i=1}^n
	\langle z_iz_i^\top-I_p,A\rangle^2
	\right)^{1/2}
	\le C_1p.
	$$
	We prove the conclusion deterministically on this event.

	Let $\mathsf H\subset\R^p$ be any subspace with
	$\dim(\mathsf H)=\lfloor\eps p\rfloor$, and let $P_{\mathsf H}$ be
	its orthogonal projection. Since
	$\|P_{\mathsf H}\|_F^2=\dim(\mathsf H)$, applying the definition of
	$\Lambda$ to
	$A=P_{\mathsf H}/\sqrt{\dim(\mathsf H)}$ gives
	$$
	\sum_{i=1}^n
	\left(\|P_{\mathsf H}z_i\|_2^2-\dim(\mathsf H)\right)^2
	\le \Lambda^2\dim(\mathsf H).
	$$
    Thus, by Markov's inequality
	$$
	\left|\left\{
	i\le n:\|P_{\mathsf H}z_i\|_2^2
	\le\frac{\dim(\mathsf H)}2
	\right\}\right|
	\le
	\frac{4\Lambda^2}{\dim(\mathsf H)}
	\le C_2d,
	$$
	where $C_2=C_2(c_0,C_0,\eps,K_z)$. Set $K_0=2C_2$.

	We now use lightness. Let $t\in \R^n$ such that $\|t\|_{\infty}\leq (Kd)^{-1/2}$ and $\|t\|_2=1$. By the inequality above, for any subspace $\mathsf H$ with $\dim(\mathsf H)=\lfloor \eps p\rfloor$, we have
	$$
	\sum_{i:
	\|P_{\mathsf H}z_i\|_2^2\le\dim(\mathsf H)/2}
	t_i^2
	\le C_2d\|t\|_\infty^2
	\le\frac{C_2}{K}\leq \frac12.
	$$
	Since $\|t\|_2=1$, the remaining coordinates at least half of the mass. Therefore,
	$$
	\begin{aligned}
	\Tr(P_{\mathsf H}\Sigma_tP_{\mathsf H})
	=
	\sum_{i=1}^{n}
	t_i^2\|P_{\mathsf H}z_i\|_2^2 \ge
	\frac{\dim(\mathsf H)}2
	\sum_{i:
	\|P_{\mathsf H}z_i\|_2^2>\dim(\mathsf H)/2}
	t_i^2
	\ge\frac{\dim(\mathsf H)}4.
	\end{aligned}
	$$
	Set $\mu=1/8$. If $\Sigma_t$ had at least
	$\lfloor\eps p\rfloor$ eigenvalues below $\mu$, let $\mathsf H$ be
	the span of that many corresponding eigenvectors. Then
	$$
	\Tr(P_{\mathsf H}\Sigma_tP_{\mathsf H})
	<\mu\dim(\mathsf H)
	=\frac{\dim(\mathsf H)}8,
	$$
	contradicting the preceding lower bound. Hence the span $\mathsf V$
	of the eigenvectors whose eigenvalues are at least $\mu$ satisfies $\dim\mathsf V\ge(1-\eps)p$ and $P_{\mathsf V}\Sigma_tP_{\mathsf V}\succeq\mu P_{\mathsf V}$, which concludes the proof.
\end{proof}

\begin{corollary}
	\label{cor:block-covariance-event}
Assume $(x_i)_{i\leq n}$ satisfy the hypotheses of
	Proposition~\ref{prop:light-tail-quotient}, and $n/d^2 \to \alpha\in (0,\infty)$. 
    
    There exists an absolute constant $\mu>0$ such that the following holds. For every fixed $L\geq 2$ and $\eps\in(0,1)$, there exist $
	K_0=K_0(\alpha,K_x,L,\eps)>0$ and $C=C(K_x,L)>0$ such that for every $K\ge K_0$,
	with probability $1-o(1)$, the following holds: for every
	$\ell=2,\ldots,L$, every $\mathcal I\subset[n]$, and every
	$t\in\R^n$ supported on $\mathcal I^c$ satisfying
	$\|t\|_2=1$ and $\|t\|_\infty\le(Kd)^{-1/2}$, let
	$\mathsf V_{t,\ell}$ be the span of the eigenvectors of
	$\Sigma_{t,\ell}$ whose eigenvalues are at least $\mu$, where
	\[
    \Sigma_{t,\ell}\equiv\sum_{i\notin \mathcal{I}} t_i^2 u_{i,\ell} u_{i,\ell}^\top,\qquad m_{t,\ell}:=\dim\mathsf V_{t,\ell}.
    \]
    Then for every such $\mathcal{I}, t$, the $\mathcal{F}_{\ell-1}^{\mathcal I}$-measurable event
	\begin{equation}\label{eq:def:good:event:ell}
	\mathcal G_\ell^{\mathcal I}(t)
	:=
	\left\{
	\max_{i\le n}\|u_{i,\ell}\|_2^2\le Cd,
	\quad
	m_{t,\ell}\ge(1-\eps)p_\ell
	\right\}
	\end{equation}
	occurs, and $P_{\mathsf V_{t,\ell}}\Sigma_{t,\ell}P_{\mathsf V_{t,\ell}}
	\succeq\mu P_{\mathsf V_{t,\ell}}$ holds.
\end{corollary}

\begin{proof}
	For each $\ell=2,\ldots,L$, let $\mathcal E_\ell$ be the event on
	which $\max_{i\le n}\|u_{i,\ell}\|_2^2\le Cd$ and
	$m_{t,\ell}\ge(1-\eps)p_\ell$ simultaneously for every
	$\mathcal I\subset[n]$ and every $t$ supported on
	$\mathcal I^c$ with
	$\|t\|_2=1$ and $\|t\|_\infty\le(Kd)^{-1/2}$.

	Since $p_\ell\asymp_Ld$, $n/d^2\to\alpha$, and the coordinates of
	$u_{i,\ell}$ satisfy the same assumptions as those of $x_i$,
	Lemma~\ref{lem:effective-rank}, applied with $p=p_\ell$ and
	$z_i=u_{i,\ell}$, gives
	$m_{t,\ell}\ge(1-\eps)p_\ell$ simultaneously for all such $t$ with
	probability $1-o(1)$, after choosing
	$K_0=K_0(\alpha,K_x,L,\eps)$ sufficiently large. Moreover, standard
	subgaussian norm concentration
	\cite[Theorem~3.1.1]{vershynin2018highdimensional} gives
	$$
	\P\left(\max_{i\le n}\|u_{i,\ell}\|_2^2>Cd\right)
	\le 2ne^{-cd}=o(1).
	$$
	Thus $\P(\mathcal E_\ell)=1-o(1)$. Since $L$ is fixed,
	$\P(\bigcap_{\ell=2}^L\mathcal E_\ell)\to1$. On this intersection,
	every event $\mathcal G_\ell^{\mathcal I}(t)$ in the statement
	occurs simultaneously. The covariance lower bound follows directly
	from the definition of $\mathsf V_{t,\ell}$.

	Finally, for fixed $\mathcal I,t$, all the vectors $u_{i,\ell}$ are
	$\mathcal F_{\ell-1}^{\mathcal I}$-measurable: those with
	$i\notin\mathcal I$ are included, while $x_i$ is revealed for $i\in\mathcal I$.
	Hence $\Sigma_{t,\ell}$, $\mathsf V_{t,\ell}$, and
	$\mathcal G_\ell^{\mathcal I}(t)$ are measurable with respect to
	$\mathcal F_{\ell-1}^{\mathcal I}$.
\end{proof}

The next lemma bounds the small-ball probability when one block is revealed.
	
	\begin{lemma}
	\label{lem:one-stage}
	Assume $(x_i)_{i\le n}$ satisfy the hypotheses of
	Proposition~\ref{prop:light-tail-quotient}. Fix $L\ge2$ and
	$\eps\in(0,1)$. There exists
	$C=C(L,\eps,K_x)>0$ such that the following holds.

	Fix $K>0$. Consider $\ell\in\{2,\ldots,L\}$, $\mathcal I\subset[n]$, and
	$t\in\R^n$ supported on $\mathcal I^c$ with $\|t\|_2=1$ and $\|t\|_{\infty}\leq (Kd)^{-1/2}$. Then, for every
		$\gamma\ge C K^{-1/2}$, and $\mathcal F_{\ell-1}^{\mathcal I}$-measurable $B_\ell\in\R^{p_\ell\times q_\ell}$,
	$$
	\mathbf 1_{\mathcal G_\ell^{\mathcal I}(t)}
	\P\left(
		\|\mathcal L_{X,\ell}^*t+B_\ell\|_*\le\gamma d
	\,\middle|\,
	\mathcal F_{\ell-1}^{\mathcal I}
	\right)
	\le
	\mathbf 1_{\mathcal G_\ell^{\mathcal I}(t)}
	(C \gamma)^{m_{t,\ell}q_\ell}.
	$$
\end{lemma}
	
	\begin{proof}
	Throughout, we denote $C>0$ by a constant that depends only on $L,\eps,K_x$. Also, we fix $\ell, \mathcal{I}$, and $t$ as in the statement, and write $\mathsf V\equiv \mathsf V_{t,\ell}, \Sigma\equiv \Sigma_{t,\ell}$ and $m\equiv m_{t,\ell}.$ for simplicity. Note that we have $P_{\mathsf V}\Sigma P_{\mathsf V}\succeq \mu P_{\mathsf V}$ by definition of $\mathsf{V}\equiv \mathsf{V}_{t,\ell}$, and $m\geq (1-\eps)p_{\ell}$ holds on the event $\mathcal{G}_{\ell}^{\mathcal I}(t)$.

Since left multiplication by $P_{\mathsf V}$ does not increase the
	nuclear norm,
	$$
		\|\mathcal L_{X,\ell}^*t+B_\ell\|_*\le\gamma d
	\quad\Longrightarrow\quad
		P_{\mathsf V}\mathcal L_{X,\ell}^*t \in-P_{\mathsf V}B_\ell+
	\left\{
	A\in\mathsf V\otimes\R^{q_\ell}:\|A\|_*\le\gamma d
	\right\}.
	$$
	The columns of $P_{\mathsf V}\mathcal L_{X,\ell}^*t$ are
	$$
	Y_j
	=
	\frac1{\sqrt d}\sum_{i\notin\mathcal I}
	t_iP_{\mathsf V}u_{i,\ell}x_{ij},
	\qquad j\in I_\ell,
	$$
	  which are conditionally independent given $\mathcal{F}_{\ell-1}^{\mathcal I}$. Since $m,q_\ell\asymp_{L,\eps}d$, Lemma~\ref{lem:product-covering},
after identifying $\mathsf V$ isometrically with $\R^m$, covers
$$
\left\{
A\in\mathsf V\otimes\R^{q_\ell}:\|A\|_*\le\gamma d
\right\}
$$
by at most $C^{mq_\ell}$ balls of radius $\gamma$ in
$\|\cdot\|_{1,2}$. Indeed, this identification preserves singular
values and Euclidean norms of the columns, and thus both $\|\cdot\|_*$ and
$\|\cdot\|_{1,2}$. Conditional independence of the columns then gives
	\begin{equation}\label{eq:one-stage-covering}
	\begin{aligned}
	\mathbf 1_{\mathcal G_\ell^{\mathcal I}(t)}
	\P\left(
		\|\mathcal L_{X,\ell}^*t+B_\ell\|_*\le\gamma d
	\,\middle|\,
	\mathcal F_{\ell-1}^{\mathcal I}
	\right)\le
	\mathbf 1_{\mathcal G_\ell^{\mathcal I}(t)}
	C^{mq_\ell}
	\prod_{j\in I_\ell}\left\{
	\sup_{z\in\mathsf V}
	\P\left(
	\|Y_j-z\|_2\le\gamma
	\,\middle|\,
	\mathcal F_{\ell-1}^{\mathcal I}
	\right)\right\}.
	\end{aligned}
	\end{equation}
    We next bound each supremum in \eqref{eq:one-stage-covering} using
Proposition~\ref{prop:vector-smallball} by checking its assumptions. Set
$$
w_i:=\mu^{-1/2}t_iP_{\mathsf V}u_{i,\ell},
\qquad i\notin\mathcal I.
$$
By the definition of $\mathsf V$ and the event
$\mathcal G_\ell^{\mathcal I}(t)$ in~\eqref{eq:def:good:event:ell},
$$
\sum_{i\notin\mathcal I}w_iw_i^\top
=
\mu^{-1}P_{\mathsf V}\Sigma_{t,\ell}P_{\mathsf V}
\succeq P_{\mathsf V},
\qquad
\max_{i\notin\mathcal I}\|w_i\|_2
\le CK^{-1/2},
$$
where the last inequality uses $\|t\|_\infty\le(Kd)^{-1/2}$ and
$\max_i\|u_{i,\ell}\|_2\le C\sqrt d$.

Now, choose a Euclidean isometry $U:\mathsf V\to\R^m$ and set
$\widetilde w_i:=Uw_i$. Then,
$$
\sum_{i\notin\mathcal I}\widetilde w_i\widetilde w_i^\top
\succeq I_m,
\qquad
\max_{i\notin\mathcal I}\|\widetilde w_i\|_2
\le CK^{-1/2},
\qquad
UY_j
=
\sqrt{\frac{\mu}{d}}
\sum_{i\notin\mathcal I}\widetilde w_i x_{ij},\quad j\in I_{\ell}.
$$
Since we assumed that $(x_i)_{i\leq n}$ has independent coordinates, $(x_{ij})_{i\notin\mathcal I, j\in I_{\ell}}$ is independent of $\mathcal F_{\ell-1}^{\mathcal I}$. Also, since $\E[x_{ij}]=0$, $\Var(x_{ij})=1$ with $\sup_{i,j} \|x_{ij}\|_{\psi_2}\leq K_x$, we have for some $M=M(K_x)>0$,
$$
\inf_{i\notin \mathcal I,j\in I_{\ell}}
\E\left[
x_{ij}^2\mathbf 1_{\{|x_{ij}|\le M\}}
\,\Big|\, \mathcal{F}_{\ell-1}^{\mathcal I}\right]
\ge\frac12.
$$
Note that  on the $\mathcal{F}_{\ell-1}^{\mathcal I}$-measurable event $\mathcal G_\ell^{\mathcal I}(t)$, we have
$m\asymp_{L,\eps}d$, thus $\eta:=\gamma\sqrt{\frac{d}{\mu m}}$ satisfies $\eta\asymp_{L,\eps}\gamma$. Therefore, by taking $C$ large enough in the assumption
$\gamma\ge CK^{-1/2}$, all the hypotheses in Proposition~\ref{prop:vector-smallball} are satisfied for $(UY_j)_{j\in I_{\ell}}\mid \mathcal{F}_{\ell-1}^{\mathcal I}$ on the event $\mathcal{G}_\ell^{\mathcal I}(t)$. Therefore, Proposition~\ref{prop:vector-smallball} yields  
$$
\sup_{z\in\mathsf V}
\P\left(
\|Y_j-z\|_2\le\gamma
\,\middle|\,
\mathcal F_{\ell-1}^{\mathcal I}
\right)\mathbf 1_{\mathcal G_\ell^{\mathcal I}(t)}=\sup_{z'\in\R^m}
\P\left(
\|UY_j-z'\|_2\le\gamma
\,\middle|\,
\mathcal F_{\ell-1}^{\mathcal I}
\right)\mathbf 1_{\mathcal G_\ell^{\mathcal I}(t)}
\le(C\gamma)^m.
$$
Substituting this bound into \eqref{eq:one-stage-covering} concludes the proof.
\end{proof}
	\begin{proof}[Proof of Proposition~\ref{prop:light-tail-quotient}]
	Fix $C_0>0$. Since $\alpha<1/2$, we can choose $L=L(\alpha)\ge2$ large enough and $\eps=\eps(\alpha)\in(0,1)$ small enough so that  
    \[
    \alpha<(1-\eps)\frac{1-1/L}{2}.
    \]
    We will apply Corollary~\ref{cor:block-covariance-event} and
		Lemma~\ref{lem:one-stage} to this choice of $L$ and $\eps$. We let $K_0=K_0(\alpha,K_x)>0$ denote the constant appearing in Corollary~\ref{cor:block-covariance-event}, and let $C=C(\alpha,K_x)>0$ denote the constant appearing in Lemma~\ref{lem:one-stage}. At the end of the proof, we will choose $K\ge K_0$ and $\gamma>0$ so that $C\gamma<1$ and $2\gamma\ge CK^{-1/2}$. With abuse of notation, we write $C>0$ as a constant depending only on $\alpha$ and $K_x$.

	Let
	$\mathcal{G}_{\circ}$ be the event on which all of the following hold:
\begin{enumerate}[label=\textup{(\roman*)}]
	\item By \eqref{eq:block-dominated},
	$\mathcal M_L(\lambda)\le\|\mathcal L_X^*\lambda\|_*$.
	Proposition~\ref{thm:apriori-estimates} therefore gives
	$$
	\mathcal M_L(\lambda)\le Cd\|\lambda\|_2,
	\qquad \lambda\in\R^n.
	$$

		\item By norm concentration of subgaussian vectors
\cite[Theorem~3.1.1]{vershynin2018highdimensional} and a union bound over
$n=O(d^2)$ samples, with probability $1-o(1)$,
we have $\max_{i\le n}\|x_i\|_2^2\le2d$. Since
$\mathcal L_X^*e_i=(x_ix_i^\top-I_d)/\sqrt d$, linearity and the triangle
inequality give, uniformly over $\lambda\in\R^n$,
$$
\mathcal M_L(\lambda)
\le\|\mathcal L_X^*\lambda\|_*
\le
\|\lambda\|_1
\max_{i\le n}\left\{d^{-1/2}\|x_ix_i^\top-I_d\|_*\right\}
\le3\sqrt d\,\|\lambda\|_1.
$$

	\item For every $2\le\ell\le L$, $\mathcal I\subset[n]$, and
$t\in\R^n$ supported on $\mathcal I^c$ with $\|t\|_2=1$ and
$\|t\|_\infty\le(Kd)^{-1/2}$, consider the event
$\mathcal G_\ell^{\mathcal I}(t)$ defined in
\eqref{eq:def:good:event:ell}. Corollary~\ref{cor:block-covariance-event}
guarantees that, with probability $1-o(1)$, all the events $\mathcal G_\ell^{\mathcal I}(t)$ occur
simultaneously. In particular,
$$
m_{t,\ell}\ge(1-\eps)p_\ell,
\qquad
P_{\mathsf V_{t,\ell}}\Sigma_{t,\ell}P_{\mathsf V_{t,\ell}}
\succeq\mu P_{\mathsf V_{t,\ell}}.
$$
\end{enumerate}
Thus, $\mathcal{G}_{\circ}$ holds with probability $1-o(1)$. The properties $(i), (ii)$ are used in netting argument below.

	We first prove a probability bound for a fixed admissible pair
$(h,t)$ on the event $\mathcal{G}_{\circ}$. Note that by homogeneity of $\mathcal{L}^*_X$, it suffices to consider tails $t\in\R^n$ with
$\|t\|_2=1$. Here, we call a pair
$(h,t)\in\R^n\times\R^n$ admissible if
$$
|\supp h|\le C_0d\log d,\qquad
	\supp h\bigcap\supp t=\varnothing,\qquad
\|t\|_2=1,\qquad
\|t\|_\infty\le(Kd)^{-1/2},\qquad
\|h\|_1\le d^{2/3}.
$$
 Set $\mathcal I:=\supp h$. Then $t$ is supported on
$\mathcal I^c$ and $|\mathcal I|\le C_0d\log d$. Define the event
$$
\mathcal H_\ell(h,t)
:=
\left\{
\|\mathcal L_{X,\ell}^*t+B_\ell(h)\|_*
\le2\gamma d
\right\},\quad\textnormal{where}\quad B_\ell(h):=\mathcal L_{X,\ell}^*h.
$$
Since $h$ is supported on $\mathcal I$, $B_\ell(h)$ is
	$\mathcal F_{\ell-1}^{\mathcal I}$-measurable. Moreover, $\mathcal{H}_{r}(h,t)$ is $\mathcal{F}_{\ell-1}^{\mathcal I}$-measurable for $r<\ell$. Thus, by Lemma~\ref{lem:one-stage} and the tower property,
$$
\begin{aligned}
\P\left(
\bigcap_{r=2}^{\ell}
	\bigl(\mathcal G_r^{\mathcal I}(t)\bigcap\mathcal H_r(h,t)\bigr)
\right)
&=
\E\left[
\mathbf 1_{\bigcap_{r=2}^{\ell-1}
	(\mathcal G_r^{\mathcal I}(t)\bigcap\mathcal H_r(h,t))}
\mathbf 1_{\mathcal G_\ell^{\mathcal I}(t)}
\P\left(\mathcal H_\ell(h,t)\mid
\mathcal F_{\ell-1}^{\mathcal I}\right)
\right] \\
&\le
(C\gamma)^{(1-\eps)p_\ell q_\ell}
\P\left(
\bigcap_{r=2}^{\ell-1}
	\bigl(\mathcal G_r^{\mathcal I}(t)\bigcap\mathcal H_r(h,t)\bigr)
\right),
\end{aligned}
$$
where we used that $m_{t,\ell}\ge(1-\eps)p_\ell$ on
the $\mathcal{F}_{\ell-1}^{\mathcal I}$-measurable event $\mathcal G_\ell^{\mathcal I}(t)$.
Iterating over $\ell=2,\ldots,L$,
\begin{equation}\label{eq:fixed-admissible-block-bound}
\P\left(
	\mathcal{G}_{\circ}\bigcap \bigcap_{\ell=2}^L
\mathcal H_\ell(h,t)
\right)\leq \P\left(
\bigcap_{\ell=2}^L
\bigl(\mathcal G_\ell^{\mathcal I}(t)
	\bigcap\mathcal H_\ell(h,t)\bigr)
\right)
\le
(C\gamma)^{(1-\eps)D_L},
\end{equation}
where we recall $D_L\equiv \sum_{\ell=2}^{L}p_{\ell}q_{\ell}$.

We now pass the estimate~\eqref{eq:fixed-admissible-block-bound} for fixed pair $(h,t)$ to a uniform estimate using nets. Define
$$
\mathcal E_{\rm bad}
:=
\left\{
\mathcal M_L(h+t)<\gamma d
\text{ for some admissible pair }(h,t)
\right\}.
$$
By \eqref{eq:block-dominated}, it suffices to prove
$\P(\mathcal E_{\rm bad})=o(1)$. Since
$\P(\mathcal G_\circ^c)=o(1)$, it suffices to show
	$\P(\mathcal E_{\rm bad}\bigcap\mathcal G_\circ)=o(1)$. Increase $C\ge1$, if necessary, so that properties \textup{(i)}, \textup{(ii)} of $\mathcal G_\circ$ hold with the same constant. For
every $\mathcal I\subset[n]$ with
$|\mathcal I|\le C_0d\log d$, choose an $\ell_2$-net
$\mathcal N_{\rm tail}(\mathcal I)$ of $\left\{
t\in\R^{\mathcal I^c}:
\|t\|_2=1,\ 
\|t\|_\infty\le(Kd)^{-1/2}
\right\}$ with radius $\gamma/(16C)$ and centers in the same set. Also choose an
$\ell_1$-net $\mathcal N_{\rm head}(\mathcal I)$ of
$\{h\in\R^{\mathcal I}:\|h\|_1\le d^{2/3}\}$ with radius
$\gamma\sqrt d/(16C)$. Standard volume estimates
\cite[Section~4.2.1]{vershynin2018highdimensional} give, for a possibly larger constant $C$,
\begin{equation}\label{eq:net:bound:head:tail}
|\mathcal N_{\rm tail}(\mathcal I)|\le(C/\gamma)^n,
\qquad
|\mathcal N_{\rm head}(\mathcal I)|
\le(1+C\gamma^{-1}d^{1/6})^{|\mathcal I|}=\exp(o(d^2)).
\end{equation}
	Suppose $\mathcal E_{\rm bad}\bigcap\mathcal G_\circ$ occurs, and let
$(h,t)$ be an admissible pair such that $\mathcal{M}_L(h+t)<\gamma d$. Set
$\mathcal I=\supp h$ and choose the closest $(h_0,t_0)\in \mathcal{N}_{\rm head}(\mathcal I)\times \mathcal{N}_{\rm tail}(\mathcal I)$ in $\ell_1\times \ell_2$ sense.
On $\mathcal G_\circ$, properties \textup{(i)} and \textup{(ii)} give
$$
\mathcal M_L(t-t_0)\le\frac{\gamma d}{16},
\qquad
\mathcal M_L(h-h_0)\le\frac{\gamma d}{16}.
$$
Consequently, by a triangle inequality,
$$
\mathcal M_L(h_0+t_0)
\le
\mathcal M_L(h+t)
+\mathcal M_L(h-h_0)
+\mathcal M_L(t-t_0)
<2\gamma d.
$$
Thus $\mathcal H_\ell(h_0,t_0)$ occurs for every
$\ell=2,\ldots,L$. By construction, $(h_0,t_0)$ is admissible, thus
\eqref{eq:fixed-admissible-block-bound} gives
$$
\P\left(
	\mathcal G_\circ\bigcap
\bigcap_{\ell=2}^L\mathcal H_\ell(h_0,t_0)
\right)
\le
(C\gamma)^{(1-\eps)D_L}.
$$
	There are at most $\sum_{k\le \lfloor C_0d\log d\rfloor}\binom nk =
	\exp(o(d^2))$ choices of $\mathcal I$, so combining with \eqref{eq:net:bound:head:tail}, a union bound yields
$$
	\P\left(\mathcal E_{\rm bad}\bigcap \mathcal{G}_\circ \right)
\le
\exp(o(d^2))
\left(\frac C\gamma\right)^n
(C\gamma)^{(1-\eps)D_L}.
$$
Using $n/d^2\to\alpha$ and
$D_L/d^2\to(1-1/L)/2$, the exponent in the right-hand side is at most
$$
d^2\left(\left(
(1-\eps)\frac{1-1/L}{2}-\alpha
\right)\log\gamma+O(1)\right).
$$
This tends to $-\infty$ as $\gamma\downarrow0$. Choose $\gamma>0$
sufficiently small, and then choose $K\ge K_0$ sufficiently large that
$2\gamma\ge CK^{-1/2}$. Hence
	$\P(\mathcal E_{\rm bad}\bigcap \mathcal{G}_\circ)=o(1)$, completing the proof.
	\end{proof}

\section{Proof of the least-squares universality theorem}
\label{sec:shared-row-universality-proof}

This section provides the technical results used in the least-squares
universality argument. We first prove
Propositions~\ref{prop:quadratic-one-row-response}
and~\ref{prop:one-row-covariance-comparison}, which complete the proof of
Theorem~\ref{thm:positive-barrier-row-universality} given in
Section~\ref{subsec:proof:overview:unsat}. We then prove
Proposition~\ref{thm:approx-tensorization-feature-regularity} and
Proposition~\ref{thm:approx-fit}, the latter providing the approximate fit
used in the SAT argument. Finally, we prove
Lemma~\ref{lem:feature-ridge-removal}, which is used to remove the ridge
regularization in
Corollary~\ref{cor:matrix-scale-universality}~\textup{(ii)}.

\subsection{Leave-one-out localization and feature comparison}
\label{subsec:row-replacement-proof}

We now prove
Propositions~\ref{prop:quadratic-one-row-response}
and~\ref{prop:one-row-covariance-comparison}. Throughout this subsection,
we use the leave-one-out notation introduced in
Section~\ref{subsubsec:proof-positive-barrier-row-universality}, including
the objective $F^{(i)}$, its minimizer
$(\theta^{(i)},b^{(i)})$, and its Hessian $H^{(i)}$.

\begin{lemma}
	\label{lem:soft-leave-one-out-hessian-bounds}
	For every $i\in[n]$ and
	$u=(u^\theta,u^b)\in\mathsf V\oplus\R$,
	\begin{equation}
		\label{eq:generic-hessian-lower}
		\langle u,H^{(i)}u\rangle
		\ge
		2\beta\left(
		\lambda_\Sigma\|u^\theta\|_2^2+(u^b)^2
		\right)
		+\gamma\|u^\theta\|_{\mathcal B,\theta^{(i)}}^2.
	\end{equation}
	In particular, writing
	\[
		(H^{(i)})^{-1}
		=
		\begin{pmatrix}
			Q^{(i)} & v^{(i)}\\
			(v^{(i)})^* & c^{(i)}
		\end{pmatrix}
	\]
	with respect to $\mathsf V\oplus\R$, where
	$Q^{(i)}:\mathsf V\to\mathsf V$, $v^{(i)}\in\mathsf V$, and
	$c^{(i)}\in\R$, we have
	\begin{equation}
		\label{eq:soft-inverse-block-bounds}
		\|Q^{(i)}\|_{\op}
		\le C\beta^{-1}\lambda_\Sigma^{-1},
		\qquad
		\|v^{(i)}\|_2
		\le C\beta^{-1}\lambda_\Sigma^{-1/2},
		\qquad
		|c^{(i)}|
		\le C\beta^{-1},
	\end{equation}
	for a universal constant $C<\infty$.
\end{lemma}

\begin{proof}
Recall that $H^{(i)}=\mathrm D^2 F^{(i)}(\theta^{(i)}, b^{(i)})$. Since the Hessian of the least-squares term \[
\frac1{2n}\bigg(
		\sum_{j<i}\bigl(\langle W_j,\theta\rangle-b\bigr)^2
		+\sum_{j>i}\bigl(\langle G_j,\theta\rangle-b\bigr)^2
		\bigg)
        \]
        in the definition of $F^{(i)}$ in~\eqref{eq:def:objective:omit:i} is positive semidefinite, omitting this term from $H^{(i)}$ gives
	\[
		\langle u,H^{(i)}u\rangle
		\ge
		2\beta\left(
		\lambda_\Sigma\|u^\theta\|_2^2+(u^b)^2
		\right)
		+\gamma\,
		\mathrm D^2\mathcal B(\theta^{(i)})
		[u^\theta,u^\theta],
	\]
	which is \eqref{eq:generic-hessian-lower}. In particular, $H^{(i)}
		\succeq
		2\beta
		\begin{pmatrix}
			\lambda_\Sigma I_{\mathsf V}&0\\
			0&1
		\end{pmatrix},$ and hence
	\[
		(H^{(i)})^{-1}
		\preceq
		\frac1{2\beta}
		\begin{pmatrix}
			\lambda_\Sigma^{-1}I_{\mathsf V}&0\\
			0&1
		\end{pmatrix}.
	\]
	Thus, $	\|Q^{(i)}\|_{\op}
		\le\frac{1}{2\beta\lambda_\Sigma}$ and $0\le c^{(i)}\le\frac{1}{2\beta}$ hold. Finally, positivity of $(H^{(i)})^{-1}$ gives
	\[
		|\langle v^{(i)},u^\theta\rangle|^2
		\le
		\langle u^\theta,Q^{(i)}u^\theta\rangle\,c^{(i)}
		\le
		\frac{\|u^\theta\|_2^2}{4\beta^2\lambda_\Sigma}.
	\]
	Taking the supremum over $\|u^\theta\|_2=1$ proves the bound on
	$\|v^{(i)}\|_2$ and completes the proof.
\end{proof}
Recall from \eqref{eq:insertion-optimizer-update} that $u_i(z)\equiv (\theta_i(z),b_i(z))-(\theta^{(i)},b^{(i)})$ denotes
the change from the leave-one-out optimizer to the optimizer obtained
after inserting $z$. Define its unconstrained quadratic approximation
$\widehat u_i(z)=(\widehat u_i^{\theta}(z),\widehat u_i^b(z)) \in\mathsf V\oplus\R$ by
\[
	\widehat u_i(z)
	:=
	\operatorname*{arg\,min}_{u\in\mathsf V\oplus\R}
	\left\{
	\frac12\langle u,H^{(i)}u\rangle
	+\frac{\bigl(r_i(z)+\langle a(z),u\rangle\bigr)^2}{2n}
	\right\}
	=
	-\frac{r_i(z)}{n+\ell_i(z)}(H^{(i)})^{-1}a(z).
\]

\begin{lemma}
	\label{lem:self-concordant-insertion-control}
	Let $\mathcal B$ be a $\vartheta$-self-concordant barrier for some $\vartheta\geq 1$. There are universal constants $C,n_0>0$ such that the following
	holds. Fix $i\in[n]$ and condition on $\mathcal F^{(i)}$. For every
	$z\in\R^N$ satisfying
	\[
		n\gamma\ge n_0,
		\qquad
		|r_i(z)|\le(n\gamma)^{1/4},
	\]
	we have
	\[
		\theta^{(i)}+\widehat u_i^\theta(z)\in\mathcal D,
		\qquad
		\max_{u\in\{u_i(z),\widehat u_i(z)\}}
		\|u^\theta\|_{\mathcal B,\theta^{(i)}}^2
		\le
		C\frac{r_i(z)^2}{n\gamma},
	\]
	and
	\[
		\max_{u\in\{u_i(z),\widehat u_i(z)\}}
		|R_i(u)|
		\le
		C\gamma^{-1/2}n^{-3/2}|r_i(z)|^3.
	\]
\end{lemma}

\begin{proof}
	Using \eqref{eq:generic-hessian-lower} with $u=\widehat u_i(z)$ gives
	\[
		\gamma
		\|\widehat u_i^\theta(z)\|_{\mathcal B,\theta^{(i)}}^2\leq \big\langle \widehat{u}_i(z), H^{(i)}\widehat{u}_i(z)\big\rangle=\frac{\ell_i(z)}{(n+\ell_i(z))^2} r_i(z)^2
		\le
		\frac{r_i(z)^2}{4n}.
	\]
By optimality of $(\theta_i(z),b_i(z))$, and using $(\theta^{(i)},b^{(i)})$ as a competitor, we have
\[
	F^{(i)}(\theta_i(z),b_i(z))
	+\frac{\bigl(b_i(z)-\langle z,\theta_i(z)\rangle\bigr)^2}{2n}
	\le
	F^{(i)}(\theta^{(i)},b^{(i)})
	+\frac{r_i(z)^2}{2n}.
\]
Recall the definition of Bregman divergence in \eqref{eq:bregman-divergence-def}. Since $\mathrm{D}F^{(i)}(\theta^{(i)},b^{(i)})[u]=0$ for $u\in \mathsf{V}\oplus \R$ by first order optimiality, we thus have
\[
		\gamma D_{\mathcal B}
	\bigl(\theta^{(i)}+u_i^\theta(z),\theta^{(i)}\bigr)
	\le D_{F^{(i)}}\bigl(
	(\theta_i(z),b_i(z)),(\theta^{(i)},b^{(i)})
	\bigr)
	\le
	\frac{r_i(z)^2}{2n},
\]
where the first inequality holds since $(\theta,b)\mapsto F^{(i)}(\theta,b)-\gamma \mathcal{B}(\theta)$ is convex, and Bregman divergences are additive and nonnegative. Therefore,
\[
	\|\widehat u_i^\theta(z)\|_{\mathcal B,\theta^{(i)}}^2
	\le
	\frac{r_i(z)^2}{4n\gamma},
	\qquad
	D_{\mathcal B}
	\bigl(\theta^{(i)}+u_i^\theta(z),\theta^{(i)}\bigr)
	\le
	\frac{r_i(z)^2}{2n\gamma}.
\]
If $|r_i(z)|\le(n\gamma)^{1/4}$, both right-hand sides are at most
$C(n\gamma)^{-1/2}$. Thus, for $n\gamma\ge n_0$ with $n_0$
sufficiently large,
Lemma~\ref{lem:self-concordance-admissibility-consequences} gives
\[
	\theta^{(i)}+\widehat u_i^\theta(z)\in\mathcal D,
	\qquad
	\max_{u\in\{u_i(z),\widehat u_i(z)\}}
	\|u^\theta\|_{\mathcal B,\theta^{(i)}}^2
	\le
	C\frac{r_i(z)^2}{n\gamma}.
\]
Applying Lemma~\ref{lem:self-concordance-admissibility-consequences}~\textup{(ii)} for $u\in\{u_i(z),\widehat u_i(z)\}$, it follows that
\[
	|R_i(u)|
	\le
	C\gamma\|u^\theta\|_{\mathcal B,\theta^{(i)}}^3
	\le
	C\gamma^{-1/2}n^{-3/2}|r_i(z)|^3,\qquad u\in\{u_i(z),\widehat u_i(z)\}.
\]
where the first inequality holds by the expression of $R_i(u)$ in \eqref{eq:definition-R-i}. This concludes the proof.
\end{proof}
\begin{proof}[Proof of Proposition~\ref{prop:quadratic-one-row-response}]
	Fix $i\in[n]$. By definition of $\Lambda_{\rm ref}$, we can choose $\theta_{\rm ref}\in\mathcal D$ such that
	\[
		\lambda_\Sigma\|\theta_{\rm ref}\|_2^2
		+\gamma\mathcal B(\theta_{\rm ref})
		\le
		\frac{\Lambda_{\rm ref}}{1+\beta^{-1}}.
	\]
	Since $(\theta^{(i)},b^{(i)})$ minimizes $F^{(i)}$, comparison with
	$(\theta_{\rm ref},0)$ gives
	\begin{equation}\label{eq:deterministic:bound:S:i}
		\beta S^{(i)}
		\le
		\frac1{2n}\left(
		\sum_{j<i}\langle W_j,\theta_{\rm ref}\rangle^2
		+\sum_{j>i}\langle G_j,\theta_{\rm ref}\rangle^2
		\right)
		+\beta\lambda_\Sigma\|\theta_{\rm ref}\|_2^2
		+\gamma\mathcal B(\theta_{\rm ref}).
	\end{equation}
    Using the elementary inequality $(n^{-1}\sum_{j\neq i} a_j)^2\leq n^{-1}\sum_{j\neq i} a_j^2$, the second moment of the first term in the right-hand side is at most
\begin{align*}
	\frac1{4n}\left(
	\sum_{j<i}\E\langle W_j,\theta_{\rm ref}\rangle^4
	+\sum_{j>i}\E\langle G_j,\theta_{\rm ref}\rangle^4
	\right)
	\le
	C\lambda_\Sigma^2\|\theta_{\rm ref}\|_2^4,
\end{align*}
The last inequality follows from $\fone$, which is assumed for $W_i$; the Gaussian features, $G_j$'s, satisfy $\fone$ with
$K_{\rm feat}=3^{1/4}$ by \eqref{eq:gaussian-feature-regularity}.  Squaring
\eqref{eq:deterministic:bound:S:i}, taking expectations, and using
$(a+b)^2\le2(a^2+b^2)$ now gives
\[
	\beta^2 \E(S^{(i)})^2
	\le
	C(1+\beta^2)
	\left(
	\lambda_\Sigma\|\theta_{\rm ref}\|_2^2
	+\gamma\mathcal B(\theta_{\rm ref})
	\right)^2.
\]
Dividing by $\beta^2$, using
$1+\beta^{-2}\le(1+\beta^{-1})^2$, and recalling the choice of
$\theta_{\rm ref}$ yields $\E(S^{(i)})^2\le C\Lambda_{\rm ref}^2$, which proves \eqref{eq:integrated-leave-one-out-scale}.

	We next condition on $\mathcal F^{(i)}$ and let
	$Z_i\in\{W_i,G_i\}$. Then $Z_i$ is independent of
	$\mathcal F^{(i)}$, whereas
	$(\theta^{(i)},b^{(i)})$ is fixed under the conditioning. By
	Minkowski's inequality in $L^4$, 
	\begin{align*}
		\bigl(\E_i|r_i(Z_i)|^4\bigr)^{1/4}
		\le
		|b^{(i)}|
		+\bigl(\E_i|\langle Z_i,\theta^{(i)}\rangle|^4\bigr)^{1/4}\le
		|b^{(i)}|
		+C\sqrt{\lambda_\Sigma}\|\theta^{(i)}\|_2
		\le
		C(S^{(i)})^{1/2},
	\end{align*}
    where the second inequality uses $\fone$ and the last inequality holds by Cauchy Schwartz. This proves~\eqref{eq:integrated-residual-moments}.

It remains to prove \eqref{eq:soft-one-row-update}. Assume that
$n\gamma\ge n_0$, where $n_0$ is as in
Lemma~\ref{lem:self-concordant-insertion-control}, and abbreviate
\[
	r_i:=r_i(Z_i),\qquad
	\ell_i:=\ell_i(Z_i),\qquad
	\Delta_i:=\Delta_i(Z_i),\qquad
	\widehat u_i:=\widehat u_i(Z_i),\qquad
	u_i:=u_i(Z_i).
\]
On the event $\{|r_i|\le(n\gamma)^{1/4}\}$,
Lemma~\ref{lem:self-concordant-insertion-control} shows that
$\theta^{(i)}+\widehat u_i^\theta\in\mathcal D$ and $|R_i(u)|
	\le
	C\gamma^{-1/2}n^{-3/2}|r_i|^3$ for both $u\in \{u_i,\widehat u_i\}$. Using $u_i$ as a competitor for the quadratic problem and
$\widehat u_i$ as a competitor for the exact problem gives
\[
	R_i(u_i)
	\le
	\Delta_i-\frac{r_i^2}{2(n+\ell_i)}
	\le
	R_i(\widehat u_i).
\]
Thus,
\[
	\E_i\left[
	\left|
	\Delta_i-\frac{r_i^2}{2(n+\ell_i)}
	\right|
	\mathbf 1_{\{|r_i|\le(n\gamma)^{1/4}\}}
	\right]
	\le
	C\gamma^{-1/2}n^{-3/2}\E_i|r_i|^3
	\le
	C\gamma^{-1/2}n^{-3/2}(S^{(i)})^{3/2},
\]
where the last inequality follows from Hölder's inequality and
\eqref{eq:integrated-residual-moments}. On the complementary event,
\eqref{eq:elementary-insertion-bound} gives
$\Delta_i\in[0,r_i^2/(2n)]$. Since $\ell_i\ge0$, we also have
$r_i^2/[2(n+\ell_i)]\in[0,r_i^2/(2n)]$. Consequently,
\[
	\E_i\left[
	\left|
	\Delta_i-\frac{r_i^2}{2(n+\ell_i)}
	\right|
	\mathbf 1_{\{|r_i|>(n\gamma)^{1/4}\}}
	\right]
	\le
	\frac{\E_i r_i^4}{2n(n\gamma)^{1/2}}
	\le
	C\gamma^{-1/2}n^{-3/2}(S^{(i)})^2.
\]
Combining the two displays and using
$(S^{(i)})^{3/2}\le1+(S^{(i)})^2$ proves
\eqref{eq:soft-one-row-update}.
\end{proof}

\begin{proof}[Proof of
	Proposition~\ref{prop:one-row-covariance-comparison}]
	For $Z_i\in\{W_i,G_i\}$, the inverse-Hessian block decomposition in Lemma~\ref{lem:soft-leave-one-out-hessian-bounds} gives
	\[
	\ell_i(Z_i)
	=\langle P_{\mathsf V}Z_i,Q^{(i)}P_{\mathsf V}Z_i\rangle
	-2\langle P_{\mathsf V}Z_i,v^{(i)}\rangle+c^{(i)}.
	\]
  Since $W_i$ and $G_i$ are independent of $\mathcal F^{(i)}$, centered,
and have covariance $\Sigma$,
\[
	\E_i\ell_i(W_i)
	=\E_i\ell_i(G_i)
	=\operatorname{tr}
	\bigl(Q^{(i)}P_{\mathsf V}\Sigma P_{\mathsf V}\bigr)
	+c^{(i)}
	=:\bar\ell_i.
\]
In particular, $\bar{\ell}_i\geq 0$.

Let $\Var_i$ denote conditional variance given $\mathcal F^{(i)}$.
By \eqref{eq:soft-inverse-block-bounds},
	$\|Q^{(i)}\|_{F}\le C\beta^{-1}\lambda_{\Sigma}^{-1}\sqrt{\dim\mathsf V}$
	and $\|v^{(i)}\|_2\le C\beta^{-1}\lambda_{\Sigma}^{-1/2}$.
Applying $\ftwo$ with $T=P_{\mathsf V}Q^{(i)}P_{\mathsf V}$ for $W_i$,
and using \eqref{eq:gaussian-feature-regularity} for $G_i$, gives, for either $Z_i\in\{W_i,G_i\}$,
\[
	\Var_i\left(
	\langle P_{\mathsf V}Z_i,Q^{(i)}P_{\mathsf V}Z_i\rangle
	\right)
	\le
	C(\tau^2+\lambda_\Sigma^2)\|Q^{(i)}\|_F^2
	\le
	C\beta^{-2}(\dim\mathsf V)
	\left(\frac{\tau}{\lambda_\Sigma}+1\right)^2.
\]
Moreover, 
\[
	\Var_i\left(\langle P_{\mathsf V}Z_i,v^{(i)}\rangle\right)= \langle v^{(i)}, \Sigma v^{(i)}\rangle \leq \lambda_{\Sigma}\|v^{(i)}\|_2^2\le C\beta^{-2}.
\]
Therefore, 
	\[
	\E_i\left[|\ell_i(Z_i)-\bar\ell_i|^2\right]
	=
	\Var_i(\ell_i(Z_i))
	\le
	C\beta^{-2}(\dim\mathsf V)
	\left(\frac{\tau}{\lambda_\Sigma}+1\right)^2.
\]
Finally, for $Z_i\in\{W_i,G_i\}$,
\[
	\E_i\left[r_i(Z_i)^2\right]
	\equiv \E_i\left[\big(b^{(i)}- \langle Z_i,\theta^{(i)}\rangle\big)^2\right]=
	(b^{(i)})^2
	+\langle\theta^{(i)},\Sigma\theta^{(i)}\rangle.
\]
The two conditional expectations therefore agree, and hence $\E_i\frac{r_i(W_i)^2}{2(n+\bar{\ell}_i)}=\E_i\frac{r_i(G_i)^2}{2(n+\bar{\ell}_i)}$.
\end{proof}

\subsection{Quadratic-feature regularity and approximate fitting}
\label{subsec:approx-tensorization-regularity}
This section proves Proposition~\ref{thm:approx-tensorization-feature-regularity} and Proposition~\ref{thm:approx-fit}. 

Approximate tensorization of variance,
\eqref{eq:approximate-variance-tensorization}, is the Poincar\'{e} inequality
for the Glauber dynamics. It is well-known that Poinca\'{e} inequalities imply $L^p$ moment estimates; see e.g.~\cite[Section 3]{adamczak2022modified}. Using standard moment estimates, we prove the quadratic-feature regularity, Proposition~\ref{thm:approx-tensorization-feature-regularity}.

\begin{proof}[Proof of
	Proposition~\ref{thm:approx-tensorization-feature-regularity}]
	Throughout, let $C_1, C_1'>0$ denote constants depending only on
	$C,M$, and $\kappa$, whose value may change from line to line. For
	each $j\in[d]$, let $x^{[j]}$ be obtained by conditionally resampling
	the $j$th coordinate of $x$. Then $x^{[j]}\stackrel{d}{=}x$, and
	conditional exchangeability gives
	$\E(f(x)-f(x^{[j]}))^2
	=2\E\Var(f(x)\mid x_{-j})$. Hence, approximation tensorixation of variance~\eqref{eq:approximate-variance-tensorization} is equivalent to
	\begin{equation}
		\label{eq:conditional-resampling-poincare}
		\Var(f(x))
		\le
		\frac C2\sum_{j=1}^d
		\E\bigg[\Big(f(x)-f\big(x^{[j]}\big)\Big)^2\bigg]
	\end{equation}
	for every $f$ such that $\E[f(x)^2]<\infty$. We first claim the following consequence of ~\eqref{eq:conditional-resampling-poincare}: for $p\in \{2,4,8\}$, and any function $g$ such that $\E[|g(x)|^p]<\infty$, let
\[
	Z\equiv Z_g:=g(x),
	\qquad
	Z_j\equiv Z_{g,j}:=g\bigl(x^{[j]}\bigr),
	\qquad
	\mathfrak D_p(g)
	:=
	\bigg(
	\sum_{j=1}^d\|Z-Z_j\|_{L^p}^2
	\bigg)^{1/2}.
\]
Then, for $p\in \{2,4,8\}$,
\begin{equation}
	\label{eq:resampling-lp-moment-bound}
	\|Z-\E Z\|_{L^p}
	\le C_1\mathfrak D_p(g).
\end{equation}
We first prove \eqref{eq:resampling-lp-moment-bound}.  Since replacing $g$ by $g-\E Z$ does not change $Z-Z_j$, we may assume w.l.o.g. that $\E Z=0$.  For $p=2$, \eqref{eq:resampling-lp-moment-bound} is exactly~\eqref{eq:conditional-resampling-poincare} with $C_1=\sqrt{C/2}$.

Now take $p\in\{4,8\}$. We will apply
\eqref{eq:conditional-resampling-poincare} to $f=|g|^{p/2}$. The elementary
inequality
\[
	\bigl||u|^{p/2}-|v|^{p/2}\bigr|
	\le
	\frac{p}{2}|u-v|\bigl(|u|^{p/2-1}+|v|^{p/2-1}\bigr)\leq \frac{p}{\sqrt{2}}|u-v|\bigl(|u|^{p-2}+|v|^{p-2}\bigr)^{1/2}
\]
gives
\begin{align*}
	\sum_{j=1}^d
	\E\bigg[\left(|Z|^{p/2}-|Z_j|^{p/2}\right)^2\bigg]
	&\le
	C_p\sum_{j=1}^d
	\E\Big[
	|Z-Z_j|^2
	\big(|Z|^{p-2}+|Z_j|^{p-2}\big)
	\Big]
    \leq C_p' \mathfrak{D}_p(g)^2 \|Z\|_{L^p}^{p-2},
\end{align*}
where we used H\"{o}lder's inequality  $\E[|X|^2|Y|^{p-2}]\leq \|X\|_{L^p}^2 \|Y\|_{L^p}^{p-2}$ and $\|Z_j\|_{L^p}=\|Z\|_{L^p}$ in the last step. Applying \eqref{eq:conditional-resampling-poincare} to
$f=|g|^{p/2}$ and using the preceding estimate, we obtain
\begin{equation}
	\label{eq:resampling-lp-recursion}
\|Z\|_{L^p}^p
	=
	\Var(|Z|^{p/2})
	+\|Z\|_{L^{p/2}}^p
	\le
	C_1\mathfrak D_p(g)^2\|Z\|_{L^p}^{p-2}
	+\|Z\|_{L^{p/2}}^p.
\end{equation}
	For $p=4$, the $p=2$ estimate and
	$\mathfrak D_2(g)\le\mathfrak D_4(g)$ give
	\[
		\|Z\|_{L^4}^4
		\le
		C_1\mathfrak D_4(g)^2\|Z\|_{L^4}^2
		+C_1\mathfrak D_4(g)^4.
	\]
	Applying Young's inequality to the first term on the right gives $	C_1\mathfrak D_4(g)^2\|Z\|_{L^4}^2
	\le
	\frac12\|Z\|_{L^4}^4
	+C_1'\mathfrak D_4(g)^4$, thus we conclude
$\|Z\|_{L^4}\le C_1\mathfrak D_4(g)$.  Taking $p=8$ in
	\eqref{eq:resampling-lp-recursion}, and using this $L^4$ estimate and
	$\mathfrak D_4(g)\le\mathfrak D_8(g)$, gives
	\[
		\|Z\|_{L^8}^8
		\le
		C_1\mathfrak D_8(g)^2\|Z\|_{L^8}^6
		+C_1\mathfrak D_8(g)^8.
	\]
	Another application of Young's inequality yields
	$\|Z\|_{L^8}\le C_1\mathfrak D_8(g)$. This finishes the proof of~\eqref{eq:resampling-lp-moment-bound}.

We next verify $\fone$. For $a\in\R^d$, set
$g_a(x):=\langle a,x\rangle$, and write
$\widetilde x_j$ for the resampled $j$'th coordinate of $x^{[j]}$. Since
\[
	g_a(x)-g_a\bigl(x^{[j]}\bigr)
	=
	a_j(x_j-\widetilde x_j),
	\qquad
	\|x_j-\widetilde x_j\|_{L^8}\le2M^{1/8},
\]
we have $\mathfrak D_8(g_a)\le C_1\|a\|_2$. Moreover,
$\E g_a(x)=0$. Applying
\eqref{eq:resampling-lp-moment-bound} to $g_a$ and $p=8$ therefore gives
\begin{equation}
	\label{eq:approx-tensorization-linear-l8}
	\|\langle a,x\rangle\|_{L^8}
	\le C_1\|a\|_2.
\end{equation}
For $A\in\Ssym^d$, set $g_A(x):=x^\top Ax-\Tr A$. Then
$\E g_A(x)=0$ and
\[
	g_A(x)-g_A\bigl(x^{[j]}\bigr)
	=
	(x_j-\widetilde x_j)
	\bigg(
	2\sum_{\ell\ne j}A_{j\ell}x_\ell
	+A_{jj}(x_j+\widetilde x_j)
	\bigg).
\]
Thus, H\"older's inequality,
\eqref{eq:approx-tensorization-linear-l8}, and the eighth-moment
assumption imply
\[
	\left\|g_A(x)-g_A\bigl(x^{[j]}\bigr)\right\|_{L^4}
	\le
	C_1\bigg(\sum_{\ell=1}^dA_{j\ell}^2\bigg)^{1/2}.
\]
Hence, $\mathfrak D_4(g_A)\le C_1\|A\|_F$. Applying \eqref{eq:resampling-lp-moment-bound} to $g_A$ and $p=4$ therefore gives
\begin{equation}
	\label{eq:approx-tensorization-quadratic-l4}
	\|x^\top Ax-\Tr A\|_{L^4}
	\le C_1\|A\|_F.
\end{equation}
Since
$\langle W,A\rangle_F=(x^\top Ax-\Tr A)/\sqrt d$ and
$\|\Sigma_\kappa\|_{\op}\asymp_\kappa d^{-1}$, this proves $\fone$.

Finally, we verify $\ftwo$. Let $(S_k)_k$ be an orthonormal basis of $(\Ssym^d, \langle \cdot,\cdot\rangle)$. For any linear map
$B:\Ssym^d\to\Ssym^d$,
\[
	\|BW\|_F^2
	=
	\sum_k\langle W,B^*S_k\rangle^2,
\]
where $B^*: \Ssym^d\to \Ssym^d$ denotes the adjoint of $B$. Therefore, Minkowski's inequality and
\eqref{eq:approx-tensorization-quadratic-l4} give

\begin{equation}
	\label{eq:projected-quadratic-feature-fourth-moment}
	\bigl(\E\|BW\|_F^4\bigr)^{1/2}
	\le \sum_k\|\langle W,B^*S_k\rangle^2\|_{L^2}=
	\sum_k\|\langle W,B^*S_k\rangle\|_{L^4}^2
	\le
	\frac{C_1}{d}\|B\|_F^2.
\end{equation}
For each $j\in[d]$, set
\[
	W^{[j]}
	:=
	\frac{x^{[j]}\bigl(x^{[j]}\bigr)^\top-I_d}{\sqrt d}.
\]
Let $T:\Ssym^d\to\Ssym^d$ be self-adjoint, and let $\Pi_j$ be the
Frobenius-orthogonal projection onto matrices supported on row or
column $j$. Since
$\Pi_j(W-W^{[j]})=W-W^{[j]}$, we have
\[
	\left\langle
	W-W^{[j]},\Pi_jT(W+W^{[j]})
	\right\rangle
	=
	\langle W,TW\rangle
	-\langle W^{[j]},TW^{[j]}\rangle.
\]
Applying \eqref{eq:conditional-resampling-poincare} to
$x\mapsto\langle W,TW\rangle$ therefore gives
\[
	\Var\bigl(\langle W,TW\rangle_F\bigr)
	\le
	\frac C2\sum_{j=1}^d
	\E\left[
	\left\langle
	W-W^{[j]},\Pi_jT(W+W^{[j]})
	\right\rangle^2
	\right].
\]
A direct calculation using the eighth-moment assumption, together with
$W^{[j]}\stackrel d=W$ and
\eqref{eq:projected-quadratic-feature-fourth-moment}, gives
\[
	\E\|W-W^{[j]}\|_F^4\le C_1,
	\qquad
	\E\|\Pi_jT(W+W^{[j]})\|_F^4
	\le\frac{C_1}{d^2}\|\Pi_jT\|_F^4.
\]
Thus, Cauchy Schwarz and
$\sum_j\Pi_j\preceq2I_{\Ssym^d}$ yield
\[
	\Var\bigl(\langle W,TW\rangle_F\bigr)
	\le\frac{C_1}{d}\|T\|_F^2.
\]
Hence $\ftwo$ holds with
$\tau_d\le C_{1}d^{-1/2}$.
\end{proof}

We finish this subsection by deriving the approximate fit used in the SAT
argument.

\begin{proof}[Proof of Proposition~\ref{thm:approx-fit}]
	Write $W(X)=(W_i)_{i\le n}$, where $W_i:=\frac{x_ix_i^\top-I_d}{\sqrt d}$. Since $x_i$'s have independent coordinates satisfying~\eqref{eq:independent:coordinate:distribution:assumption}, we have $\Cov(\operatornorm{vec}W_i)=\Sigma_\kappa$. Moreover,
	product measures satisfy approximate tensorization of variance~\eqref{eq:approximate-variance-tensorization} with $C=1$, and the uniform subgaussian bound gives a uniform
	eighth-moment bound. Therefore,
	Proposition~\ref{thm:approx-tensorization-feature-regularity} gives a
	dimension-free regularity constant and $\tau_d=O(d^{-1/2})$.

	Let $G=(G_i)_{i\le n}$ be i.i.d. Gaussian matrices with covariance $\Sigma_{\kappa}$, and set
	$\lambda_d:=\|\Sigma_\kappa\|_{\op}$, so that
	$\lambda_d\asymp_\kappa d^{-1}$.
	Proposition~\ref{thm:gaussian-master}~\textup{(i)} gives constants
	$0<m_0<1<M_0<\infty$ such that, with probability tending to one,
	there exists $R_G\in\mathcal R_d$ satisfying
	\[
		m_0I_d\preceq R_G\preceq M_0I_d,
		\qquad
		\mathcal A_{G,\kappa}R_G=0.
	\]
	Set $m:=m_0/2$ and $M:=2M_0$. $\mathcal A_{G,\kappa}R_G=0$ implies that $\langle G_i, R_G\rangle$ is constant over $i$, so write its common value as
	\[
		b_G:=\langle \overline{G},R_G\rangle_F,\quad\textnormal{where}\quad \overline{G}:=\frac{1}{n}\sum_{i=1}^{n}G_i.
	\]
    Since the upper-triangular entries of $G_i$ are independent with $N(0,(\kappa-1)/d)$ on the diagonal and $N(0,1/d)$ on the off-diagonal~\eqref{eq:def:gaussian:features}, $\overline{G}$ have independent upper-triangular entries with variance $O(1/\sqrt{nd})$. Thus, standard operator norm bounds~\cite[Theorem 4.4.5]{vershynin2018highdimensional} gives $\|\overline G\|_{\op}=O_\P(d^{-1})$ for $n\asymp d^2$. Since
	$R_G\succeq0$ and $\Tr R_G=d$,
	\[
		|b_G|
		=
		|\langle\overline G,R_G\rangle_F|
		\le d\|\overline G\|_{\op}
		=O_\P(1).
	\]
Since $R_G\preceq M_0I_d$ and $\Tr R_G=d$, $\|R_G\|_F^2\leq M_0 d$. Thus, $\lambda_d\|R_G\|_F^2=O_\P(1)$, since
$\lambda_d\asymp_\kappa d^{-1}$. Take
\[
\beta_d=(\log d/d)^{1/9}.
\]
Using $\tau_d=O(d^{-1/2})$,
Corollary~\ref{cor:matrix-scale-universality}\textup{(i)}, applied to
$\mathcal R_d^{(m,M)}$, gives
\[
\begin{aligned}
\left|
\Phi_{W(X)}(\mathcal R_d^{(m,M)};\beta_d)
-\Phi_G(\mathcal R_d^{(m,M)};\beta_d)
\right|
=
O_\P\left(
\beta_d^{-2}
\left[d^{-1/2}
+\left(\frac{\log d}{d}\right)^{1/3}\right]
\right)=O_\P(\beta_d).
\end{aligned}
\]
	The Gaussian fit lies in $\mathcal R_d^{(m,M)}$, and its empirical loss
	is zero. Testing the Gaussian ridge objective at $(R_G,b_G)$ therefore 
	gives $\Phi_G(\mathcal R_d^{(m,M)};\beta_d)
		=O_\P(\beta_d).$ Consequently,
	\begin{equation}
		\label{eq:coordinate-small-ridge-proxy}
		\Phi_{W(X)}(\mathcal R_d^{(m,M)};\beta_d)
		=O_\P(\beta_d).
	\end{equation}
	By continuity, this infimum is unchanged on
	$\overline{\mathcal R_d^{(m,M)}}$, where the overline denotes closure.
	Let $(\widehat R_d,\widehat b_d)$ be a minimizer on
	$\overline{\mathcal R_d^{(m,M)}}\times\R$. In particular, $mI_d\preceq\widehat R_d\preceq MI_d$. Since minimizing over the intercept centers the row values, while the
	ridge terms are nonnegative,
	\[
		\frac1{2n}\|\mathcal A_X\widehat R_d\|_2^2
		\le
		\Phi_{W(X)}(\mathcal R_d^{(m,M)};\beta_d).
	\]
	Hence \eqref{eq:coordinate-small-ridge-proxy} gives
	\[
		\frac1n\|\mathcal A_X\widehat R_d\|_2^2
		=
		O_\P(\beta_d)
		=
		o_\P\left(
		\left(\frac{\log d}{d}\right)^{1/10}
		\right).
	\]
    It remains to prove the $\ell_{\infty}$ bound. For each $i$, let $(R^{(i)},b^{(i)})$ minimize the same ridge objective
	over $\overline{\mathcal R_d^{(m,M)}}\times\R$, but with the $i$th loss
	term omitted and the normalization $1/(2n)$ unchanged. Its objective
	value is at most the full objective evaluated at
	$(\widehat R_d,\widehat b_d)$. Thus
	\eqref{eq:coordinate-small-ridge-proxy} gives, simultaneously for all
	$i$,
	\[
		\beta_d(b^{(i)})^2
		\le
		\Phi_{W(X)}(\mathcal R_d^{(m,M)};\beta_d)
		=O_\P(\beta_d),
	\]
	and therefore $\max_{i\le n}|b^{(i)}|=O_\P(1)$. Comparing the full and leave-one-out optimality inequalities gives
	\[
		|\langle W_i,\widehat R_d\rangle_F-\widehat b_d|
		\le
		|\langle W_i,R^{(i)}\rangle_F-b^{(i)}|.
	\]
	Let
	\[
		\mathcal F_X^{(i)}:=\sigma(x_j:j\ne i).
	\]
	Conditionally on $\mathcal F_X^{(i)}$, the pair
	$(R^{(i)},b^{(i)})$ is fixed and independent of $x_i$. On
	$\{|b^{(i)}|\le B\}$, the bounds
	\[
		\|R^{(i)}\|_F\le M\sqrt d,
		\qquad
		\|R^{(i)}\|_{\op}\le M,
	\]
	together with Hanson--Wright inequality in Lemma~\ref{lem:hanson-wright}, give, for $s\ge1$,
	\[
		\P\left\{
		|\langle W_i,R^{(i)}\rangle_F-b^{(i)}|>s,\
		|b^{(i)}|\le B
		\,\middle|\,
		\mathcal F_X^{(i)}
		\right\}
		\le Ce^{-cs}.
	\]
	The constants are uniform in $i$, since the coordinates of $x_i$ are
	independent, have variance one, and satisfy the common subgaussian
	bound $K_x$.

	Choose $B$ so that $\max_i|b^{(i)}|\le B$ with arbitrarily high
	probability, and then take $s=C_0\log d$. A union bound over
	$n=O(d^2)$ rows gives
	\[
		\max_{i\le n}
		|\langle W_i,\widehat R_d\rangle_F-\widehat b_d|
		=O_\P(\log d).
	\]
	Finally,
	\[
		\mathcal A_X\widehat R_d
		=
		P_{\one^\perp}
		\bigl(
		\langle W_i,\widehat R_d\rangle_F-\widehat b_d
		\bigr)_{i\le n},
	\]
	and $\|P_{\one^\perp}v\|_\infty\le2\|v\|_\infty$. Hence
	\[
		\|\mathcal A_X\widehat R_d\|_\infty
		=
		O_\P(\log d)
		=
		o_\P((\log d)^2),
	\]
	which completes the proof.
\end{proof}

\subsection{Removal of ridge regularization}
\label{subsec:feature-ridge-removal-proof}
This section proves Lemma~\ref{lem:feature-ridge-removal} using a Rademacher
complexity argument.

\begin{lemma}
	\label{lem:feature-weighted-operator-control}
	Assume the hypotheses of
	Corollary~\ref{cor:matrix-scale-universality}~\textup{(ii)}, and let
	$\varepsilon_1,\ldots,\varepsilon_n$ be independent Rademacher signs.
	Then
	\begin{equation}
		\label{eq:feature-signed-operator-control}
		\E\left\|\frac1n\sum_{i=1}^n\varepsilon_i W_i\right\|_{\op}
		\le C\frac{\sqrt{\log d}}d.
	\end{equation}
\end{lemma}

\begin{proof}
	Set $Y=\|W\|_F^2$. Covariance matching and $\mathrm{(F2)}$ give
	\[
		\E Y=\Tr\Sigma_\kappa=O_\kappa(d),
		\qquad
		\Var(Y)
		\le\tau_d^2\dim(\Ssym^d)
		\le C\frac{d^2}{(\log d)^6}.
	\]
	Moreover, for $v\in S^{d-1}$ and
	$S_{v,j}=(ve_j^\top+e_jv^\top)/2$,
	\[
		v^\top\E W^2v
		=
		\sum_{j=1}^d\E\langle W,S_{v,j}\rangle_F^2
		\le
		\|\Sigma_\kappa\|_{\op}
		\sum_{j=1}^d\|S_{v,j}\|_F^2
		\le C_\kappa,
	\]
	where we used
	$\sum_j\|S_{v,j}\|_F^2=(d+1)/2$ and
	$\|\Sigma_\kappa\|_{\op}=O_\kappa(d^{-1})$. Thus
	$\|\E W^2\|_{\op}\le C_\kappa$. Set 
	\[
		\widetilde W_i
		:=
		\varepsilon_iW_i
		\mathbf 1_{\{\|W_i\|_F\le L_d\}}, \qquad L_d:=\frac{d}{(\log d)^{7/4}}
	\]
	These matrices are independent and centered, with
	$\|\widetilde W_i\|_{\op}\le L_d$ and
	\[
		\left\|\sum_{i=1}^n\E\widetilde W_i^2\right\|_{\op}
		\le C_\kappa n.
	\]
	Matrix Bernstein inequality~\cite[Theorem~5.4.1]{vershynin2018highdimensional}, followed by
	integration of its tail bound, yields
	\[
		\E\left\|\frac1n\sum_{i=1}^n\widetilde W_i\right\|_{\op}
		\le
		C\left(
		\sqrt{\frac{\log d}{n}}+\frac{L_d\log d}{n}
		\right)
		\le C\frac{\sqrt{\log d}}d.
	\]
	Since $\E Y=o(L_d^2)$, on $\{Y>L_d^2\}$ and for all large $d$, $	\sqrt Y\le C\frac{(Y-\E Y)^2}{L_d^3}.$ Consequently,
	\[
		\E\!\left[
		\|W\|_F\mathbf 1_{\{\|W\|_F>L_d\}}
		\right]
		\le
		C\frac{\Var(Y)}{L_d^3}
		\le\frac{C}{d(\log d)^{3/4}}.
	\]
	Decomposing each $W_i$ into its truncated and tail parts and using the
	triangle inequality now gives
	\[
		\E\left\|\frac1n\sum_{i=1}^n\varepsilon_iW_i\right\|_{\op}
		\le
		C\frac{\sqrt{\log d}}d
		+\frac{C}{d(\log d)^{3/4}}
		\le C\frac{\sqrt{\log d}}d,
	\]
	which proves \eqref{eq:feature-signed-operator-control}.
\end{proof}

\begin{proof}[Proof of Lemma~\ref{lem:feature-ridge-removal}]
	Write $\lambda_d=\|\Sigma_\kappa\|_{\op}$.  The proof of
	Lemma~\ref{lem:feature-weighted-operator-control} applies separately to
	the feature and matching Gaussian arrays; for the latter,
	$\tau_d=\sqrt2\lambda_d=O(d^{-1})$.  Hence, for
	$Z\in\{W,G\}$ and independent Rademacher signs,
	\begin{equation}
		\label{eq:feature-common-signed-operator-control}
		\E\left\|\frac1n\sum_{i=1}^n\varepsilon_iZ_i\right\|_{\op}
		\le C\frac{\sqrt{\log d}}d.
	\end{equation}
	For $R\in\mathcal R_d$ and $b\in\R$, define the population affine
	scale
	\[
		\sigma_{R,b}^2
		:=\E\bigl(\langle Z,R\rangle_F-b\bigr)^2
		=\mathcal C_\kappa(R,R)+b^2.
	\]
	This quantity is the same for $Z=W$ and $Z=G$.  Writing
	$R_{\diag}$ for the diagonal part of $R$ and
	$R_{\off}:=R-R_{\diag}$, we have
	\begin{equation}
		\label{eq:feature-covariance-coercivity}
		\mathcal C_\kappa(R,R)
		=\frac1d\left(
		2\|R_{\off}\|_F^2
		+(\kappa-1)\|R_{\diag}\|_F^2
		\right)
		\ge \frac{\min\{2,\kappa-1\}}d\|R\|_F^2.
	\end{equation}
	The explicit form of $\Sigma_\kappa$ gives
	$\lambda_d\asymp_\kappa d^{-1}$.  Therefore $\mathrm{(F1)}$ and
	\eqref{eq:feature-covariance-coercivity} give
	\[
		\|\langle Z,R\rangle_F-b\|_{L^4}
		\le K_{\rm feat}\sqrt{\lambda_d}\|R\|_F+|b|
		\le C\sigma_{R,b}.
	\]
	Fix a constant $A>0$, to be chosen below, and consider the (pointwise separable) class of functions
	\[
	\mathcal F_{Z,A}
		:=
		\left\{
		z\mapsto\frac{\langle z,R\rangle_F-b}{\sigma_{R,b}}:
		R\in\mathcal R_d,\ b\in\R,\ 
		\sigma_{R,b}^2\ge A\log d
		\right\}.
	\]
    Every
	$f\in\mathcal F_{Z,A}$ satisfies
	$\E f(Z)^2=1$ and $\E f(Z)^4\le C_4$, for a constant $C_4$
	independent of $d$.  Paley--Zygmund, applied to $f(Z)^2$, therefore
	gives
	\[
		\inf_{f\in\mathcal F_{Z,A}}
		\P\{|f(Z)|\ge1/2\}
		\ge \frac{9}{16C_4}=:q_0>0.
	\]
	Put $u_0=1/2$, so the small-ball function of the class satisfies
	$Q_{\mathcal F_{Z,A}}(u_0)\ge q_0$.

	Let $\overline\varepsilon=n^{-1}\sum_i\varepsilon_i$.  Since
	$\|R\|_*=\Tr R=d$ on $\mathcal R_d$, nuclear/operator duality and
	\eqref{eq:feature-common-signed-operator-control} give
	\begin{align*}
		\E\sup_{f\in\mathcal F_{Z,A}}
		\left|\frac1n\sum_{i=1}^n\varepsilon_i f(Z_i)\right|
		&\le
		\frac d{\sqrt{A\log d}}\,
		\E\left\|\frac1n\sum_{i=1}^n\varepsilon_iZ_i\right\|_{\op}
		+\E|\overline\varepsilon| \\
		&\le \frac C{\sqrt A}+\frac1{\sqrt n}.
	\end{align*}
	The left-hand side is by definition the \emph{Rademacher complexity}
	$R_n(\mathcal F_{Z,A})$ of the function class. 

    Now that we have a Rademacher complexity bound, we can obtain a lower bound on the squared loss for all sufficiently large $A$ by applying small-ball estimates from the literature. Following the notation of \cite[Theorem~2.1]{koltchinskii-mendelson2013smallest}, we do the following.
    Set $\tau=u_0/2$.  Choose $A$ so large
	that $C/\sqrt A\le\tau q_0/32$, and then take $d$ large enough that
	$n^{-1/2}\le\tau q_0/32$.  It follows that
	\[
		R_n(\mathcal F_{Z,A})
		\le\frac{\tau q_0}{16}
		\le\frac{\tau Q_{\mathcal F_{Z,A}}(2\tau)}{16}.
	\]
	This is exactly the assumption of the small-ball estimate of Koltchinskii and Mendelson
	\cite[Theorem~2.1]{koltchinskii-mendelson2013smallest}, which now gives, with
	probability at least $1-2\exp(-q_0^2n/8)$, the uniform lower bound
	\[
		\inf_{f\in\mathcal F_{Z,A}}\frac1n\sum_{i=1}^nf(Z_i)^2
		\ge\frac{\tau^2q_0}{2}.
	\]
	Thus, with $c_{\rm sb}:=\tau^2q_0/2>0$,
	\begin{equation}
		\label{eq:feature-affine-small-ball-localization}
		\frac1n\sum_{i=1}^n
		\bigl(\langle Z_i,R\rangle_F-b\bigr)^2
		\ge c_{\rm sb}\sigma_{R,b}^2
	\end{equation}
	simultaneously for all $R\in\mathcal R_d$ and $b\in\R$ satisfying
	$\sigma_{R,b}^2\ge A\log d$.  A union bound makes this event
	simultaneous for $Z=W,G$, with probability tending to one.

	It remains to apply this estimate only to the unregularized minimizers.
	Fix $B_{\rm I}>\kappa-1$.  Condition $\mathrm{(F1)}$ gives a
	uniform fourth-moment bound for $\langle Z_i,I_d\rangle_F$, while
	$\E\langle Z_i,I_d\rangle_F^2=\kappa-1$.  Chebyshev's inequality
	therefore gives, simultaneously for $Z=W,G$,
	\[
		\frac1n\sum_{i=1}^n\langle Z_i,I_d\rangle_F^2\le B_{\rm I}
	\]
	with probability tending to one.

	Work on the intersection of these events.  For $Z\in\{W,G\}$, let
	$(R_\star,b_\star)\in\mathcal R_d\times\R$ minimize the
	unregularized affine loss.  Such a minimizer exists because the bias can
	first be minimized explicitly and $\mathcal R_d$ is compact.  Comparison
	with $(I_d,0)$ gives
	\[
		\frac1n\sum_{i=1}^n
		\bigl(\langle Z_i,R_\star\rangle_F-b_\star\bigr)^2
		\le B_{\rm I}.
	\]
	If $\sigma_{R_\star,b_\star}^2\ge A\log d$, then
	\eqref{eq:feature-affine-small-ball-localization} would make the
	left-hand side at least $c_{\rm sb}A\log d>B_{\rm I}$ for all large
	$d$.  Hence
	$\sigma_{R_\star,b_\star}^2<A\log d$.  Using
	\eqref{eq:feature-covariance-coercivity} once more,
	\begin{equation}
		\label{eq:feature-minimizer-affine-localization}
		\lambda_d\|R_\star\|_F^2+b_\star^2
		\le C\log d.
	\end{equation}

	The domain $\mathcal R_d^\circ$ is dense in $\mathcal R_d$, and the
	objective defining $\Phi_Z(\mathcal R_d^\circ;\beta)$ is continuous on
	$\mathcal R_d$.
	Thus its infimum is unchanged on passing to the closure and may be tested
	at $(R_\star,b_\star)$.  This gives
	$\Phi_Z(\mathcal R_d^\circ;\beta)-\Phi_Z(\mathcal R_d^\circ)
	\le C\beta\log d$, while the reverse inequality
	follows because the ridge terms are nonnegative.  The argument holds for $Z\in\{W,G\}$ and every $\beta>0$, completing the
	proof.
\end{proof}

\section{The Gaussian feature model}
\label{sec:gaussian-anisotropy-threshold}

This section proves Proposition~\ref{thm:gaussian-master}, the Gaussian analog of the ellipsoid fitting. 
The proof is based on Gaussian process techniques and convex geometry. First, Lemma~\ref{lem:projected-anisotropic-gaussian-representation}
rewrites the problem in terms of an i.i.d. Gaussian matrix
acting on a rescaled positive-semidefinite cone. The Gaussian min--max theorem 
then reduces the analysis to an ``auxiliary'' problem which is much easier to analyze. Proposition~\ref{prop:shifted-cone-projections} then evaluates the auxiliary problem: using the semicircle law for GOE matrices, it reduces to a scalar variational problem which is explicitly solved in Section~\ref{subsec:auxiliary}.  For our exact fitting results, we need the existence of a well-conditioned fit in the Gaussian analog, which is shown in the last subsection.

Throughout, we let $N_d:=d(d+1)/2$ be the dimension of $\Ssym^d$.
For $R\in\Ssym^d$, let $R_{\diag}$ be its
diagonal part and $R_{\off}:=R-R_{\diag}$.  For $\kappa>1$, set
$$
c_\kappa:=\frac{\kappa-3}{2},
\qquad
T_\kappa R
:=R_{\off}+\sqrt{\frac{\kappa-1}{2}}\,R_{\diag},
\qquad
\mathbb S_\kappa^d:=T_\kappa(\Ssym_+^d).
$$
Thus $\mathbb S_\kappa^d$ is the image of the positive-semidefinite cone
under the diagonal rescaling determined by $\kappa$.  Since $T_\kappa$
is invertible, $\mathbb S_\kappa^d$ is a closed convex cone. With this notation we have
\begin{equation}\label{eq:gaussian-section-covariance}
	\mathcal C_\kappa(R,S)
	:=
	\frac1d\left(
	2\Tr(RS)+(\kappa-3)
	\langle R_{\diag},S_{\diag}\rangle_F
	\right)
	=\frac2d\langle T_\kappa R,T_\kappa S\rangle_F .
\end{equation}

Let $\mathsf G_d$ be a GOE matrix normalized so its diagonal entries are i.i.d. $N(0,1)$, its
upper-triangular off-diagonal entries are i.i.d. $N(0,1/2)$, and these
entries are independent.  Equivalently,
$\langle\mathsf G_d,R\rangle_F\sim N(0,\|R\|_F^2)$ for every deterministic
$R\in\Ssym^d$.

\paragraph{Dual cone and Moreau decomposition.} The case that $\kappa = 3$ is special and easier to analyze, because $T_{3} = I$ and $c_{\kappa} = 0$. For general $\kappa$, it is helpful to observe a related duality from convex geometry between the cases $\kappa < 3$ and $\kappa > 3$.
For a nonempty closed convex set $\mathcal K\subset\R^N$, write
$\Pi_{\mathcal K}(z):=\operatorname*{argmin}_{y\in\mathcal K}\|z-y\|_2$
for the Euclidean projection onto $\mathcal K$.  For a convex cone
$\mathcal C\subset\R^N$, define its dual cone and polar cone by
$$
\mathcal C^*
:=\{u\in\R^N:\langle u,y\rangle\ge0
\text{ for every }y\in\mathcal C\},
\qquad
\mathcal C^\circ
:=\{u\in\R^N:\langle u,y\rangle\le0
\text{ for every }y\in\mathcal C\}.
$$
Thus $\mathcal C^\circ=-\mathcal C^*$.

\begin{lemma}[Moreau decomposition for closed convex cones
	{\cite[Section~6.3]{bauschke2017convex}}]
	\label{lem:moreau-cone-decomposition}
	Let $\mathcal C\subset\R^N$ be a closed convex cone.  Then, for every
	$z\in\R^N$,
	\begin{equation}
		\label{eq:moreau-cone-decomposition}
		z=\Pi_{\mathcal C}(z)+\Pi_{\mathcal C^\circ}(z),
		\qquad
		\left\langle
		\Pi_{\mathcal C}(z),\Pi_{\mathcal C^\circ}(z)
		\right\rangle=0.
	\end{equation}
	Consequently,
$\|z\|_2^2
		=\|\Pi_{\mathcal C}(z)\|_2^2
		+\|\Pi_{\mathcal C^\circ}(z)\|_2^2$
	and
	\begin{equation}
		\label{eq:cone-projection-support}
		\sup_{\substack{y\in\mathcal C\\ \|y\|_2\le1}}
		\langle z,y\rangle
		=\|\Pi_{\mathcal C}(z)\|_2.
	\end{equation}
\end{lemma}
The last conclusion above follows by writing
$
\langle z,y\rangle
= 
\langle \Pi_{\mathcal C}(z) + \Pi_{\mathcal C^\circ}(z),y\rangle \le \langle \Pi_{\mathcal C}(z), y \rangle$ and using Cauchy Schwarz to show the maximizing $y$ is proportional to $\Pi_{\mathcal C}(z)$. 

We can now define the key involution $\kappa\mapsto\kappa^\vee$ that exchanges
$1<\kappa<3$ with $\kappa>3$.  Its fixed point is $\kappa=3$, which corresponds to the fact that
$\mathbb S_3^d=\Ssym_+^d$ is self-dual.

\begin{lemma}
	\label{lem:rescaled-psd-cone-duality}
	For $\kappa>1$, define $\kappa^\vee:=1+4/(\kappa-1)$.
	Then
	$$
	(\kappa^\vee)^\vee=\kappa,
	\qquad
	T_{\kappa^\vee}=T_\kappa^{-1},
	\qquad
	c_{\kappa^\vee}=-\frac{c_\kappa}{1+c_\kappa}.
	$$
	Moreover,
	$
	(\mathbb S_\kappa^d)^*=\mathbb S_{\kappa^\vee}^d$ and
	$(\mathbb S_\kappa^d)^\circ=-\mathbb S_{\kappa^\vee}^d.
	$
\end{lemma}

\begin{proof}
	The first three identities follow directly from the definitions.  Since
	$T_\kappa$ is self-adjoint for the Frobenius inner product and
	$\Ssym_+^d$ is self-dual,
	\begin{align*}
		Y\in(\mathbb S_\kappa^d)^*
		&\Longleftrightarrow
		\langle T_\kappa Y,R\rangle_F\ge0
		\quad\text{for every }R\succeq0\\
		&\Longleftrightarrow T_\kappa Y\succeq0
		\Longleftrightarrow
		Y\in\mathbb S_{\kappa^\vee}^d.
	\end{align*}
	This proves the dual- and polar-cone identities.
\end{proof}

\subsection{CGMT reduction to a shifted cone projection}

We recall (a special case of) the nonasymptotic CGMT \cite{gordon1988milman,thrampoulidis-oymak-hassibi2015regularized}. The comparison relates the ``primary optimization'' $\Phi$ involving a Gaussian random matrix $\mathsf A$ to the ``auxiliary optimization'' $\phi$ which is much easier to directly analyze.

\begin{theorem}[Convex Gaussian min--max theorem (CGMT)
	{\cite[Theorem~3\textup{(i)--(ii)}]{thrampoulidis-oymak-hassibi2015regularized}}]
	\label{thm:cgmt-specialized}
	Let $K\subset\R^D$ be nonempty, compact, and convex.  Let
	$\mathsf A\in\R^{(n-1)\times D}$, $g\in\R^{n-1}$, and $h\in\R^D$ be
	independent and have i.i.d. standard Gaussian entries.  Define
	$$
	\Phi(K):=\inf_{x\in K}\|\mathsf A x\|_2,
	\qquad
	\phi(K):=\inf_{x\in K}
	\bigl(\|g\|_2\|x\|_2-\langle h,x\rangle\bigr)_+ .
	$$
	Then, for every $c\in\R$,
	$$
	\Pp\{\Phi(K)<c\}\le2\Pp\{\phi(K)\le c\},
	\qquad
	\Pp\{\Phi(K)>c\}\le2\Pp\{\phi(K)\ge c\}.
	$$
\end{theorem}

\begin{lemma}[Projected Gaussian representation]
	\label{lem:projected-anisotropic-gaussian-representation}
	Let $\mathsf A\in\R^{(n-1)\times N_d}$ have i.i.d. standard Gaussian entries, acting
	on $\Ssym^d$ with the Frobenius inner product.  Then, as processes indexed
	by $R\in\Ssym^d$,
	\begin{equation}
		\label{eq:projected-gaussian-A-representation}
		\mathcal A_{G,\kappa}(R) \stackrel{d}{=} \sqrt{\frac2d}\,\mathsf A T_\kappa R .
	\end{equation}
\end{lemma}

\begin{proof}
	By \eqref{eq:gaussian-section-covariance}, one row of the process on the
	right has covariance
	$$
	\E\left[
	\left(\sqrt{\frac2d}\langle\mathsf G_d,T_\kappa R\rangle_F\right)
	\left(\sqrt{\frac2d}\langle\mathsf G_d,T_\kappa S\rangle_F\right)
	\right]
	=\frac2d\langle T_\kappa R,T_\kappa S\rangle_F
	=\mathcal C_\kappa(R,S).
	$$
	Here $\mathsf G_d$ has the GOE normalization fixed above.  Since the
	Gaussian features $G_i$ are i.i.d., an orthogonal identification of $\one^\perp$
	with $\R^{n-1}$ turns their projection by $P_{\one^\perp}$ into
	$n-1$ independent rows with this covariance.
\end{proof}

We use the preceding representation to apply
Theorem~\ref{thm:cgmt-specialized}. Let
$\mathsf A\in\R^{(n-1)\times N_d}$ have i.i.d. standard Gaussian entries,
let $g\sim N(0,I_{n-1})$, and let $\mathsf G_d$ have the GOE law above,
with all three objects mutually independent. Define the normalized primary
and signed auxiliary values by
\begin{align*}
	\Phi_d
	&:=\frac1{d^{3/2}}\inf_{R\in\mathcal R_d}
	\|\mathsf A T_\kappa R\|_2,\\
	\widehat{\phi}_d
	&:=\frac1{d^{3/2}}\inf_{R\in\mathcal R_d}
	\left\{
	\|g\|_2\|T_\kappa R\|_F
	-\langle\mathsf G_d,T_\kappa R\rangle_F
	\right\}.
\end{align*}
Both values are normalized to remain of constant order when $n\asymp d^2$. The hat
records that this is the CGMT auxiliary prediction \emph{before taking its positive
part}; this distinction is important in the SAT phase. Formally:
\begin{lemma}
For every
$c\in\R$,
\begin{equation}\label{eq:cgmt-comparisons}
	\Pp\{\Phi_d<c\}
	\le2\Pp\{(\widehat{\phi}_d)_+\le c\},\qquad
	\Pp\{\Phi_d>c\}
	\le2\Pp\{(\widehat{\phi}_d)_+\ge c\}.
\end{equation}
\end{lemma}
\begin{proof}
Apply Theorem~\ref{thm:cgmt-specialized} to
$K=T_\kappa\mathcal R_d$, identify $\R^{N_d}$ isometrically with
$\Ssym^d$, and use
$-\mathsf G_d\stackrel d=\mathsf G_d$.
\end{proof}

It remains to evaluate the signed auxiliary value $\widehat{\phi}_d$;
the first step is the following lemma, which rewrites it in terms of cone projections. For
$\zeta\in\R$, set
$$
\mathsf G_{d,\kappa}(\zeta)
:=\mathsf G_d+\zeta\sqrt{2d/(\kappa-1)}\,I_d.
$$
\begin{lemma}
	\label{lem:signed-auxiliary-projection-crossing}
	One has
	\begin{equation}\label{eq:signed-auxiliary-projection-crossing}
		\widehat{\phi}_d
		=
		\sup\left\{\zeta:
		\left\|\Pi_{\mathbb S_\kappa^d}
		\bigl(\mathsf G_{d,\kappa}(\zeta)\bigr)\right\|_F
		\le\|g\|_2
		\right\}.
	\end{equation}
\end{lemma}

\begin{proof}
	For every $\zeta\in\R$, positive homogeneity gives
	$$
	\zeta\le\widehat{\phi}_d
	\quad\Longleftrightarrow\quad
	\|g\|_2\|T_\kappa R\|_F
	-\langle\mathsf G_d,T_\kappa R\rangle_F
	\ge\zeta\sqrt d\,\Tr R
	\qquad\text{for every }R\succeq0.
	$$
	Indeed, restricting to $\Tr R=d$ gives one implication, while every
	nonzero $R\succeq0$ can be rescaled to have trace $d$, giving the other.
	
	If $Y=T_\kappa R$, then
	$\Tr R=\langle T_\kappa^{-1}I_d,Y\rangle_F
	=\sqrt{2/(\kappa-1)}\langle I_d,Y\rangle_F$.  The preceding condition is
	therefore equivalent to
	$$
	\langle \mathsf G_{d,\kappa}(\zeta),Y\rangle_F
	\le\|g\|_2\|Y\|_F
	\qquad\text{for every }Y\in\mathbb S_\kappa^d.
	$$
	By the projection formula \eqref{eq:cone-projection-support}, this holds
	if and only if
	$$
	\left\|\Pi_{\mathbb S_\kappa^d}
	\bigl(\mathsf G_{d,\kappa}(\zeta)\bigr)\right\|_F
	\le\|g\|_2.
	$$
	Taking the supremum over such $\zeta$ proves
	\eqref{eq:signed-auxiliary-projection-crossing}.
\end{proof}

Thus the Gaussian optimization has been reduced to computing the size of a
shifted GOE projection onto $\mathbb S_\kappa^d$. This is the task of
the next two subsections.

\subsection{Analysis of the scalar variational problem}
\label{subsec:auxiliary}
By \eqref{eq:signed-auxiliary-projection-crossing}, we need to determine the
asymptotic size of
$\Pi_{\mathbb S_\kappa^d}(\mathsf G_{d,\kappa}(\zeta))$. 
In this section, we study the scalar quantity $\chi_{\kappa}(\zeta)$ which will determine its limit, as well as resulting scalar variational problem described in the introduction. See Figure~\ref{fig:gaussian-scalar-level-crossings} for an illustration of the variational problem when $\kappa = 3$. The rigorous connection to \eqref{eq:signed-auxiliary-projection-crossing} is proved in the next section. 

Recall that $\nu$ is the
semicircle law on $[-\sqrt2,\sqrt2]$, normalized by
$\int x^2\,d\nu(x)=1/2$; it arises as the (appropriately normalized) limiting distribution of the spectrum of a GOE matrix $\mathsf{G}_d$ \cite{anderson-guionnet-zeitouni2010random}.
For $\omega\in\R$, define
$$
m(\omega):=\int (x-\omega)_+\,d\nu(x),\qquad
s(\omega):=\int (x-\omega)_+^2\,d\nu(x),\qquad
\mathcal E_\kappa(\omega)
:=s(\omega)+c_\kappa m(\omega)^2.
$$
The formula for the threshold can be defined abstractly as the solution of a variational problem:
$$
\alpha_\star(\kappa)
:=
\sup_{\substack{f\in L^2(\nu),\ f\ge0\\ \int f\,d\nu=1}}
\frac{\left(\int xf(x)\,d\nu(x)\right)^2}
{\int f(x)^2\,d\nu(x)+c_\kappa}.
$$
Its denominator is strictly positive: by Cauchy Schwarz,
$\int f^2\,d\nu+c_\kappa\ge(\int f\,d\nu)^2+(\kappa-3)/2
=(\kappa-1)/2$.
For $\zeta\in\R$, let $\omega_{\kappa,\zeta}$ denote the unique solution
\begin{equation}\label{eq:shifted-omega-eqn}
	\omega_{\kappa,\zeta}+\zeta
	=c_\kappa m(\omega_{\kappa,\zeta}),
\end{equation}
whose existence and uniqueness are proved below, and set
$$
\chi_\kappa(\zeta)
:=\mathcal E_\kappa(\omega_{\kappa,\zeta}),
\qquad
\omega_\kappa:=\omega_{\kappa,0}.
$$
Thus $\omega_\kappa=c_\kappa m(\omega_\kappa)$ and
$\chi_\kappa(0)=s(\omega_\kappa)+c_\kappa m(\omega_\kappa)^2$.

The following lemma shows how the variational problem is solved; it will be reused later, so we state it for a general probability measure $\mu$. 
\begin{lemma}
	\label{lem:positive-cone-variational}
	Let $\mu$ be a probability measure on $\R$ with finite second moment,
	let $c>-1$, and let $\zeta\in\R$.  Define
	$$
	m_\mu(\omega):=\int(x-\omega)_+\,d\mu(x),
	\qquad
	s_\mu(\omega):=\int(x-\omega)_+^2\,d\mu(x).
	$$
	There is a unique solution $\omega$ of
	$\omega+\zeta=c\,m_\mu(\omega)$.  For this solution,
	$$
	\sup_{\substack{f\in L^2(\mu),\ f\ge0\\f\not\equiv0}}
	\frac{\left(\int(x+\zeta)f(x)\,d\mu(x)\right)_+^2}
	{\int f(x)^2\,d\mu(x)
		+c\left(\int f\,d\mu\right)^2}
	=
	s_\mu(\omega)+c\,m_\mu(\omega)^2.
	$$
	If the right-hand side is positive, the supremum is attained by every
	positive multiple of $f(x)=(x-\omega)_+$.
\end{lemma}

\begin{proof}
	Since $m_\mu$ is decreasing and $1$-Lipschitz, the difference quotients
	of $H(\omega):=\omega+\zeta-c\,m_\mu(\omega)$ are bounded below by
	$\min\{1,1+c\}>0$, so $H$ is strictly increasing.  Moreover,
	$m_\mu(\omega)\to0$ as $\omega\to\infty$, while
	$m_\mu(\omega)=-\omega+\int x\,d\mu+o(1)$ as $\omega\to-\infty$.
	Thus $H(\omega)\to\pm\infty$ as $\omega\to\pm\infty$, proving
	existence and uniqueness.
	
	Put $p(x)=(x-\omega)_+$ and $\overline p=\int p\,d\mu$.  The quadratic
	form
	\begin{align*}
		Q_c(u,v)&:=\int uv\,d\mu
		+c\left(\int u\,d\mu\right)\left(\int v\,d\mu\right),\\
		Q_c(u,u)&=\int\left(u-\int u\,d\mu\right)^2d\mu
		+(1+c)\left(\int u\,d\mu\right)^2
	\end{align*}
	is an inner product.  The defining equation for $\omega$ gives the
	pointwise bound $x+\zeta\le p(x)+c\overline p$.
	Therefore, for every $f\ge0$,
	$$
	\left(\int(x+\zeta)f\,d\mu\right)_+
	\le |Q_c(p,f)|
	\le Q_c(p,p)^{1/2}Q_c(f,f)^{1/2}.
	$$
	This proves the upper bound.  Moreover,
	$$
	\int(x+\zeta)p\,d\mu
	=\int p^2\,d\mu+(\omega+\zeta)\overline p
	=s_\mu(\omega)+c\,m_\mu(\omega)^2
	=Q_c(p,p),
	$$
	so equality holds at $f=p$ whenever $p\not\equiv0$.  If
	$p\equiv0$, then $x+\zeta\le0$ almost surely and both sides vanish.
\end{proof}
We continue with a couple of useful technical lemmas.
\begin{lemma}[Stability of the scalar variational problem]
\label{lem:positive-cone-variational-stability}
Let $\mu_d,\mu$ be probability measures on $\R$ with finite second moments,
and suppose that $m_{\mu_d}\to m_\mu$ and $s_{\mu_d}\to s_\mu$
locally uniformly.  Fix $c>-1$ and $\zeta\in\R$, and let $\omega_d,\omega$
be the unique solutions of
$\omega_d+\zeta=c\,m_{\mu_d}(\omega_d)$ and
$\omega+\zeta=c\,m_\mu(\omega)$, respectively.  Then
\[
\omega_d\to\omega,
\qquad
s_{\mu_d}(\omega_d)+c\,m_{\mu_d}(\omega_d)^2
\to
s_\mu(\omega)+c\,m_\mu(\omega)^2.
\]
The same conclusion holds in probability if $\mu_d$ are random and the
local uniform convergence holds in probability.
\end{lemma}

\begin{proof}
Let $H_d(t):=t+\zeta-c\,m_{\mu_d}(t)$ and
$\delta:=\min\{1,1+c\}>0$.  As in the preceding proof,
$H_d(t)-H_d(s)\ge\delta(t-s)$ for $t\ge s$.  Hence
\[
\delta|\omega_d-\omega|
\le |H_d(\omega)|
=|c|\,|m_{\mu_d}(\omega)-m_\mu(\omega)|
\to0.
\]
The local uniform convergence of $m_{\mu_d}$ and $s_{\mu_d}$ then gives
the convergence of the displayed variational values.  The same argument
applies in probability.
\end{proof}

\begin{lemma}
\label{lem:shifted-duality}
For $\kappa>1$ and $\zeta\in\R$, let $\kappa^\vee$ be as in
Lemma~\ref{lem:rescaled-psd-cone-duality} and set
\begin{equation}
\label{eq:dual-shift-parameter}
\zeta^\vee:=-\frac{2\zeta}{\kappa-1}.
\end{equation}
Then
$\omega_{\kappa^\vee,\zeta^\vee}=-\omega_{\kappa,\zeta}$,
\begin{equation}\label{eq:shifted-alpha-involution}
\chi_\kappa(\zeta)+\chi_{\kappa^\vee}(\zeta^\vee)
=\frac12+\frac{2\zeta^2}{\kappa-1},
\end{equation}
and $-\mathsf G_{d,\kappa}(\zeta)
\stackrel d=
\mathsf G_{d,\kappa^\vee}(\zeta^\vee)$.

\end{lemma}

\begin{proof}
Write $\omega=\omega_{\kappa,\zeta}$.  By
Lemma~\ref{lem:rescaled-psd-cone-duality},
$c_{\kappa^\vee}=-c_\kappa/(1+c_\kappa)$.
Symmetry of $\nu$ gives
$$
m(-\omega)=m(\omega)+\omega,\qquad
s(-\omega)+s(\omega)=\frac12+\omega^2.
$$
Using $\omega+\zeta=c_\kappa m(\omega)$, one obtains
$$
-\omega+\zeta^\vee
=-\omega-\frac{\zeta}{1+c_\kappa}
=c_{\kappa^\vee}\bigl(m(\omega)+\omega\bigr)
=c_{\kappa^\vee}m(-\omega).
$$
Thus uniqueness in \eqref{eq:shifted-omega-eqn} (from Lemma~\ref{lem:positive-cone-variational}) gives
$\omega_{\kappa^\vee,\zeta^\vee}=-\omega$.  Therefore
\begin{align*}
\chi_\kappa(\zeta)+\chi_{\kappa^\vee}(\zeta^\vee)
&=s(\omega)+s(-\omega)+c_\kappa m(\omega)^2
+c_{\kappa^\vee}\bigl(m(\omega)+\omega\bigr)^2\\
&=\frac12+\frac{\zeta^2}{1+c_\kappa}
=\frac12+\frac{2\zeta^2}{\kappa-1},
\end{align*}
which proves \eqref{eq:shifted-alpha-involution}.  Finally, since
$\kappa^\vee-1=4/(\kappa-1)$,
\[
\zeta^\vee\sqrt{\frac{2d}{\kappa^\vee-1}}
=-\zeta\sqrt{\frac{2d}{\kappa-1}}.
\]
Together with $-\mathsf G_d\stackrel d=\mathsf G_d$, this proves the last claim.
\end{proof}

The following proposition collects the key facts about the scalar variational problem. 
\begin{proposition}
	\label{prop:scalar-gaussian-energy-curve}
	For every $\kappa>1$, the solution in
	\eqref{eq:shifted-omega-eqn} exists and is
	unique for every $\zeta\in\R$. The map
	$\zeta\mapsto\chi_\kappa(\zeta)$ is nondecreasing and is strictly
	increasing wherever it is positive. Moreover:
	\begin{enumerate}[label=\textup{(\roman*)}]
		\item For every $\alpha>0$, there is a unique
		$\omega_{\alpha,\kappa}<\sqrt2$ satisfying
		\begin{equation}\label{eq:energy-curve-inverse}
			\mathcal E_\kappa(\omega_{\alpha,\kappa})=\alpha.
		\end{equation}
		Define $\zeta_{\alpha,\kappa}
		:=c_\kappa m(\omega_{\alpha,\kappa})-\omega_{\alpha,\kappa}$.
		Then $\chi_\kappa(\zeta_{\alpha,\kappa})=\alpha$ and
		\begin{equation}\label{eq:energy-positive-iff-above-threshold}
			\zeta_{\alpha,\kappa}>0
			\quad\Longleftrightarrow\quad
			\alpha>\alpha_\star(\kappa).
		\end{equation}
		
		\item For every $\zeta\in\R$,
		$$
		\sup_{\substack{f\in L^2(\nu),\ f\ge0\\f\not\equiv0}}
		\frac{\left(\int(x+\zeta)f(x)\,d\nu(x)\right)_+^2}
		{\int f(x)^2\,d\nu(x)
			+c_\kappa\left(\int f\,d\nu\right)^2}
		=\chi_\kappa(\zeta).
		$$
		In particular, $\alpha_\star(\kappa)=\chi_\kappa(0)$.
		
		\item For all $\kappa > 1$ we have
		$0\le\alpha_\star(\kappa)\le1/2$, with
		$\alpha_\star(\kappa)\to0$ as $\kappa\to\infty$ and
		$\alpha_\star(\kappa)\to1/2$ as $\kappa\to 1$.
	\end{enumerate}
\end{proposition}
\begin{proof}
	Applying Lemma~\ref{lem:positive-cone-variational} with
	$\mu=\nu$ and $c=c_\kappa$ proves existence and uniqueness in
	\eqref{eq:shifted-omega-eqn}, as well as the variational identity in
	part~\textup{(ii)}.  At $\zeta=0$, symmetry of $\nu$ allows a profile
	with negative $\int xf\,d\nu$ to be reflected without changing the
	denominator.  Homogeneity then permits the normalization
	$\int f\,d\nu=1$, and hence
	$\alpha_\star(\kappa)=\chi_\kappa(0)$.
	
	Since the semicircle law has no atoms, $m$ is continuously differentiable,
	with $m'(\omega)=-\nu((\omega,\infty))$.  Let
	$F_\kappa(\omega)=\omega-c_\kappa m(\omega)$.  Then
	$$
	F_\kappa'(\omega)
	=1+c_\kappa\nu((\omega,\infty))
	\ge
	\begin{cases}
		1, & \kappa\ge3,\\
		(\kappa-1)/2, & 1<\kappa<3.
	\end{cases}
	$$
	In particular, $F_\kappa'(\omega)>0$.
	
	Also, $s'(\omega)=-2m(\omega)$, so
	$\mathcal E_\kappa'(\omega)=-2m(\omega)F_\kappa'(\omega)$.
	Differentiating $F_\kappa(\omega_{\kappa,\zeta})=-\zeta$ therefore gives
	$$
	\chi_\kappa'(\zeta)=2m(\omega_{\kappa,\zeta}).
	$$
	This derivative is nonnegative, and is strictly positive exactly when
	$\omega_{\kappa,\zeta}<\sqrt2$, equivalently when
	$\chi_\kappa(\zeta)>0$.  This proves the asserted monotonicity.
	Moreover, $\mathcal E_\kappa(\omega)\to0$ as
	$\omega\uparrow\sqrt2$.  For $\omega\le-\sqrt2$, symmetry and the
	variance normalization give
	$$
	m(\omega)=-\omega,\qquad
	s(\omega)=\omega^2+\frac12,\qquad
	\mathcal E_\kappa(\omega)
	=\frac{\kappa-1}{2}\omega^2+\frac12,
	$$
	which tends to infinity as $\omega\to-\infty$.  This proves existence and
	uniqueness in \eqref{eq:energy-curve-inverse}.
	
	The definition of $\zeta_{\alpha,\kappa}$ gives
	$\omega_{\alpha,\kappa}+\zeta_{\alpha,\kappa}
	=c_\kappa m(\omega_{\alpha,\kappa})$, so uniqueness in
	\eqref{eq:shifted-omega-eqn} gives
	$\omega_{\kappa,\zeta_{\alpha,\kappa}}
	=\omega_{\alpha,\kappa}$, and hence
	$\chi_\kappa(\zeta_{\alpha,\kappa})=\alpha$.
	Since $\mathcal E_\kappa$ is strictly decreasing on
	$(-\infty,\sqrt2)$,
	$\alpha>\alpha_\star(\kappa)=\chi_\kappa(0)$ if and only if
	$\omega_{\alpha,\kappa}<\omega_\kappa$.
	But $\zeta_{\alpha,\kappa}=-F_\kappa(\omega_{\alpha,\kappa})$, while
	$F_\kappa$ is strictly increasing and $F_\kappa(\omega_\kappa)=0$.
	This proves \eqref{eq:energy-positive-iff-above-threshold}.

    We now prove part (iii) based on Lemma~\ref{lem:shifted-duality}.
	Taking $\zeta=0$ in \eqref{eq:shifted-alpha-involution} gives
	$\chi_\kappa(0)+\chi_{\kappa^\vee}(0)=1/2$.  By part~\textup{(ii)},
	$\chi_\kappa(0)=\alpha_\star(\kappa)$ and
	$\chi_{\kappa^\vee}(0)=\alpha_\star(\kappa^\vee)$, hence
	\[ \alpha_\star(\kappa) + \alpha_\star(\kappa^\vee)=1/2. \]
	Part \textup{(ii)} shows that both terms are nonnegative, proving
	$0\le\alpha_\star(\kappa)\le1/2$. For $\kappa>3$, every admissible profile in the definition of
	$\alpha_\star(\kappa)$ satisfies
	$|\int xf\,d\nu|\le\sqrt2\int f\,d\nu=\sqrt2$ and
	$\int f^2\,d\nu+c_\kappa\ge c_\kappa=(\kappa-3)/2$.
	Hence $\alpha_\star(\kappa)\le 2/c_\kappa=4/(\kappa-3)\to0$ as
	$\kappa\to\infty$. In the limit $\kappa\downarrow1$,
	we have $\kappa^\vee\to\infty$ so 
	$\alpha_\star(\kappa)=1/2-\alpha_\star(\kappa^\vee) \to 1/2$.
\end{proof}

\begin{figure}[tbp]
\centering

\definecolor{curveblue}{RGB}{38,114,180}
\definecolor{beloworange}{RGB}{213,94,0}
\definecolor{abovegreen}{RGB}{0,145,92}

\begin{tikzpicture}[font=\footnotesize,>=Latex]


\node[font=\small\bfseries] at (2.55,6.55) {(a) GOE cutoffs};
\node[font=\small\bfseries] at (7.10,6.55) {(b) CGMT spectral profiles};
\node[font=\small\bfseries] at (12.65,6.55)
  {(c) Level crossings $\chi_3(\zeta)=\alpha$};

\node[beloworange,anchor=east,font=\scriptsize\bfseries] at (0.72,5.20)
  {$\alpha_-<\alpha_\star$};
\node[beloworange,anchor=east,font=\scriptsize] at (0.72,4.92)
  {$\omega_->0$};

\node[curveblue,anchor=east,font=\scriptsize\bfseries] at (0.72,3.30)
  {$\alpha=\alpha_\star$};
\node[curveblue,anchor=east,font=\scriptsize] at (0.72,3.02)
  {$\omega=0$};

\node[abovegreen,anchor=east,font=\scriptsize\bfseries] at (0.72,1.40)
  {$\alpha_+>\alpha_\star$};
\node[abovegreen,anchor=east,font=\scriptsize] at (0.72,1.12)
  {$\omega_+<0$};


\begin{scope}[shift={(2.55,4.77)},x=1.02cm,y=0.72cm]
  \begin{scope}
    \clip
      (-1.414,0)
      plot[domain=-1.414:1.414,samples=100,smooth,variable=\x]
        (\x,{0.80*sqrt(max(0,2-\x*\x))})
      -- (1.414,0) -- cycle;
    \fill[black!14] (-1.60,0) rectangle (0.333,1.30);
    \fill[beloworange!18] (0.333,0) rectangle (1.60,1.30);
  \end{scope}

  \draw[black!48,thin]
    plot[domain=-1.414:1.414,samples=100,smooth,variable=\x]
      (\x,{0.80*sqrt(max(0,2-\x*\x))});

  \draw[->] (-1.58,0) -- (1.63,0) node[right] {$x$};
  \draw[beloworange,thick,dashed] (0.333,0) -- (0.333,1.15);
  \node[beloworange,above] at (0.333,1.10) {$\omega_-$};
\end{scope}

\begin{scope}[shift={(2.55,2.87)},x=1.02cm,y=0.72cm]
  \begin{scope}
    \clip
      (-1.414,0)
      plot[domain=-1.414:1.414,samples=100,smooth,variable=\x]
        (\x,{0.80*sqrt(max(0,2-\x*\x))})
      -- (1.414,0) -- cycle;
    \fill[black!14] (-1.60,0) rectangle (0,1.30);
    \fill[curveblue!15] (0,0) rectangle (1.60,1.30);
  \end{scope}

  \draw[black!48,thin]
    plot[domain=-1.414:1.414,samples=100,smooth,variable=\x]
      (\x,{0.80*sqrt(max(0,2-\x*\x))});

  \draw[->] (-1.58,0) -- (1.63,0) node[right] {$x$};
  \draw[curveblue,thick,dashed] (0,0) -- (0,1.15);
  \node[black,below=1pt] at (0,0) {$0$};
\end{scope}

\begin{scope}[shift={(2.55,0.97)},x=1.02cm,y=0.72cm]
  \begin{scope}
    \clip
      (-1.414,0)
      plot[domain=-1.414:1.414,samples=100,smooth,variable=\x]
        (\x,{0.80*sqrt(max(0,2-\x*\x))})
      -- (1.414,0) -- cycle;
    \fill[black!14] (-1.60,0) rectangle (-0.322,1.30);
    \fill[abovegreen!18] (-0.322,0) rectangle (1.60,1.30);
  \end{scope}

  \draw[black!48,thin]
    plot[domain=-1.414:1.414,samples=100,smooth,variable=\x]
      (\x,{0.80*sqrt(max(0,2-\x*\x))});

  \draw[->] (-1.58,0) -- (1.63,0) node[right] {$x$};
  \draw[abovegreen,thick,dashed] (-0.322,0) -- (-0.322,1.15);
  \node[abovegreen,above] at (-0.322,1.10) {$\omega_+$};
\end{scope}


\begin{scope}[shift={(5.55,4.76)},x=0.43cm,y=5.35cm]

  \fill[beloworange!10]
    plot[domain=0:6.824,samples=120,variable=\l]
      (\l,{
        0.1584483/3.14159265
        *sqrt(max(0,2-(0.333+0.1584483*\l)^2))
      })
    -- (6.824,0) -- (0,0) -- cycle;

  \draw[->] (0,0) -- (7.25,0) node[right] {$\lambda$};
  \draw[black!35,thin,->] (0,0) -- (0,0.300);

  \draw[beloworange!80!black,very thick,dashed,smooth]
    plot[domain=0:6.824,samples=120,variable=\l]
      (\l,{
        0.1584483/3.14159265
        *sqrt(max(0,2-(0.333+0.1584483*\l)^2))
      });

  \draw[white,line width=3.8pt,
        -{Latex[length=2.0mm,width=1.8mm]}]
    (0,0.002) -- (0,0.215);
  \draw[beloworange!80!black,very thick,dashed,
        -{Latex[length=1.7mm,width=1.5mm]}]
    (0,0.002) -- (0,0.215);
  \node[beloworange!80!black,font=\scriptsize,anchor=west]
    at (0.16,0.195) {$\approx 0.65\,\delta_0$};
\end{scope}

\begin{scope}[shift={(5.55,2.86)},x=0.43cm,y=5.35cm]

  \fill[curveblue!10]
    plot[domain=0:4.71239,samples=120,variable=\l]
      (\l,{
        0.3001054/3.14159265
        *sqrt(max(0,2-(0.3001054*\l)^2))
      })
    -- (4.71239,0) -- (0,0) -- cycle;

  \draw[->] (0,0) -- (7.25,0) node[right] {$\lambda$};
  \draw[black!35,thin,->] (0,0) -- (0,0.300);

  \draw[curveblue,very thick,smooth]
    plot[domain=0:4.71239,samples=120,variable=\l]
      (\l,{
        0.3001054/3.14159265
        *sqrt(max(0,2-(0.3001054*\l)^2))
      });

  \draw[white,line width=3.8pt,
        -{Latex[length=2.0mm,width=1.8mm]}]
    (0,0.002) -- (0,0.1654);
  \draw[curveblue!80!black,very thick,
        -{Latex[length=1.7mm,width=1.5mm]}]
    (0,0.002) -- (0,0.1654);
  \node[curveblue!80!black,font=\scriptsize,anchor=west]
    at (0.20,0.172) {$0.5\,\delta_0$};
\end{scope}

\begin{scope}[shift={(5.55,0.96)},x=0.43cm,y=5.35cm]

  \fill[abovegreen!10]
    plot[domain=0:3.5847,samples=120,variable=\l]
      (\l,{
        0.4843412/3.14159265
        *sqrt(max(0,2-(-0.322+0.4843412*\l)^2))
      })
    -- (3.5847,0) -- (0,0) -- cycle;

  \draw[->] (0,0) -- (7.25,0) node[right] {$\lambda$};
  \draw[black!35,thin,->] (0,0) -- (0,0.300);
  \node[black!55,font=\scriptsize,anchor=north]
    at (0,-0.012) {$0$};

  \draw[abovegreen,very thick,smooth]
    plot[domain=0:3.5847,samples=120,variable=\l]
      (\l,{
        0.4843412/3.14159265
        *sqrt(max(0,2-(-0.322+0.4843412*\l)^2))
      });

  \draw[white,line width=3.8pt,
        -{Latex[length=2.0mm,width=1.8mm]}]
    (0,0.002) -- (0,0.1191);
  \draw[abovegreen!75!black,very thick,
        -{Latex[length=1.7mm,width=1.5mm]}]
    (0,0.002) -- (0,0.1191);
  \node[abovegreen!75!black,font=\scriptsize,anchor=west]
    at (0.16,0.101) {$\approx 0.36\,\delta_0$};
\end{scope}


\begin{scope}[shift={(12.55,0.80)},x=2.75cm,y=4.85cm]

  \draw[->] (-0.94,0) -- (0.78,0) node[right] {$\zeta$};
  \draw[->] (-0.90,-0.015) -- (-0.90,1.04) node[above] {$\alpha$};

  \draw (-0.5,0.015) -- (-0.5,-0.015)
    node[below=2pt] {$-\tfrac12$};
  \draw (0,0.015) -- (0,-0.015)
    node[below=2pt] {$0$};
  \draw (0.5,0.015) -- (0.5,-0.015)
    node[below=2pt] {$\tfrac12$};

  \draw[beloworange,dashed]
    (-0.90,0.10) -- (0.39,0.10);
  \draw[curveblue,dashed]
    (-0.90,0.25) -- (0.39,0.25);
  \draw[abovegreen,dashed]
    (-0.90,0.50) -- (0.39,0.50);

  \node[beloworange,anchor=west,font=\scriptsize]
    at (0.41,0.10) {$\alpha_- = 0.1$};
  \node[curveblue,anchor=west,font=\scriptsize]
    at (0.41,0.25) {$\alpha_\star=0.25$};
  \node[abovegreen,anchor=west,font=\scriptsize]
    at (0.41,0.50) {$\alpha_+ = 0.5$};

  \draw[black!75,very thick,smooth]
    plot coordinates {
      (-0.90,0.00771) (-0.80,0.01426) (-0.70,0.02403)
      (-0.60,0.03776) (-0.50,0.05626) (-0.40,0.08035)
      (-0.30,0.11089) (-0.20,0.14876) (-0.10,0.19483)
      ( 0.00,0.25000) ( 0.10,0.31517) ( 0.20,0.39124)
      ( 0.30,0.47911) ( 0.40,0.57965) ( 0.50,0.69374)
      ( 0.60,0.82224) ( 0.70,0.96597)
    };
  \node[black!75,anchor=east]
    at (0.68,1.00) {$\chi_3(\zeta)$};

  \draw[beloworange,densely dashed]
    (-0.333,0) -- (-0.333,0.10);
  \draw[curveblue,densely dashed]
    (0,0) -- (0,0.25);
  \draw[abovegreen,densely dashed]
    (0.322,0) -- (0.322,0.50);

  \fill[beloworange] (-0.333,0.10) circle[radius=1.5pt];
  \fill[curveblue]   (0,0.25)      circle[radius=1.5pt];
  \fill[abovegreen]  (0.322,0.50)  circle[radius=1.5pt];

  \node[beloworange,anchor=north]
    at (-0.333,-0.035) {$\zeta_-$};
  \node[abovegreen,anchor=north]
    at (0.322,-0.035) {$\zeta_+$};

\end{scope}

\end{tikzpicture}
\caption{ Spectral interpretation of the CGMT calculation for $\kappa=3$. Panel~\textup{(a)} shows the GOE semicircle law and the cutoff $\omega=-\zeta$: the gray region $x\le\omega$ is clipped to zero, while $x>\omega$ is retained. Panel~\textup{(b)} shows the resulting trace-normalized spectral profiles; the curves are the continuous parts and the spikes are the atoms at zero. In the SAT phase $\alpha_- < \alpha_\star$, the orange profile is the continuation of the signed CGMT auxiliary optimizer, but our exact fitting proof instead uses the well-conditioned fit constructed in Lemma~\ref{lem:gaussian-bounded-witness}. Panel~\textup{(c)} shows how the level crossings $\chi_3(\zeta)=\alpha$ select the corresponding values of $\zeta = -\omega$. }	\label{fig:gaussian-scalar-level-crossings}
\end{figure}
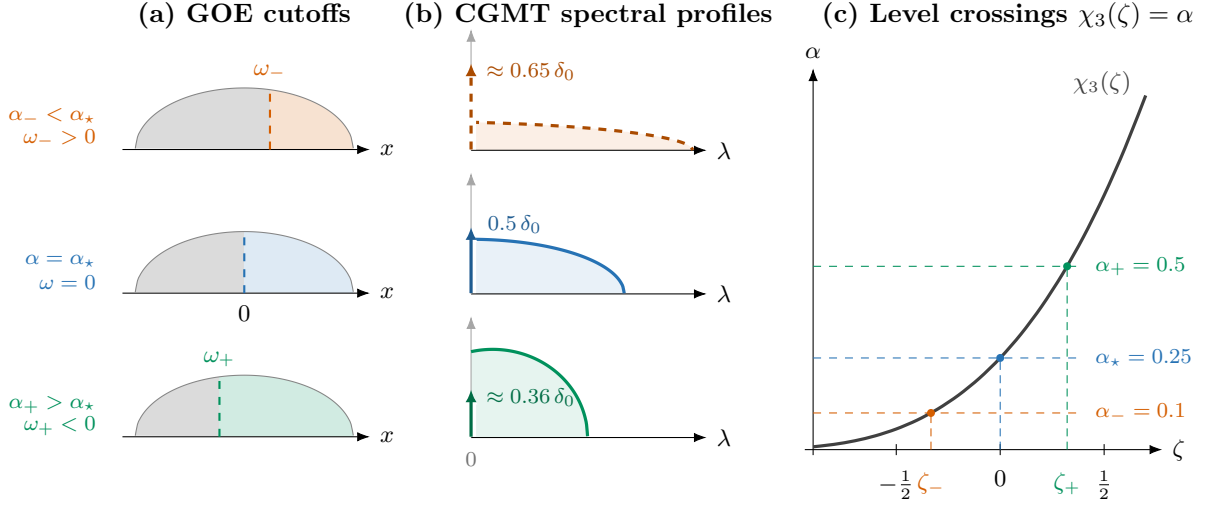

\begin{remark}[Spectral profiles at $\kappa=3$ and the replica prediction]
\label{rem:replica-spectral-correspondence}
For $\kappa=3$, the spectral meaning of the scalar optimizer is especially
transparent; see Figure~\ref{fig:gaussian-scalar-level-crossings}.  Since
$T_3=I$, the cone projection is simply projection onto $\Ssym_+^d$.  Writing
\[
    \mathsf G_d
    =U\diag(\sqrt d\,g_1,\ldots,\sqrt d\,g_d)U^\top,
\]
we have
\[
    \Pi_{\Ssym_+^d}
    \bigl(\mathsf G_d+\zeta\sqrt d\,I_d\bigr)
    =
    U\diag\bigl(\sqrt d\,(g_i+\zeta)_+\bigr)U^\top .
\]
Thus the projection shifts the GOE spectrum by $\zeta$ and clips at zero.
After trace normalization, its eigenvalues are asymptotically described by
\[
    f_\zeta(x)
    :=\frac{(x+\zeta)_+}{m(-\zeta)}
    =\frac{(x-\omega)_+}{m(\omega)},
    \qquad \omega=-\zeta.
\]
In particular, if $X\sim\nu$, the corresponding limiting spectral law is
the law of $f_\zeta(X)$: the mass $X\le\omega$ is placed at zero, while the
remaining part of the semicircle is shifted and rescaled onto the positive
axis.  Notice that $\omega<0$ causes no conflict with positive
semidefiniteness: it only means that some negative GOE eigenvalues survive
the cutoff, and they are mapped to positive eigenvalues of the projected
matrix.

At the transition $\alpha_\star(3) = 1/4$, we have $\zeta=\omega=0$ and
\[
    f_{\star,3}(x)=\frac{x_+}{m(0)}.
\]
Hence the critical spectral law consists of an atom of mass $1/2$ at zero
and the rescaled positive half of the semicircle law.  Since
$m(0)=2\sqrt2/(3\pi)$, this is
\[
    \frac12\delta_0
    +\frac{4}{9\pi^3}\sqrt{9\pi^2-4\lambda^2}\,
    \mathbf 1_{\{0<\lambda<3\pi/2\}}\,d\lambda,
\]
which matches the critical law $\mu_c$ predicted by Maillard and
Kunisky~\cite{maillard-kunisky2023replica}.

For the UNSAT phase $\alpha>\alpha_\star(3)=1/4$, the crossing equation
\[
    \chi_3(\zeta_{\alpha,3})=\alpha
\]
selects $\zeta_{\alpha,3}>0$, and the profile above is genuinely the
trace-normalized optimizer of the CGMT auxiliary problem.  The same profile
continues algebraically to the SAT phase $\alpha<\alpha_\star(3)$, as illustrated in Figure~\ref{fig:gaussian-scalar-level-crossings}, but this continuation
has zero eigenvalues so it is not well-conditioned. Instead, the exact SAT proof
 uses Lemma~\ref{lem:gaussian-bounded-witness} to obtain a well-conditioned fit in the interior of the PSD cone.  Note that there are many ellipsoids which fit the data well in the SAT phase; see
\cite{maillard-kunisky2023replica} for more discussion and replica predictions for ``typical'' ellipsoid fits.
\end{remark}

We now compute the expansion about the critical threshold that was stated in the introduction. This shows that the limiting squared loss has continuous first derivative and discontinuous second derivative, i.e., it has a second-order phase transition.
\begin{lemma}
	\label{lem:gaussian-energy-quadratic-onset}
	Fix $\kappa>1$. As
	$\alpha\downarrow\alpha_\star(\kappa)$ from above,
	\begin{equation}
		\label{eq:gaussian-energy-quadratic-onset-proof}
		e_\star(\alpha,\kappa)
		=
		\frac{(\alpha-\alpha_\star(\kappa))^2}
		{2\alpha_\star(\kappa)m(\omega_\kappa)^2}
		+o\!\left((\alpha-\alpha_\star(\kappa))^2\right).
	\end{equation}
\end{lemma}

\begin{proof}
	By Proposition~\ref{prop:scalar-gaussian-energy-curve},
	$\chi_\kappa(0)=\alpha_\star(\kappa)$,
	$\chi_\kappa(\zeta_{\alpha,\kappa})=\alpha$, and
	$\chi_\kappa'(0)=2m(\omega_\kappa)>0$.
	The inverse function theorem therefore gives, as
	$\alpha\downarrow\alpha_\star(\kappa)$,
	\[
	\zeta_{\alpha,\kappa}
	=
	\frac{\alpha-\alpha_\star(\kappa)}{2m(\omega_\kappa)}
	+o\!\left(\alpha-\alpha_\star(\kappa)\right).
	\]
	Since $e_\star(\alpha,\kappa)=2\zeta_{\alpha,\kappa}^2/\alpha$ for
	$\alpha>\alpha_\star(\kappa)$, and
	$\alpha\to\alpha_\star(\kappa)$, the claimed expansion follows.
\end{proof}

\subsection{Evaluation of the shifted projection}
We now evaluate the shifted projection in
\eqref{eq:signed-auxiliary-projection-crossing}, connecting it back to the scalar $\chi_{\kappa}(\zeta)$ from the previous section. The following elementary lemma is useful because the GOE eigenvector matrix is Haar and independent of the eigenvalues.

\begin{lemma}
\label{lem:haar-diagonal-averaging}
Let $U$ be Haar distributed on $O(d)$ and let
$D=\diag(\rho_1,\ldots,\rho_d)$ for a deterministic
$\rho\in\R^d$.  Write
\[
\bar\rho_d:=\frac1d\sum_{i=1}^d\rho_i.
\]
Then, for every fixed $C<\infty$, uniformly over
$\|\rho\|_\infty\le C$,
\begin{equation}
\label{eq:haar-diagonal-averaging}
\frac1d
\left\|(UDU^\top)_{\diag}-\bar\rho_d I_d\right\|_F^2
=o_\Pp(1).
\end{equation}
In particular,
\[
\frac1d\left\|(UDU^\top)_{\diag}\right\|_F^2
=\bar\rho_d^{\,2}+o_\Pp(1).
\]
\end{lemma}

\begin{proof}
Write
$q_j=(UDU^\top)_{jj}=\sum_iU_{ji}^2\rho_i$.
Since each row of $U$ is uniform on $S^{d-1}$, for $i\ne k$ we have
\[
\E U_{ji}^2=\frac1d,\qquad
\E U_{ji}^4=\frac{3}{d(d+2)},\qquad
\E[U_{ji}^2U_{jk}^2]=\frac1{d(d+2)}.
\]
Expanding $q_j^2$ using these identities gives
\[
\E\left[\frac1d\sum_{j=1}^d(q_j-\bar\rho_d)^2\right]
=
\frac{2}{d+2}
\left(
\frac1d\sum_{i=1}^d(\rho_i-\bar\rho_d)^2
\right).
\]
If $\|\rho\|_\infty\le C$, the right-hand side is at most
$2C^2/(d+2)$, so Markov's inequality proves
\eqref{eq:haar-diagonal-averaging} uniformly over the stated class.
Finally, $\sum_jq_j=\Tr D=d\bar\rho_d$, and hence
\[
\frac1d\sum_jq_j^2
=
\bar\rho_d^{\,2}
+\frac1d\sum_j(q_j-\bar\rho_d)^2,
\]
which gives the last assertion.
\end{proof}

\begin{lemma}[Spectral comparison value]
\label{lem:spectral-comparison-value}
Fix $\kappa\ge3$ and $\zeta\in\R$. Take the eigendecomposition $\mathsf G_d=U\diag(\sqrt d\,g_i)U^\top$ so
$\lambda_i(\mathsf G_d)=\sqrt d\,g_i$ and define
\[
V_d:=
\sup_{R\succeq0,\ R\ne0}
\frac{\bigl(\langle\mathsf G_d,R\rangle_F
+\zeta\sqrt d\,\Tr R\bigr)_+}
{\bigl(\|R\|_F^2+c_\kappa(\Tr R)^2/d\bigr)^{1/2}}.
\]
Let $\nu_d=d^{-1}\sum_i\delta_{g_i}$,
$m_d(\omega)=\int(x-\omega)_+\,d\nu_d(x)$, and
$s_d(\omega)=\int(x-\omega)_+^2\,d\nu_d(x)$, and let $\omega_d$ be the
unique solution of $\omega_d+\zeta=c_\kappa m_d(\omega_d)$.  Then
\begin{equation}
    \label{eq:comparison-value-limit}
    d^{-2}V_d^2
    \pto\chi_\kappa(\zeta)
\end{equation}
Moreover,
$d^{-2}V_d^2=s_d(\omega_d)+c_\kappa m_d(\omega_d)^2$, and whenever this
value is positive, an optimizer is given by
$R=U\diag(\sqrt d\,\rho_i)U^\top$ with
$\rho_i=(g_i-\omega_d)_+$.
\end{lemma}

\begin{proof}
By the semicircle law and the spectral-norm bound for Gaussian Wigner
matrices, the empirical law of the $g_i$'s converges almost surely to $\nu$,
and $\max_i|g_i|=\sqrt2+o_\Pp(1)$; see, e.g.,
\cite[Chapter~2]{anderson-guionnet-zeitouni2010random}.
For symmetric matrices $A,B$, von Neumann's trace inequality
\cite{bhatia2013matrix} gives
$\Tr(AB)\le\sum_i\lambda_i(A)\lambda_i(B)$ when the eigenvalues are
ordered decreasingly, with equality when their eigenbases are aligned.
We may therefore take $R$ diagonal in the eigenbasis of $\mathsf G_d$.
By homogeneity, writing $r_i=\sqrt d\,\rho_i$ gives
\begin{equation}
    \label{eq:projection-finite-var}
    d^{-2}V_d^2
    =
    \sup_{\substack{\rho_i\ge0\\\rho\ne0}}
    \frac{\left(d^{-1}\sum_i(g_i+\zeta)\rho_i\right)_+^2}
    {d^{-1}\sum_i\rho_i^2
    +c_\kappa(d^{-1}\sum_i\rho_i)^2}.
\end{equation}
Lemma~\ref{lem:positive-cone-variational} now gives
$d^{-2}V_d^2=s_d(\omega_d)+c_\kappa m_d(\omega_d)^2$ and, whenever this
value is positive, the stated optimizer.  Semicircle convergence and the
spectral-norm bound give local uniform convergence of $m_d,s_d$ to $m,s$.
Hence Lemma~\ref{lem:positive-cone-variational-stability} gives
$\omega_d\to\omega_{\kappa,\zeta}$ and
\eqref{eq:comparison-value-limit}.
\end{proof}

\begin{proposition}[Shifted cone-projection limit]
    \label{prop:shifted-cone-projections}
    For every fixed $\kappa>1$ and $\zeta\in\R$,
    \[
    \frac1{d^2}
    \left\|\Pi_{\mathbb S_\kappa^d}
    \bigl(\mathsf G_{d,\kappa}(\zeta)\bigr)\right\|_F^2
    \pto \chi_\kappa(\zeta)
    \]
    In particular, at $\zeta=0$,
    $d^{-2}\|\Pi_{\mathbb S_\kappa^d}(\mathsf G_d)\|_F^2
    \pto\alpha_\star(\kappa)$.
\end{proposition}

\begin{proof}
We start with the case $\kappa\ge3$.
The cone-projection identity \eqref{eq:cone-projection-support} gives
\[
\left\|\Pi_{\mathbb S_\kappa^d}
\bigl(\mathsf G_{d,\kappa}(\zeta)\bigr)\right\|_F
=
\sup_{\substack{R\succeq0\\R\ne0}}
\frac{\bigl(\langle\mathsf G_d,T_\kappa R\rangle_F
+\zeta\sqrt d\,\Tr R\bigr)_+}{\|T_\kappa R\|_F}.
\]
Here $c_\kappa\ge0$.  Let $V_d$ and the associated notation be as in
Lemma~\ref{lem:spectral-comparison-value}.  The diagonal rescaling in the
numerator is lower order, since
\[
|\langle\mathsf G_d,T_\kappa R\rangle_F
-\langle\mathsf G_d,R\rangle_F|
\le
C(\kappa)\|(\mathsf G_d)_{\diag}\|_F\|T_\kappa R\|_F
=
O_\Pp(\sqrt d)\|T_\kappa R\|_F.
\]
Moreover,
$\|T_\kappa R\|_F^2
=\|R\|_F^2+c_\kappa\|R_{\diag}\|_F^2
\ge\|R\|_F^2+c_\kappa(\Tr R)^2/d$
for $R\succeq0$.  Hence, uniformly over $R\succeq0$,
$\|\Pi_{\mathbb S_\kappa^d}(\mathsf G_{d,\kappa}(\zeta))\|_F
\le V_d+O_\Pp(\sqrt d)$.
Also $V_d\le\|\mathsf G_d+\zeta\sqrt d\,I_d\|_F=O_\Pp(d)$, so
\begin{equation}
    \label{eq:shifted-projection-comparison-upper}
    d^{-2}
    \left\|\Pi_{\mathbb S_\kappa^d}
    \bigl(\mathsf G_{d,\kappa}(\zeta)\bigr)\right\|_F^2
    \le d^{-2}V_d^2+o_\Pp(1).
\end{equation}
Together with \eqref{eq:comparison-value-limit}, this gives the upper bound.

For the matching lower bound, suppose
$\chi_\kappa(\zeta)>0$, since otherwise there is nothing to prove.
Then $d^{-2}V_d^2>0$ with probability tending to one, so the optimizer from Lemma~\ref{lem:spectral-comparison-value} is $R_d=U\diag(\sqrt d\,\rho_i)U^\top$ with $\rho_i=(g_i-\omega_d)_+$ satisfying
$\|\rho\|_\infty=O_\Pp(1)$.  Conditional on the eigenvalues,
Lemma~\ref{lem:haar-diagonal-averaging} gives
\[
d^{-2}\|R_d\|_F^2=s_d(\omega_d),
\qquad
d^{-2}\|(R_d)_{\diag}\|_F^2
=m_d(\omega_d)^2+o_\Pp(1).
\]
Thus
$d^{-2}\|T_\kappa R_d\|_F^2=d^{-2}V_d^2+o_\Pp(1)$.
Also, using the defining equation for $\omega_d$,
\[
d^{-2}\bigl(\langle\mathsf G_d,T_\kappa R_d\rangle_F
+\zeta\sqrt d\,\Tr R_d\bigr)
=
d^{-2}V_d^2+o_\Pp(1).
\]
Since $V_d/d\pto\sqrt{\chi_\kappa(\zeta)}>0$,
evaluating the variational representation at $R_d$ gives
\[
\frac1d\left\|\Pi_{\mathbb S_\kappa^d}
\bigl(\mathsf G_{d,\kappa}(\zeta)\bigr)\right\|_F
\ge \frac{V_d}{d}+o_\Pp(1).
\]
Together with the upper bound above and \eqref{eq:comparison-value-limit},
this proves the result for $\kappa\ge3$.

\paragraph{Case $1<\kappa<3$.}
Set
\[
\kappa^\vee:=1+\frac{4}{\kappa-1}>3,
\qquad
\zeta^\vee:=-\frac{2\zeta}{\kappa-1}.
\]
Lemma~\ref{lem:rescaled-psd-cone-duality} gives
$(\mathbb S_\kappa^d)^\circ=-\mathbb S_{\kappa^\vee}^d$,
while Lemma~\ref{lem:shifted-duality} gives
\[
-\mathsf G_{d,\kappa}(\zeta)
\stackrel d=
\mathsf G_{d,\kappa^\vee}(\zeta^\vee).
\]
Since $\Pi_{-\mathcal C}(Z)=-\Pi_{\mathcal C}(-Z)$, Moreau's norm identity
 gives
\[
\left\|\Pi_{\mathbb S_\kappa^d}
\bigl(\mathsf G_{d,\kappa}(\zeta)\bigr)\right\|_F^2
+
\left\|\Pi_{\mathbb S_{\kappa^\vee}^d}
\bigl(-\mathsf G_{d,\kappa}(\zeta)\bigr)\right\|_F^2
=
\|\mathsf G_{d,\kappa}(\zeta)\|_F^2.
\]
By the duality identity above and the already proved
$\kappa^\vee>3$ case, the second term divided by $d^2$ converges to
$\chi_{\kappa^\vee}(\zeta^\vee)$.  The law of large numbers gives
\[
d^{-2}\|\mathsf G_{d,\kappa}(\zeta)\|_F^2
\to
\frac12+\frac{2\zeta^2}{\kappa-1}.
\]
Hence
\[
d^{-2}
\left\|\Pi_{\mathbb S_\kappa^d}
\bigl(\mathsf G_{d,\kappa}(\zeta)\bigr)\right\|_F^2
\to
\frac12+\frac{2\zeta^2}{\kappa-1}
-\chi_{\kappa^\vee}(\zeta^\vee)
=
\chi_\kappa(\zeta),
\]
where the last equality is the scalar identity in
Lemma~\ref{lem:shifted-duality}.  This completes the proof.
\end{proof}

\paragraph{Gaussian optimization limits}
The shifted-projection formula now identifies the signed auxiliary value,
and the CGMT comparison transfers this limit to the primary optimization.

\begin{lemma}
	\label{lem:gaussian-auxiliary-limit}
	If $n/d^2\to\alpha\in(0,\infty)$, then
	$$
	\widehat{\phi}_d\pto\zeta_{\alpha,\kappa},
	\qquad
	\Phi_d\pto\bigl(\zeta_{\alpha,\kappa}\bigr)_+.
	$$
\end{lemma}
\begin{proof}
	Set $\zeta_*=\zeta_{\alpha,\kappa}$.
	By Proposition~\ref{prop:scalar-gaussian-energy-curve},
	$\chi_\kappa(\zeta_*)=\alpha>0$, and $\chi_\kappa$ is strictly increasing
	wherever it is positive. Thus, for every $\eps>0$,
	$\chi_\kappa(\zeta_*-\eps)<\alpha<\chi_\kappa(\zeta_*+\eps)$.
        Proposition~\ref{prop:shifted-cone-projections} and
	$\|g\|_2/d\pto\sqrt\alpha$ therefore imply, with
	probability tending to one,
	$$
	\left\|\Pi_{\mathbb S_\kappa^d}
	\bigl(\mathsf G_{d,\kappa}(\zeta_*-\eps)\bigr)\right\|_F
	<\|g\|_2
	<
	\left\|\Pi_{\mathbb S_\kappa^d}
	\bigl(\mathsf G_{d,\kappa}(\zeta_*+\eps)\bigr)\right\|_F.
	$$
	The crossing formula
	\eqref{eq:signed-auxiliary-projection-crossing} therefore yields
	$\zeta_*-\eps\le\widehat{\phi}_d<\zeta_*+\eps$ with probability tending
	to one. This proves the convergence of the
	signed auxiliary value. Since the CGMT auxiliary value is
	$(\widehat{\phi}_d)_+$, the two comparisons in
	\eqref{eq:cgmt-comparisons} then give the convergence of
	$\Phi_d$: apply the upper-tail inequality above
	$(\zeta_*)_++\eps$, and, when $(\zeta_*)_+>0$, the lower-tail inequality
	below $(\zeta_*)_+-\eps$.  If $(\zeta_*)_+=0$, the lower bound is automatic.
\end{proof}

\subsection{Well-conditioned fit}

In the SAT phase there are many shape matrices which fit the data.  For the
exact fitting result we need the Gaussian analogue to have a solution with uniformly bounded
condition number, which the following lemma guarantees exists.

\begin{lemma}[Well-conditioned exact fit below the threshold]
    \label{lem:gaussian-bounded-witness}
    Suppose that $n/d^2\to\alpha$ with
    $0<\alpha<\alpha_\star(\kappa)$.  Then there are constants
    $0<m<M<\infty$, depending only on $\kappa$ and the distance from
    $\alpha$ to $\alpha_\star(\kappa)$, such that
    \[
    \Pp\left\{
        \exists R:\ 
        mI_d\preceq R\preceq MI_d,\quad
        \Tr R=d,\quad
        \mathcal A_{G,\kappa}R=0
    \right\}\to1 .
    \]
\end{lemma}

\begin{proof}
Let $p(x)=(x-\omega_\kappa)_+$.  By
Proposition~\ref{prop:scalar-gaussian-energy-curve} and the equality case in
Lemma~\ref{lem:positive-cone-variational},
\[
\frac{\int xp\,d\nu}
{\left(\int p^2\,d\nu+c_\kappa(\int p\,d\nu)^2\right)^{1/2}}
=\sqrt{\alpha_\star(\kappa)}.
\]
Since $\alpha<\alpha_\star(\kappa)$, continuity allows us to choose
$\ell>0$ sufficiently small so that, for $q:=p+\ell$,
\[
\frac{\int xq\,d\nu}
{\left(\int q^2\,d\nu+c_\kappa(\int q\,d\nu)^2\right)^{1/2}}
>\sqrt\alpha.
\]
Write
$\mathsf G_d=U\diag(\sqrt d\,g_i)U^\top$ and set
$\bar q_d=d^{-1}\sum_iq(g_i)$ and
\[
R_q:=U\diag\left(\frac{q(g_i)}{\bar q_d}\right)U^\top.
\]
Fix $C>\sqrt2$.  By the empirical semicircle law and GOE spectral-norm
convergence, with probability tending to one,
$\bar q_d\to\int q\,d\nu$ and $\max_i|g_i|\le C$.  On this event
$\Tr R_q=d$ and, for deterministic constants $0<m<M<\infty$,
$mI_d\preceq R_q\preceq MI_d$.
Moreover, the empirical semicircle law and
Lemma~\ref{lem:haar-diagonal-averaging} give
\[
\frac1d
\frac{\langle\mathsf G_d,T_\kappa R_q\rangle_F}
{\|T_\kappa R_q\|_F}
\to
\frac{\int xq\,d\nu}
{\left(\int q^2\,d\nu+c_\kappa(\int q\,d\nu)^2\right)^{1/2}}
>\sqrt\alpha.
\]
Let $\mathcal K_{m,M}:=\{R:mI_d\preceq R\preceq MI_d,\ \Tr R=d\}$.
Since $R_q\in\mathcal K_{m,M}$ and $\|g\|_2/d\to\sqrt\alpha$, the preceding
limit implies, with probability tending to one that
$\|g\|_2\|T_\kappa R_q\|_F
-\langle\mathsf G_d,T_\kappa R_q\rangle_F<0$.
Hence the CGMT auxiliary value
\[
\phi_d^{m,M}:=
\inf_{R\in\mathcal K_{m,M}}
\left(
\|g\|_2\|T_\kappa R\|_F
-\langle\mathsf G_d,T_\kappa R\rangle_F
\right)_+
\]
is equal to zero with probability tending to one.

Now set
$\Phi_d^{m,M}:=\inf_{R\in\mathcal K_{m,M}}\|\mathsf A T_\kappa R\|_2$.
Since $T_\kappa\mathcal K_{m,M}$ is compact and convex, the CGMT gives,
for every $c>0$,
\[
\Pp\{\Phi_d^{m,M}>c\}
\le2\Pp\{\phi_d^{m,M}\ge c\}.
\]
Letting $c\downarrow0$ shows that
$\Pp\{\Phi_d^{m,M}>0\}\to0$.  Therefore, with probability tending to one,
compactness of $\mathcal K_{m,M}$ yields some
$R\in\mathcal K_{m,M}$ with $\mathsf A T_\kappa R=0$.
By \eqref{eq:projected-gaussian-A-representation},
$\mathcal A_{G,\kappa}R=0$, as claimed.
\end{proof}

Now we can easily conclude the proof of Proposition~\ref{thm:gaussian-master} from the introduction.

\begin{proof}[Proof of Proposition~\ref{thm:gaussian-master}]
Part~\textup{(i)} is Lemma~\ref{lem:gaussian-bounded-witness}.
For part~\textup{(ii)}, Lemma~\ref{lem:gaussian-auxiliary-limit} gives
$\Phi_d\pto(\zeta_{\alpha,\kappa})_+$.  By
\eqref{eq:projected-gaussian-A-representation},
\[
\Gamma_{G,\kappa}
\stackrel d=
\frac{2d^2}{n}\Phi_d^2.
\]
Since $d^2/n\to1/\alpha$, the continuous mapping theorem gives the claimed
limit.  If $\alpha>\alpha_\star(\kappa)$, then
$e_\star(\alpha,\kappa)>0$ by
\eqref{eq:energy-positive-iff-above-threshold}; taking
$c_0=e_\star(\alpha,\kappa)/2$ proves the final assertion.
\end{proof}

	\appendix 
	\section{Extra corollaries, obstructions, and numerics}
	\label{sec:extras}
	
	This appendix records an abstract consequence of the universality theorem,
	several obstructions to exact ellipsoid fitting, and numerical details that
	are not used in the proof of the phase transition.

\subsection{Dual certificates and interpretation}
\label{sec:unsat-dual-certificate}

Recall that
\[
	\mathcal R_d=\{R\succeq0:\Tr R=d\},
	\qquad
	\Gamma_X=\inf_{R\in\mathcal R_d}\frac1n\|\mathcal A_XR\|_2^2,
\]
where $\mathcal A_X=P_{\one^\perp}\mathcal L_X$.  We briefly record the
dual interpretation of this optimization problem in the SAT and UNSAT regimes.

The boundary of the PSD cone has a simple geometric interpretation.
If $Q\succeq0$ is singular and $K=\ker Q$, then $Q$ is positive definite on
$K^\perp$, while the set $\{x:x^\top Qx\le1\}$ is unrestricted in the
directions of $K$.  Equivalently, it is the product of a bounded ellipsoid in
$K^\perp$ with the linear space $K$.  Thus boundary points correspond to
 cylindrical degenerations of ellipsoids.  Positive
definiteness is precisely what distinguishes the genuine ellipsoid supplied
by the SAT theorem from such a degenerate fit.

\paragraph{UNSAT and the dual certificate.}
In the UNSAT regime, the positive least-squares energy has an exact dual
interpretation in terms of sample reweighting, which we record below.

\begin{proposition}[Dual certificate identity]
	\label{prop:exact-unsat-dual-energy}
	For every deterministic sample $X$,
	\begin{equation}
		\label{eq:exact-dual-energy-identity}
		\max_{\substack{\lambda\in\one^\perp\\ \|\lambda\|_2\le1}}
		\lambda_{\min}(\mathcal A_X^*\lambda)
		=
		\frac{\sqrt{n\Gamma_X}}{d}.
	\end{equation}
	In particular, $\Gamma_X>0$ if and only if there exists
	$\lambda\in\one^\perp$ such that
	\[
		\mathcal A_X^*\lambda
		=
		\frac1{\sqrt d}\sum_{i=1}^n\lambda_i x_i x_i^\top \succ 0.
	\]
\end{proposition}
\begin{proof}
    This is a special case of general SDP duality, see \cite{boyd-vandenberghe2004convex}, but we can also prove it directly.
	We first prove the weak-duality bound.  For any
	$\lambda\in\one^\perp$ with $\|\lambda\|_2\le1$ and any
	$R\in\mathcal R_d$,
	\[
		d\,\lambda_{\min}(\mathcal A_X^*\lambda)
		\le
		\langle\mathcal A_X^*\lambda,R\rangle
		=
		\langle\lambda,\mathcal A_XR\rangle
		\le
		\|\mathcal A_XR\|_2.
	\]
	Taking the infimum over $R\in\mathcal R_d$ gives
	\[
		\lambda_{\min}(\mathcal A_X^*\lambda)
		\le
		\frac{\sqrt{n\Gamma_X}}{d},
	\]
    for any such $\lambda$.

	It remains to show that this bound is attained.  If $\Gamma_X=0$, then
	$\lambda=0$ gives equality.  Suppose therefore that $\Gamma_X>0$.  Let
	\[
		C_X:=\mathcal A_X(\mathcal R_d),
	\]
	let $R_\star$ minimize the least-squares objective, and set
	$y_\star:=\mathcal A_XR_\star$.  Then
	$\|y_\star\|_2=\sqrt{n\Gamma_X}$ and
	$y_\star=\arg\min_{y\in C_X}\|y\|_2^2$.
	By first-order optimality of
	$y_\star=\arg\min_{y\in C_X}\|y\|_2^2$,
	\[
		\langle y_\star,y-y_\star\rangle\ge0
		\qquad\forall y\in C_X.
	\]
	Therefore, with $y=\mathcal A_XR$ and
	$\lambda_\star:=y_\star/\|y_\star\|_2$,
	\[
		\langle\lambda_\star,\mathcal A_XR\rangle
		\ge
		\|y_\star\|_2
		=
		\sqrt{n\Gamma_X}
		\qquad\forall R\in\mathcal R_d.
	\]
	Equivalently,
	$\langle\mathcal A_X^*\lambda_\star,R\rangle
	\ge\sqrt{n\Gamma_X}$ for every $R\succeq0$ with $\Tr R=d$.  Since
	\[
		\min_{\substack{R\succeq0\\\Tr R=d}}
		\langle M,R\rangle
		=
		d\,\lambda_{\min}(M),
	\]
	it follows that
	\[
		\lambda_{\min}(\mathcal A_X^*\lambda_\star)
		\ge
		\frac{\sqrt{n\Gamma_X}}{d}.
	\]
	Combining this with the weak-duality bound proves
	\eqref{eq:exact-dual-energy-identity}.
\end{proof}

Thus, when $\Gamma_X>0$, the normalized optimal residual
$\lambda_\star=\mathcal A_XR_\star/\|\mathcal A_XR_\star\|_2$ itself gives a
strict positive-definite separating certificate.  Combining the above proposition with
Theorem~\ref{thm:phase-transition}, we find
\[
	\max_{\substack{\lambda\in\one^\perp\\ \|\lambda\|_2\le1}}
	\lambda_{\min}(\mathcal A_X^*\lambda)
	\pto
	\sqrt{\alpha\,e_\star(\alpha,\kappa)}.
\]

	\subsection{A general-domain universality consequence}
	\label{subsec:abstract-barrier-removal}

	Section~\ref{sec:self-concordant-ridge-universality} separates the two
	differential conditions defining a $\vartheta$-self-concordant barrier.
	The fixed-$\gamma$ row comparison follows from the normalized local
	condition $\mathrm{(SC)}$.
	Lemma~\ref{lem:general-barrier-removal-sandwich} uses the
	barrier-parameter bound $(\mathrm{BP}_\vartheta)$ to remove the auxiliary
	penalty.  Combining those two results gives the following statement for a general domain.

	\begin{proposition}
		\label{prop:constrained-ridge-comparison}
		Fix $N,n\ge1$, $K_{\rm feat}<\infty$, $\beta>0$, and
		$\vartheta\ge1$.  Let
		$W\in\R^N$ be a $(K_{\rm feat},\tau)$-regular feature with
		covariance matrix $\Sigma$, put
		$\lambda_{\Sigma}=\|\Sigma\|_{\op}>0$, and let $G$ be the matching centered
		Gaussian feature.  Let $W_1,\ldots,W_n$ and $G_1,\ldots,G_n$ be
		independent arrays of iid copies of $W$ and $G$, respectively.
		There are constants $C<\infty$ and $c_0>0$, depending only on
		$K_{\rm feat}$, with the following property.

		Let
		$\mathcal D\subset\R^N$ be nonempty, bounded, convex, and relatively
		open in its affine hull, and let
		$\mathcal B:\mathcal D\to[0,\infty)$ be a
		$\vartheta$-self-concordant barrier.  Fix $(\theta_{\rm ref},b_{\rm ref})\in\mathcal D\times\R$, and set
		\[
			K_{\rm ref}:=
			1+\lambda_{\Sigma}\|\theta_{\rm ref}\|_2^2+b_{\rm ref}^2
			+\mathcal B(\theta_{\rm ref}).
		\]
		Put
		\[
			\mathfrak L_n:=1+\vartheta\log n
		\]
		and
		\[
			\Delta_n:=
			(1+\beta^{-1})^2K_{\rm ref}^2\left[
			\frac{\sqrt{\dim\mathsf V_{\mathcal D}}(\tau/\lambda_{\Sigma}+1)}{n}
			+\left(\frac{\mathfrak L_n}{n}\right)^{1/3}
			\right]
			+\frac{(1+\beta)K_{\rm ref}}{n}.
		\]
		If $\mathfrak L_n\le c_0n$, then
		\begin{equation}
			\label{eq:ridge-universality-expectation-theta-barrier}
			\left|
			\E\Phi_W(\mathcal D;\beta)
			-\E\Phi_G(\mathcal D;\beta)
			\right|\le C\Delta_n.
		\end{equation}
		Moreover, for every $t>0$,
		\begin{equation}
			\label{eq:ridge-universality-high-prob-theta-barrier}
			\Pp\left\{
			\left|\Phi_W(\mathcal D;\beta)-\Phi_G(\mathcal D;\beta)\right|
			>C\Delta_n+t
			\right\}
			\le\frac{C(1+\beta^{-1})^2K_{\rm ref}^2}{nt^2}.
		\end{equation}
	\end{proposition}

	\begin{proof}[Proof of Proposition~\ref{prop:constrained-ridge-comparison}]
		Let $n_0$ be the constant in
		Theorem~\ref{thm:positive-barrier-row-universality}.  Choose
		$c_0\le n_0^{-3/2}$, and set
		$\varepsilon=n^{-1}$ and
		$\gamma=n^{-1/3}\mathfrak L_n^{-2/3}$.  The hypothesis on
		$\mathfrak L_n$ gives $n\gamma\ge n_0$, while $\gamma\le1$.
		The intrinsic scale in
		Theorem~\ref{thm:positive-barrier-row-universality} satisfies
		\[
			\Lambda_{\rm ref}
			\le
			(1+\beta^{-1})
			\left[
			1+\lambda_{\Sigma}\|\theta_{\rm ref}\|_2^2
			+\gamma\mathcal B(\theta_{\rm ref})
			\right]
			\le (1+\beta^{-1})K_{\rm ref}.
		\]
		Thus that theorem applies with this upper bound.

		Lemma~\ref{lem:general-barrier-removal-sandwich}, applied with
		the anchor $(\theta_{\rm ref},b_{\rm ref})$, gives, for $Z=W,G$,
		\begin{equation}
			\label{eq:barrier-ridge-only-sandwich}
			0\le
			\Phi_Z(\mathcal D,\mathcal B;\beta,\gamma)
			-\Phi_Z(\mathcal D;\beta)
			\le
			\varepsilon F_Z(\theta_{\rm ref},b_{\rm ref};\beta)
			+\gamma\mathcal B(\theta_{\rm ref})
			+\gamma\vartheta\log(1/\varepsilon).
		\end{equation}

		Centering and covariance matching imply
		$\E F_Z(\theta_{\rm ref},b_{\rm ref};\beta)
		\le C(1+\beta)K_{\rm ref}$.
		Our choice of $\gamma$ balances the fixed-$\gamma$ row-comparison
		error with the barrier-removal cost:
		\[
		\frac{(1+\beta^{-1})^2K_{\rm ref}^2}{\sqrt{n\gamma}}
		+\gamma\mathcal B(\theta_{\rm ref})+\gamma\vartheta\log n
		\le C(1+\beta^{-1})^2K_{\rm ref}^2
		\left(\frac{\mathfrak L_n}{n}\right)^{1/3}.
		\]
		Taking expectations in \eqref{eq:barrier-ridge-only-sandwich} and using
		\eqref{eq:feature-universality-expectation} proves
		\eqref{eq:ridge-universality-expectation-theta-barrier}.

		For the probability bound, combine the same sandwich with
		\eqref{eq:feature-universality-high-probability}.  Condition
		$\mathrm{(F1)}$ and the Gaussian fourth-moment bound give
		$\Var(F_Z(\theta_{\rm ref},b_{\rm ref};\beta))
		\le CK_{\rm ref}^2/n$ for $Z=W,G$.  Chebyshev's inequality
		therefore controls the random anchor
		term in \eqref{eq:barrier-ridge-only-sandwich} by
		$C(1+\beta^{-1})^2K_{\rm ref}^2/(nt^2)$; the deterministic terms
		are bounded by $C\Delta_n$.  This proves
		\eqref{eq:ridge-universality-high-prob-theta-barrier}.
	\end{proof}	
	\subsection{Why approximate fitting does not imply exact fitting}
	\label{subsec:exact-fitting-obstructions}
	
 Recall that Remark~\ref{lem:zero-inflated-sphere-goe-matching-unsat} rules out a
	covariance-only extension of exact fitting.  The following constructions
	show that the distinction persists under the abstract regular-feature
	assumptions (cf. Definition~\ref{def:regular-feature-row}) and even for a nonatomic isotropic law with uniformly
	subgaussian law.

	\emph{Feature regularity is not enough.}
	Example~\ref{lem:zero-inflated-sphere-goe-matching-unsat} is not a regular
	feature.  We now add rare copies of
	$I_d/d$ and $-I_d/d$ to a Gaussian background while preserving both the
	target covariance and the regular-feature estimates.  Since every
	trace-normalized matrix pairs with these two atoms as $1$ and $-1$, the
	occurrence of both signs rules out exact centered fitting.

	\begin{lemma}
		\label{lem:rare-trace-atoms-obstruction}
		Fix $\kappa>1$. There is a sequence of random features
		$W_d\in\Ssym^d$ with $\E[W_d]=0$ and $\operatorname{Cov}(W_d)=\Sigma_\kappa$ such that
		$W_d$ is $(K_{\rm feat},\tau_d)$-regular with
		$K_{\rm feat}=O_\kappa(1)$ and $\tau_d=O_\kappa(d^{-1/2})$.  Nevertheless,
		if $W_1,\ldots,W_n$ are iid copies and $n/d^2\to\alpha>0$, then with
		probability tending to one,
		\begin{equation}
			\label{eq:rare-trace-atoms-lower-bound}
			\inf_{R\in\mathcal R_d}\frac1{\sqrt n}\|\mathcal A_WR\|_2
			\ge c_\alpha d^{-1/2}
		\end{equation}
		for some $c_\alpha>0$.  In particular, with probability tending to one
		there is no $R\in\mathcal R_d$ for which $\mathcal A_WR=0$.
	\end{lemma}

	\begin{proof}
		Set $v_d=I_d/d$ and $p_d=1/d$.  The operator $\Sigma_\kappa$ has
		$e_d=I_d/\sqrt d$ as an eigenvector with eigenvalue
		$(\kappa-1)/d$, while $v_d\otimes v_d=d^{-1}e_d\otimes e_d$.  Hence
		$\Sigma_\kappa-p_dv_d\otimes v_d\succeq0$ for all large $d$.  Let $G_d$
		be centered Gaussian with covariance
		$\Sigma'_d=(\Sigma_\kappa-p_dv_d\otimes v_d)/(1-p_d)$, and define $W_d$ by
		\[
		W_d=
		\begin{cases}
			G_d,&\text{with probability }1-p_d,\\
			v_d,&\text{with probability }p_d/2,\\
			-v_d,&\text{with probability }p_d/2,
		\end{cases}
		\]
		independently of $G_d$.  The resulting law is centered and has covariance
		$(1-p_d)\Sigma'_d+p_dv_d\otimes v_d=\Sigma_\kappa$.

		We next verify regularity.  Since $\|\Sigma_\kappa\|_{\op}\asymp_\kappa
		d^{-1}$ and $\|\Sigma'_d\|_{\op}\le C_\kappa d^{-1}$, the Gaussian and
		discrete components satisfy, respectively,
		\[
			\|\langle G_d,A\rangle_F\|_{L^4}
			\le C_\kappa d^{-1/2}\|A\|_F
			\le C_\kappa \|\Sigma_\kappa\|_{\op}^{1/2}\|A\|_F,
			\qquad
		|\langle v_d,A\rangle_F|
		\le d^{-1/2}\|A\|_F
		\le C_\kappa \|\Sigma_\kappa\|_{\op}^{1/2}\|A\|_F .
		\]
		Thus $\mathrm{(F1)}$ holds with a constant depending only on $\kappa$.

		For $\mathrm{(F2)}$, fix a deterministic self-adjoint
		$T:\Ssym^d\to\Ssym^d$ and put $Q=\langle W_d,TW_d\rangle_F$.  On the
		Gaussian part,
		\[
		\Var\bigl(\langle G_d,TG_d\rangle_F\bigr)
		=2\|(\Sigma'_d)^{1/2}T(\Sigma'_d)^{1/2}\|_{\HS}^2
		\le C_\kappa d^{-2}\|T\|_{\HS}^2 .
		\]
		The variance from switching components is controlled by
		$\|\Sigma'_d\|_{\HS}=O_\kappa(1)$,
		$|\operatorname{tr}(T\Sigma'_d)|\le C_\kappa\|T\|_{\HS}$, and
		$|\langle v_d,Tv_d\rangle_F|\le d^{-1}\|T\|_{\HS}$.  The law of total
		variance therefore gives
		\[
		\Var(Q)\le C_\kappa(d^{-2}+p_d)\|T\|_{\HS}^2
		\le C_\kappa d^{-1}\|T\|_{\HS}^2,
		\]
		which is $\mathrm{(F2)}$ with $\tau_d=O_\kappa(d^{-1/2})$.

		It remains to use the signed atoms.  Let $N_+$ and $N_-$ count the copies of
		$+v_d$ and $-v_d$.  Since $np_d\to\infty$,
		\[
		\Pp\{N_+=0\}+\Pp\{N_-=0\}
		\le 2(1-p_d/2)^n\longrightarrow0.
		\]
		For every $R\in\mathcal R_d$, these two row values are $1$ and $-1$.
		Consequently,
		\[
		\|\mathcal A_WR\|_2^2
		=\min_{b\in\R}\sum_{i=1}^n
		\bigl(\langle W_i,R\rangle_F-b\bigr)^2
		\ge
		\min_{b\in\R}\left\{
		N_+(1-b)^2+N_-(-1-b)^2
		\right\}.
		\]
		The last minimum is $4N_+N_-/(N_++N_-)$.  Since
		$N_+/d\to\alpha/2$ and $N_-/d\to\alpha/2$ in probability, it equals
		$(\alpha+o_\Pp(1))d$.  This proves
		\eqref{eq:rare-trace-atoms-lower-bound}, and hence also rules out exact
		fitting.
	\end{proof}

	Thus $\mathrm{(F1)}$ and $\mathrm{(F2)}$, even with
	$\tau_d=O_\kappa(d^{-1/2})$, do not force exact fitting.  In this example the
	best normalized residual is at least of order $d^{-1/2}$, or equivalently the
	energy $n^{-1}\|\mathcal A_WR\|_2^2$ is at least of order $d^{-1}$.
	
	\emph{A nonatomic geometric obstruction.}
	Atoms are not essential.  The simplest geometric example is a nonatomic
	distribution on a line: restricting $x^\top Rx=\rho$ to that line gives a
	quadratic equation in one variable, so a centered ellipsoid contains at most
	two of its points.  B\'ezout's theorem gives the same mechanism for any
	algebraic curve of controlled degree.
	
	Let $V\subset\R^d$ be a real affine algebraic curve.  Its
	\emph{degree} means the degree of the complex projective Zariski closure
	\begin{equation}
		\label{eq:complex-projective-closure-curve}
		C:=\overline{\{[1:x]:x\in V\}}^{\,\mathrm{Zar}}
		\subset \mathbb P^d_{\mathbb C}.
	\end{equation}
	We assume for simplicity that $C$ is irreducible and has degree $D$.  The
	same statements hold for a pure one-dimensional projective variety of total
	degree $D$, provided no irreducible component is contained in the quadric
	under consideration.
	
	\begin{lemma}[A curve meets an ellipsoid in at most $2D$ points]
		\label{lem:bezout-curve-ellipsoid}
		Let $V\subset\R^d$ and $C\subset\mathbb P^d_{\mathbb C}$ be as above,
		with $\deg C=D$.  For $R\succ0$ and $\rho>0$, set
		\[
		E_{R,\rho}:=\{x\in\R^d:x^\top R x=\rho\}.
		\]
		If $C$ is not contained in the projective quadric obtained by
		homogenizing $x^\top R x-\rho$, then
		\begin{equation}
			\label{eq:bezout-ellipsoid-intersection-bound}
			\#\bigl(V\cap E_{R,\rho}\bigr)\le 2D.
		\end{equation}
		In particular, this holds for every centered ellipsoid if $V$ is not
		contained in any centered ellipsoid.  A sufficient condition is that
		$V$ be unbounded.
	\end{lemma}
	
	\begin{proof}
		In homogeneous coordinates $[X_0:X]$, the ellipsoid equation becomes
		\[
		\widetilde q_{R,\rho}(X_0,X):=X^\top R X-\rho X_0^2,
		\]
		whose zero set $Q_{R,\rho}$ has degree two.  Since
		$C\not\subset Q_{R,\rho}$, their intersection is zero-dimensional.  The
		hypersurface-section form
		of projective B\'ezout's theorem states that the sum of its intersection
		multiplicities over $\mathbb C$ is
		\[
		\deg(C)\deg(Q_{R,\rho})=2D;
		\]
		see \cite[Chapter~I, \S7]{hartshorne1977algebraic} or
		\cite[\S8.4]{fulton1998intersection}.  Thus there are at most $2D$
		distinct complex intersections and hence at most $2D$ real affine
		intersections.  Finally, an unbounded real curve cannot be contained in the
		compact ellipsoid $E_{R,\rho}$.
	\end{proof}
	
	The count includes complex points, points at infinity, and tangencies with
	multiplicity, so it remains an upper bound for distinct real affine points.
	Noncontainment is essential.  The same proof applies to translated
	ellipsoids; treating singular PSD fits additionally requires noncontainment in
	the corresponding degenerate quadrics.  We formulate the result for positive
	definite fitting matrices, as in the SAT theorem.
	
	We now turn this deterministic intersection bound into a probabilistic
	obstruction for mixture models.
	
	\begin{lemma}[Mixtures with a low-degree curve component]
		\label{prop:algebraic-curve-mixture-obstruction}
		For each $d$, let $V_d\subset\R^d$ be an affine algebraic curve whose
		irreducible complex projective closure has degree $D_d$, and suppose
		$V_d$ is not contained in any centered ellipsoid.  Let $\nu_{V,d}$ be a
		non-atomic probability measure supported on $V_d$, let $\nu_{0,d}$ be any
		probability measure on $\R^d$, and set
		$\mu_d=(1-p_d)\nu_{0,d}+p_d\nu_{V,d}$ for $0<p_d\le1$.  If
		$x_1,\ldots,x_n$ are iid from $\mu_d$, then
		\begin{equation}
			\label{eq:algebraic-mixture-sat-probability-bound}
			\Pp\left\{\begin{array}{c}
				\text{there exist }R\succ0\text{ and }\rho>0\text{ such that}\\[-1mm]
				x_i^\top R x_i=\rho\quad\text{for every }i\le n
			\end{array}\right\}
			\le \Pp\{\operatorname{Binomial}(n,p_d)\le 2D_d\}.
		\end{equation}
		Consequently, for every fixed $\delta\in(0,1)$,
		\begin{equation}
			\label{eq:bezout-mixture-chernoff}
			2D_d\le(1-\delta)np_d
			\quad\Longrightarrow\quad
			\Pp\{\text{an exact fit exists}\}
			\le \exp\left(-\frac{\delta^2}{2}np_d\right).
		\end{equation}
		In particular, exact fitting fails with high probability whenever
		$np_d\to\infty$ and $D_d=o(np_d)$.  If
		$n/d^2\to\alpha\in(0,\infty)$ and $\liminf_d p_d>0$, this includes every
		$D_d=o(d^2)$; more generally, the same conclusion holds whenever
			\[
			\limsup_{d\to\infty}\frac{2D_d}{d^2}
			<\alpha\liminf_{d\to\infty}p_d.
			\]
	\end{lemma}
	
	\begin{proof}
		Use independent Bernoulli labels to identify the observations drawn from
		$\nu_{V,d}$.  Their number $M_d$ has law
$\operatorname{Binomial}(n,p_d)$, and, since $\nu_{V,d}$ is nonatomic,
these observations are pairwise distinct almost surely.  If one ellipsoid fits the entire sample, then
		Lemma~\ref{lem:bezout-curve-ellipsoid} forces $M_d\le2D_d$.  This proves
		the first bound, and the second is the standard binomial lower-tail Chernoff
		bound.  The asymptotic conclusions follow by choosing a fixed $\delta>0$
		for all sufficiently large $d$.
	\end{proof}
	
	\emph{Rational normal curves.}
	For $1\le D\le d$, consider the affine moment curve
	\begin{equation}
		\label{eq:rational-normal-affine-curve}
		V_{D,d}:=\{(t,t^2,\ldots,t^D,0,\ldots,0):t\in\R\}\subset\R^d.
	\end{equation}
	Its projective closure is the degree-$D$ rational normal curve
	\[
	[s:t]\longmapsto
	[s^D:s^{D-1}t:\cdots:t^D:0:\cdots:0]
	\quad\text{in }\mathbb P^d;
	\]
	see, for example, \cite{harris1992algebraic}.  Since $V_{D,d}$ is unbounded,
	every centered ellipsoid meets it in at most $2D$ points.  Thus a mixture with
	fixed positive mass on a nonatomic part of this curve cannot be fit once more
	than $2D$ curve observations appear.  For $D=d$ and $n\asymp d^2$, there are
	typically order $d^2$ such observations, whereas an ellipsoid contains at most
	$2d$ of them.
	
	The degree-one case already gives a transparent example retaining strong
	one-dimensional concentration properties.
	
	\begin{example}[An isotropic uniformly subgaussian line mixture]
		\label{ex:isotropic-subgaussian-line-mixture}
		Fix $p\in(0,1)$.  Let $G\sim N(0,I_d)$, let $U$ be uniform on
		$[-1,1]$, and let $B\sim\operatorname{Bernoulli}(p)$, all independent.
		Define
		\[
		\widetilde x=(1-B)G+B Ue_1,
		\qquad
		\widetilde\Sigma:=\E\widetilde x\widetilde x^\top
		=(1-p)I_d+\frac p3 e_1e_1^\top,
		\qquad
		x:=\widetilde\Sigma^{-1/2}\widetilde x.
		\]
		Then $\E x=0$, $\E xx^\top=I_d$, and
		\begin{equation}
			\label{eq:line-mixture-subgaussian-bound}
			\sup_{\|u\|_2=1}\|\langle u,x\rangle\|_{\psitwo}
			\le C_p<\infty
		\end{equation}
		uniformly in $d$.  Indeed, $\lambda_{\min}(\widetilde\Sigma)\ge1-p$, while
		each linear form of $\widetilde x$ is a mixture of a centered Gaussian and a
		bounded random variable.  Conditional on $B=1$, the vector $x$ lies on the
		line $\widetilde\Sigma^{-1/2}\R e_1$ with a
		nonatomic conditional law.  Hence
		\[
		\Pp\{\text{the sample admits a positive definite exact fit}\}
		\le \Pp\{\operatorname{Binomial}(n,p)\le2\}\longrightarrow0.
		\]
		Thus isotropy and dimension-free subgaussian linear marginals do not prevent
		this obstruction; the law may still place substantial mass on a low-degree
		algebraic set.
	\end{example}

	\subsection{Nuclear-norm minimizers below the SAT threshold}
	Maillard and Kunisky \cite{maillard-kunisky2023replica} conjectured that a
	nuclear-norm minimizing exact fit is positive semidefinite in the SAT phase.
	We prove a quantitative approximation to this claim.  If the empirical
	covariance is close to the identity and a PSD exact fit exists, then every
	nuclear-norm minimizer has only a small negative spectral part.  With an
	additional trace constraint, positivity holds exactly.  These facts are not
	used elsewhere in the paper. For $x_1,\ldots,x_n\in\R^d$, write
	\[
	\widehat\Sigma:=\frac1n\sum_{i=1}^n x_i x_i^\top,
	\qquad
	\delta:=\|\widehat\Sigma-I_d\|_{\op}.
	\]
	For $S\in\Ssym^d$, let $S=S_+-S_-$ be its spectral decomposition into
	positive and negative parts, with $S_+,S_-\succeq0$ and $S_+S_-=0$.
	The mechanism is simple: a PSD competitor bounds the optimal nuclear norm,
	averaging the interpolation equations controls the trace, and
	$\|S\|_*-\Tr S=2\|S_-\|_*$ then controls the negative part.
	
	\begin{proposition}[Nuclear-norm exact fits are almost PSD]
		\label{prop:exact-nuclear-minimizers-almost-psd}
		Suppose $\delta<1$, and suppose there exists
		$S_{\rm psd}\succeq0$ satisfying
		\[
		x_i^\top S_{\rm psd}x_i=d,\qquad i=1,\ldots,n .
		\]
		Consider the exact nuclear-norm program
		\begin{equation}
			\label{eq:exact-nuclear-program}
			\min\left\{\|S\|_*:S\in\Ssym^d,\ 
			x_i^\top Sx_i=d\ \text{for all }i\le n\right\}.
		\end{equation}
		Then \eqref{eq:exact-nuclear-program} has minimizers, and every minimizer
		$\widehat S$ satisfies
		\[
		\left|\|\widehat S\|_*-d\right|,
		\ |\Tr\widehat S-d|,
		\ \|\widehat S_-\|_*
		\le\frac{\delta d}{1-\delta}.
		\]
	\end{proposition}
	
	\begin{proof}
		The PSD competitor makes the program feasible, and the corresponding
		nuclear-norm sublevel set is compact, so minimizers exist.  Moreover,
		$d=\langle\widehat\Sigma,S_{\rm psd}\rangle\ge
		(1-\delta)\Tr S_{\rm psd}$, and hence
		\[
		\|\widehat S\|_*
		\le \|S_{\rm psd}\|_*
		=\Tr S_{\rm psd}
		\le \frac d{1-\delta}.
		\]
		Averaging the interpolation constraints gives
		$d=\Tr\widehat S+\langle\widehat\Sigma-I_d,\widehat S\rangle$.  Nuclear--operator
		norm duality therefore yields
		\[
		|\Tr\widehat S-d|
		\le
		\|\widehat\Sigma-I_d\|_{\op}\,\|\widehat S\|_*
		\le \frac{\delta d}{1-\delta}.
		\]
		Together with $\|\widehat S\|_*\ge|\Tr\widehat S|$, the preceding upper
		bound on $\|\widehat S\|_*$ also gives $|\|\widehat S\|_*-d\|\le\frac{\delta d}{1-\delta}.$ 	Finally, using $2\|\widehat S_-\|_*=\|\widehat S\|_*-\Tr\widehat S$,
		\[
		2\|\widehat S_-\|_*
		=\|\widehat S\|_*-\Tr\widehat S
		\le
		\frac d{1-\delta}
		-\left(d-\frac{\delta d}{1-\delta}\right)
		=
		\frac{2\delta d}{1-\delta}.
		\]
	\end{proof}
	
	\begin{corollary}[Below-threshold nuclear minimizers are almost PSD]
		\label{cor:exact-nuclear-minimizers-almost-psd}
		Assume the hypotheses of Theorem~\ref{thm:phase-transition} and
		$0<\alpha<\alpha_\star(\kappa)$.  Then, with probability tending to one,
		the exact nuclear-norm program \eqref{eq:exact-nuclear-program} is feasible.
		If $\mathcal M_d$ denotes its minimizer set on this event, then
		\[
		\sup_{\widehat S\in\mathcal M_d}
		\left(
		|\Tr\widehat S-d|
		+\|\widehat S_-\|_*
		\right)
		=O_\Pp(\sqrt d).
		\]
		In particular,
		\[
		\sup_{\widehat S\in\mathcal M_d}
		\frac{\|\widehat S_-\|_*}{\Tr\widehat S}
		=O_\Pp(d^{-1/2}).
		\]
	\end{corollary}
	
	\begin{proof}
		The SAT assertion of Theorem~\ref{thm:phase-transition} supplies a positive
		semidefinite feasible point with probability tending to one.  Standard
		subgaussian sample-covariance concentration (e.g.
		\cite[Theorem~4.6.1]{vershynin2018highdimensional}), applied to the
		$n\times d$ matrix with rows $x_i^\top$, gives
		\[
		\delta=\|\widehat\Sigma-I_d\|_{\op}
		=O_\Pp(d^{-1/2}).
		\]
		Proposition~\ref{prop:exact-nuclear-minimizers-almost-psd} now gives all
		three conclusions uniformly over the minimizer set.  Since
		$\Tr\widehat S=d+O_\Pp(\sqrt d)$, it also gives the ratio bound.
	\end{proof}
	
	
	\begin{corollary}[Trace-normalized nuclear minimizers are PSD]
		\label{cor:trace-normalized-nuclear-minimizers-psd}
		Assume the hypotheses of Theorem~\ref{thm:phase-transition} and
		$0<\alpha<\alpha_\star(\kappa)$.  Then, with probability tending to one,
		every minimizer of
		\begin{equation}
			\label{eq:trace-normalized-nuclear-program}
			\min\left\{\|R\|_*:R\in\Ssym^d,\ \mathcal{A}_XR=0,\ \Tr R=d\right\}
		\end{equation}
		is positive semidefinite.
	\end{corollary}
	
	\begin{proof}
		The SAT theorem supplies $R_0\succ0$ with $\mathcal A_XR_0=0$.  After
		rescaling it to trace $d$, we obtain a PSD feasible point of nuclear norm
		$d$.  Every feasible $R$ has $\|R\|_*\ge\Tr R=d$, so every minimizer
		satisfies $\|R\|_*=\Tr R$.  Since
		$\|R\|_*-\Tr R=2\|R_-\|_*$, this forces $R\succeq0$.
	\end{proof}

	\subsection{Numerical experiments}
	\label{sec:numerical-experiment}

	For $1\le\kappa\le3$, we use a location mixture of two Gaussians. Let $S$ be Rademacher and $Z\sim N(0,1)$ be independent,
	and set
	\[
	\xi=\frac{S+\sigma Z}{\sqrt{1+\sigma^2}},
	\qquad
	\sigma^2=-1+\sqrt{\frac{2}{3-\kappa}}.
	\]
	This has fourth moment $\kappa$, with $\kappa=3$ interpreted as the Gaussian
	limit.  
    
    For $\kappa\ge3$, we instead use a scale mixture. Let $\xi=AZ$, where $A$ is independent
	of $Z\sim N(0,1)$ and, writing $M=\kappa/3$,
	\[
	A^2=
	\begin{cases}
		1/M, & \text{with probability }1-p,\\
		1+M, & \text{with probability }p,
	\end{cases}
	\qquad
	p=\frac{M-1}{M^2+M-1}.
	\]
	A direct calculation gives $\E A^2=1$, $\E A^4=\kappa/3$, and hence
	$\E\xi^4=\kappa$.  The two families agree at $\kappa=3$, where we just
    use a standard Gaussian.

	We used the grid $1.05,1.25,1.5,2,3,5,7.5,10,15,20,25,30$ and ran eight
	independent trials for each value.  In each trial, a nested sequence of samples
	was drawn, and a binary search in $n$ estimated the largest prefix for which
	ellipsoid fitting was numerically feasible.  At each tested $n$, feasibility
	was checked by
	\[
	\max_{R,t}\ t
	\quad\text{subject to}\quad
	R-tI_d\succeq0,
	\qquad
	x_i^\top R x_i=d,\quad i=1,\ldots,n.
	\]
	The optimum is the largest certified spectral margin.  We declared an
	instance feasible when it was nonnegative up to tolerance $10^{-5}$ and solved
	the SDP with CVXPY and CLARABEL.  Figure~\ref{fig:alpha-star-kappa-experiment}
	shows the largest feasible $n/d^2$ in each trial, their mean, and their middle
	50 percent interval together with the curve $\alpha_\star(\kappa)$.

    \bibliographystyle{alpha}
	\bibliography{main}
\end{document}